\documentclass[english]{article}

\usepackage[utf8]{inputenc}
\usepackage[T1]{fontenc}
\usepackage{amsthm} 
\usepackage{amssymb,mathrsfs} 
\usepackage{mathtools} % loads amsmath
\usepackage[a4paper]{geometry}
\usepackage{physics}
\usepackage[pdftex,dvipsnames]{xcolor}
\usepackage{enumerate}
\usepackage{graphicx}
\usepackage{enumitem}
\usepackage{array,booktabs}
\usepackage{bbold}
\usepackage{algorithm}
\usepackage{algpseudocode}
\usepackage{hyperref}
\usepackage{pgfplots}
\pgfplotsset{compat=1.16}

\hypersetup{
  unicode=false,          
  pdftoolbar=true,        
  pdfmenubar=true,        
  pdffitwindow=false,     
  pdfstartview={FitH},    
  pdfnewwindow=true,      
  colorlinks=true,       
  linkcolor=blue,          
  citecolor=green,        
  filecolor=cyan,         
  urlcolor=magenta        
  }
  
\newtheorem{theorem}{Theorem}
\newtheorem{lemma}[theorem]{Lemma}
\newtheorem{assumption}{Assumption}
\newtheorem{definition}{Definition}
\newtheorem{proposition}[theorem]{Proposition}
\newtheorem{remark}{Remark}
\newtheorem{corollary}[theorem]{Corollary}

\DeclareMathAlphabet{\mathpzc}{OT1}{pzc}{m}{it}

\DeclareMathOperator{\dist}{dist}
\let\div\relax
\DeclareMathOperator{\div}{div}
\DeclareMathOperator{\Int}{Int}
\DeclareMathOperator{\dom}{dom}
\newcommand{\dps}{\displaystyle }

\newcommand{\bbR}{\mathbb{R}}
\newcommand{\bbT}{\mathbb{T}}
\newcommand{\bbE}{\mathbb{E}} 
\newcommand{\bbN}{\mathbb{N}}

\newcommand{\calO}{\mathcal{O}}
\newcommand{\calA}{\mathcal{A}}
\newcommand{\calB}{\mathcal{B}}
\newcommand{\calX}{\mathcal{X}}
\newcommand{\calC}{\mathcal{C}}
\newcommand{\calK}{\mathcal{K}}
\newcommand{\calF}{\mathcal{F}}
\newcommand{\calL}{\mathcal{L}}
\newcommand{\calE}{\mathcal{E}}
\newcommand{\calJ}{\mathcal{J}}
\newcommand{\calR}{\mathcal{R}}
\newcommand{\calU}{\mathcal{U}}
\newcommand{\calG}{\mathcal{G}}
\newcommand{\calN}{\mathcal{N}}
\newcommand{\calV}{\mathcal{V}}
\newcommand{\calD}{\mathcal{D}}

\newcommand{\rmO}{\mathrm{O}}
\newcommand{\rmG}{\mathrm{G}}
\newcommand{\rmI}{\mathrm{I}}

\newcommand{\rme}{\mathrm{e}}
\newcommand{\rmd}{\mathrm{d}}

\newcommand{\sfT}{\mathsf{T}}

\newcommand{\scrF}{\mathscr{F}}

\newcommand{\imp}{\mathrm{imp}}
\newcommand{\rev}{\mathrm{rev}}
\newcommand{\GSV}{\mathrm{GSV}}
\newcommand{\IMR}{\mathrm{IMR}}

\newcommand{\ES}{\mathrm{ES}}
\newcommand{\GHMALA}{\mathrm{GHMALA}}
\newcommand{\FP}{\mathrm{FP}}
\newcommand{\FD}{\mathrm{FD}}
\newcommand{\SV}{\mathrm{SV}}
\newcommand{\EulerA}{\mathrm{Euler\,A}}
\newcommand{\EulerB}{\mathrm{Euler\,B}}
\newcommand{\id}{\mathrm{id}}
\newcommand{\pr}{\mathrm{pr}}
\newcommand{\prior}{\mathrm{prior}}
\newcommand{\lik}{\mathrm{lik}}
\newcommand{\post}{\mathrm{post}}
\newcommand{\rmsim}{\mathrm{sim}}
\newcommand{\iter}{\mathrm{iter}}
\newcommand{\bins}{\mathrm{bins}}
\newcommand{\Newton}{\mathrm{Newton}}
\newcommand{\TV}{\mathrm{TV}}
\let\real\relax
\newcommand{\real}{\mathrm{real}}
\newcommand{\iso}{\mathrm{iso}}
\newcommand{\aniso}{\mathrm{aniso}}
\newcommand{\pol}{\mathrm{pol}}
\newcommand{\fwd}{\mathrm{fwd}}

\begin{document}

%---------------------------------------------------------
%---------------------------------------------------------
%---------------------------------------------------------
%---------------------------------------------------------
%---------------------------------------------------------
%---------------- TITLE ----------------------------------
%---------------------------------------------------------
%---------------------------------------------------------
%---------------------------------------------------------
%---------------------------------------------------------
%---------------------------------------------------------
  
\title{Unbiasing Hamiltonian Monte Carlo algorithms for a general Hamiltonian function}
\author{T. Lelièvre,$^{1,2}$ R. Santet,$^{1,2}$  and G. Stoltz$^{1,2}$ \\
\small 1: CERMICS, Ecole des Ponts, Marne-la-Vallée, France \\
\small 2: MATHERIALS project-team, Inria, Paris, France
}

%---------------------------------------------------------
%---------------------------------------------------------
%---------------------------------------------------------
%---------------------------------------------------------
%---------------------------------------------------------
%---------------- ABSTRACT -------------------------------
%---------------------------------------------------------
%---------------------------------------------------------
%---------------------------------------------------------
%---------------------------------------------------------
%---------------------------------------------------------
\maketitle
%\tableofcontents
%\listoftodos[Notes]
\begin{abstract}
  
  Hamiltonian Monte Carlo (HMC) is a Markov chain Monte Carlo method that allows to sample high dimensional probability measures. It relies on the integration of the Hamiltonian dynamics to propose a move which is then accepted or rejected thanks to a Metropolis procedure. Unbiased sampling is guaranteed by the preservation by the numerical integrators of two key properties of the Hamiltonian dynamics: volume-preservation and reversibility up to momentum reversal. For separable Hamiltonian functions, some standard explicit numerical schemes, such as the Störmer--Verlet integrator, satisfy these properties. However, for numerical or physical reasons, one may consider a Hamiltonian function which is nonseparable, in which case the standard numerical schemes which preserve the volume and satisfy reversibility up to momentum reversal are implicit. When implemented in practice, such implicit schemes may admit many solutions or none, especially when the timestep is too large. We show here how to enforce the numerical reversibility, and thus unbiasedness, of HMC schemes in this context by introducing a reversibility check. In addition, for some specific forms of the Hamiltonian function, we discuss the consistency of these HMC schemes with some Langevin dynamics, and show in particular that our algorithm yields an efficient discretization of the metropolized overdamped Langevin dynamics with position-dependent diffusion coefficients. Numerical results illustrate the relevance of the reversibility check on simple problems.
\end{abstract}

%---------------------------------------------------------
%---------------------------------------------------------
%---------------------------------------------------------
%---------------------------------------------------------
%---------------------------------------------------------
%---------------- INTRODUCTION ---------------------------
%---------------------------------------------------------
%---------------------------------------------------------
%---------------------------------------------------------
%---------------------------------------------------------
%---------------------------------------------------------
\section{Introduction}

Many applications require sampling probability measures. For instance, in statistical physics~\cite{balian_2007}, macroscopic static properties or transport coefficients are obtained by averaging microscopic properties with respect to probability measures defining the underlying thermodynamic ensemble~\cite{frenkel_2001,evans_2008,tuckerman_2010,leimkuhler_2015,lelievre_2016}. In Bayesian inference, the distribution of parameter values explaining given data points are determined~\cite{robert_2007}. More precisely, the probability measure of interest is the posterior distribution~$P_{\post}(\theta|x)$ of a parameter~$\theta$ given the observed data~$x$, obtained by Bayes' rule as the product (up to a multiplicative constant) of a prior distribution~$P_{\prior}(\theta)$ on~$\theta$ and the likelihood of the data~$P_{\lik}(x|\theta)$. Typically, these problems are set in a high-dimensional setting (the number of atoms in a simulation or the number of parameters explaining the data). The predominant tools to sample probability measures in high dimension are Markov chain Monte Carlo (MCMC) methods~\cite{robert_2004}, where a Markov chain~$(X^{n})_{n\geqslant 0}$ which has a target probability measure~$\mu$ as its invariant probability distribution is built. Ergodic theorems are then invoked to guarantee the convergence of trajectory averages of observables to their expectations with respect to~$\mu$~\cite{meyn_2009,douc_2018}.

In this paper, we are particularly interested in (Generalized) Hamiltonian Monte Carlo--(G)HMC--algorithms~\cite{duane_1987,horowitz_1991,neal_1993,lelievre_2010} to perform MCMC sampling. These algorithms aim at performing an unbiased sampling of the Boltzmann--Gibbs probability measure defined by
\begin{equation}
  \label{eq:measure_non_separable}
  \mu(\rmd q\,\rmd p)=Z_{\mu}^{-1}\rme^{-H(q,p)}\,\rmd q\,\rmd p,
  \qquad Z_{\mu}=\int_{\bbR^{m}\times\bbR^{m}}\rme^{-H},
\end{equation}
where~$H$ is the so-called Hamiltonian function, satisfying~$H(q,-p)=H(q,p)$, and the normalization constant~$Z_{\mu}$ is assumed to be finite. These algorithms are specific cases of Metropolis--Hastings algorithms~\cite{metropolis_1953,hastings_1970} with proposal moves computed by numerically integrating the Hamiltonian dynamics over some fixed time interval:
\begin{equation}
  \label{eq:hamiltonian_dynamics}
  \frac{\rmd}{\rmd t}\begin{pmatrix}
    q_t\\p_t
  \end{pmatrix}=J\nabla H(q_t,p_t),\quad J=\begin{pmatrix}
    0_m&\rmI_m\\-\rmI_m&0_m
  \end{pmatrix},
\end{equation}
where~$\rmI_m$ denotes the identity matrix in dimension~$m$. We make precise in Algorithm~\ref{alg:hmc_standard} the usual HMC scheme, see for instance~\cite[Algorithm~2.3]{lelievre_2010}. The core construction of this paper is to properly define the discrete flow~$\varphi_{\Delta t}$ in step~\ref{step:hmc_standard_2} when dealing with nonseparable Hamiltonian functions, see the ``main result'' paragraph below and Section~\ref{sec:unbiased_HMC_schemes_implicit_integrators}.
\begin{algorithm}
  \caption[]{Standard HMC scheme.}
  \label{alg:hmc_standard}
  Consider an initial condition~$(q^{0},p^{0})\in\bbR^{m}\times\bbR^{m}$, and set~$n= 0$.
  \begin{enumerate}[label={[\thealgorithm.\roman*]}]
    \item \label{step:hmc_standard_1} For a fixed position~$q^{n}$, sample~$\tilde{p}^{n}$ according to~$Z_{\mu_n}^{-1}\rme^{-H(q^{n},\cdot)}$ with~$Z_{\mu_n}=\int_{\bbR^{m}}\rme^{-H(q^{n},p)}\rmd p$, and compute the Hamiltonian~$H(q^{n},\tilde{p}^{n})$ of the configuration~$(q^{n},\tilde{p}^{n})$;
    \item \label{step:hmc_standard_2} Integrate the Hamiltonian dynamics~\eqref{eq:hamiltonian_dynamics} starting from~$(q^{n},\tilde{p}^{n})$ using a numerical flow~$\varphi_{\Delta t}$ to obtain a new configuration (the proposal) denoted by~$(\tilde{q}^{n+1},\tilde{p}^{n+1}):=\varphi_{\Delta t}(q^{n},\tilde{p}^{n})$, and compute its Hamiltonian~$H(\tilde{q}^{n+1},\tilde{p}^{n+1})$;
    \item Accept the proposal with probability
    \begin{equation*}
      \min\left(
        1,
        \exp\left(
          -H(\tilde{q}^{n+1},\tilde{p}^{n+1})+H(q^{n},\tilde{p}^{n})
        \right)
      \right),
    \end{equation*}
    and set~$(q^{n+1},p^{n+1})=(\tilde{q}^{n+1},\tilde{p}^{n+1})$ in this case; otherwise set~$(q^{n+1},p^{n+1})=(q^{n},\tilde{p}^{n})$.
    \item Increment~$n$ and go back to~\ref{step:hmc_standard_1}.
  \end{enumerate}
\end{algorithm}

If the Hamiltonian function is separable, \emph{i.e.}~of the form
\begin{equation}
  \label{eq:hamiltonian_function_separable}
  H(q,p)=V(q)+K(p),
\end{equation}
and~$K(p)=K(-p)$, then Algorithm~\ref{alg:hmc_standard} is easy to implement using the usual Störmer--Verlet integrator~$\varphi_{\Delta t}^{\SV}$ for step~\ref{step:hmc_standard_2} (see~\cite{leimkuhler_2005,hairer_2006}):
\begin{equation}
  \label{eq:stormer_verlet_integrator}
  \varphi_{\Delta t}^{\SV}(q^{n},\tilde{p}^{n}) =: (\tilde{q}^{n+1},\tilde{p}^{n+1}),\qquad
  \left\lbrace
  \begin{aligned}
    \tilde{p}^{n+1/2}&=\tilde{p}^{n}-\frac{\Delta t}{2}\nabla V(q^{n}),\\
    \tilde{q}^{n+1}&=q^{n}+\Delta t\,\nabla K(\tilde{p}^{n+1/2}),\\
    \tilde{p}^{n+1}&=\tilde{p}^{n+1/2}-\frac{\Delta t}{2}\nabla V(\tilde{q}^{n+1}).
  \end{aligned}
  \right.
\end{equation}
Notice that in this case the Störmer--Verlet integrator is explicit.

Algorithm~\ref{alg:hmc_standard} using~$\varphi_{\Delta t}^{\SV}$ is shown to be unbiased (see for instance~\cite{neal_1993,schuette_1999}) using two fundamental properties of the numerical flow~$\varphi_{\Delta t}^{\SV}$:
\begin{itemize}
  \item~$\varphi_{\Delta t}^{\SV}$ preserves the Lebesgue measure, \emph{i.e.} denoting by~$\calB(\bbR^{m}\times\bbR^{m})$ the Borel~$\sigma$-algebra and~$\lambda$ the Lebesgue measure in~$\bbR^{m}\times\bbR^{m}$:
  \begin{equation}
    \label{eq:lebesgue_measure_preservation}
    \forall \calE\in\calB(\bbR^{m}\times\bbR^{m}),\qquad \lambda\left(\left(\varphi_{\Delta t}^{\SV}\right)^{-1}(\calE)\right)=\lambda(\calE);
  \end{equation}
  \item~$\varphi_{\Delta t}^{\SV}$ satisfies~$S$-reversibility, \emph{i.e.}
  \begin{equation}
    \label{eq:S_reversibility_Stormer_Verlet_explicit}
    S\circ\varphi_{\Delta t}^{\SV}\circ S\circ\varphi_{\Delta t}^{\SV}=\id,
  \end{equation}
  where~$\id$ is the identity map and~$S$ is the momentum reversal map acting as
  \begin{equation}
    \label{eq:S_momentum_reversal}
    S(q,p)=(q,-p).
  \end{equation}
\end{itemize}
The Hamiltonian function is required to be an even function of the momenta so that~\eqref{eq:S_reversibility_Stormer_Verlet_explicit} can be satisfied by the numerical integrator. Property~\eqref{eq:lebesgue_measure_preservation} follows from the symplecticity of~$\varphi_{\Delta t}^{\SV}$:
\begin{equation}
  \label{eq:symplecticity}
    \forall(q,p)\in\bbR^{m}\times\bbR^{m},\qquad\nabla \varphi_{\Delta t}^{\SV}(q,p)^{\sfT}J\nabla\varphi_{\Delta t}^{\SV}(q,p)=J,
\end{equation}
where~$J$ is defined in~\eqref{eq:hamiltonian_dynamics}. Property~\eqref{eq:S_reversibility_Stormer_Verlet_explicit} is a consequence of the following two properties~\cite{hairer_2006}:
\begin{equation*}
  \left\lbrace
    \begin{aligned}
      \left(\varphi_{\Delta t}^{\SV}\right)^{-1}&=\varphi_{-\Delta t}^{\SV},\\
      S\circ\varphi_{\Delta t}^{\SV}\circ S&=\varphi_{-\Delta t}^{\SV}.
    \end{aligned}
  \right.
\end{equation*}

For numerical or physical reasons, one may consider a Hamiltonian function which is not separable~\cite{izaguirre_2004,girolami_2011}. For instance, the Riemannian Manifold Hamiltonian Monte Carlo (RMHMC) algorithm proposed in~\cite{girolami_2011} uses a Hamiltonian function of the form
\begin{equation}
  \label{eq:H_RMHMC}
  H(q,p)=V(q)-\frac{1}{2}\ln\det D(q)+\frac{1}{2}p^{\sfT}D(q)p,
\end{equation}
where~$V$ is a potential energy function and~$D$ is a position-dependent diffusion coefficient. To ensure irreducibility, for all~$q\in\bbR^{m}$,~$D(q)$ is a positive definite symmetric matrix. The aim of the RMHMC algorithm is to sample the marginal in position of the probability measure~$\mu$ defined by~\eqref{eq:measure_non_separable}, with~$H$ the Hamiltonian function~\eqref{eq:H_RMHMC}, namely:
\begin{equation}
  \label{eq:invariant_measure_position_RMHMC}
  \pi(\rmd q)=Z_{\pi}^{-1}\rme^{-V(q)}\,\rmd q,\qquad Z_{\pi}=\int_{\bbR^{m}}\rme^{-V}.
\end{equation}
The RMHMC algorithm uses the Riemannian metric tensor~$D^{-1}$ as a numerical parameter to enhance convergence. The computational interest of RMHMC over standard HMC algorithms, in particular for anisotropic or multiscale target distributions, has been discussed in many papers, see for example~\cite{brofos_2021} for a recent contribution. Another situation where nonseparable Hamiltonian functions are involved is the Shadow Hybrid Monte Carlo method~\cite{izaguirre_2004}: starting from a separable Hamiltonian function of the form~\eqref{eq:hamiltonian_function_separable}, additional terms (which are nonseparable in general) define a modified Hamiltonian function which is better preserved over one time step by symplectic integrators, and this leads to faster sampling because moves are less often rejected in the Metropolis--Hastings procedure.

In addition, nonseparable Hamiltonian functions also arise from physical reasons, for instance through a change of coordinate to so-called internal coordinates, see~\cite[Section~XIII.10]{hairer_2006},~\cite{fang_2014} and Section~\ref{subsec:RMHMC_overdamped_Langevin} below. As will be shown in Section~\ref{subsec:RMHMC_overdamped_Langevin}, RMHMC efficiently discretizes overdamped Langevin dynamics (a.k.a.~Brownian dynamics) with position-dependent diffusion coefficients, as it provides a proposal which drastically reduces the rejection probability in the Metropolis--Hastings step, in contrast to simple proposals based on Euler--Maruyama discretizations. These dynamics naturally appear in biological models for example~\cite{saffman_1975,comer_2013,phillips_2023}, and can be derived as effective dynamics along some reaction coordinates in molecular dynamics, see~\cite{legoll_2010}.

When confronted with a nonseparable Hamiltonian function, it is not clear how to numerically integrate the Hamiltonian dynamics~\eqref{eq:hamiltonian_dynamics} while preserving the two fundamental properties~\eqref{eq:lebesgue_measure_preservation} and~\eqref{eq:S_reversibility_Stormer_Verlet_explicit} that ensure the unbiased sampling of the target probability distribution~\eqref{eq:measure_non_separable} by the HMC algorithm. This was already noted in~\cite{brofos_2021}. More precisely, one natural choice for the numerical integrator is the Generalized Störmer--Verlet (GSV) scheme (which reduces to~\eqref{eq:stormer_verlet_integrator} for separable Hamiltonian functions) defined by
\begin{equation}
  \label{eq:GSV_intro}
  \varphi_{\Delta t}^{\GSV}(q^{n},\tilde{p}^{n}) =: (\tilde{q}^{n+1},\tilde{p}^{n+1}),\quad
  \left\lbrace
  \begin{aligned}
    \tilde{p}^{n+1/2}&=\tilde{p}^{n}-\frac{\Delta t}{2}\nabla_{q} H(q^{n},\tilde{p}^{n+1/2}),\\
    \tilde{q}^{n+1}&=q^{n}+\frac{\Delta t}{2}\left(\nabla_{p} H(q^{n},\tilde{p}^{n+1/2})+\nabla_{p}H(\tilde{q}^{n+1},\tilde{p}^{n+1/2})\right),\\
    \tilde{p}^{n+1}&=\tilde{p}^{n+1/2}-\frac{\Delta t}{2}\nabla_{q} H(\tilde{q}^{n+1},\tilde{p}^{n+1/2}).
  \end{aligned}
  \right.
\end{equation}
Indeed, the integrator~$\varphi_{\Delta t}^{\GSV}$ satisfies, in the limit~$\Delta t\to0$, the two properties~\eqref{eq:lebesgue_measure_preservation} and~\eqref{eq:S_reversibility_Stormer_Verlet_explicit} (see for instance~\cite{leimkuhler_2005}). However, the GSV integrator is generally implicit for nonseparable Hamiltonian functions: there is no analytic and unique way to define~$(\tilde{q}^{n+1},\tilde{p}^{n+1})$ solution to~\eqref{eq:GSV_intro} for a given initial configuration~$(q^{n},\tilde{p}^{n})$. As a consequence, the relation~\eqref{eq:S_reversibility_Stormer_Verlet_explicit} may not hold true for a fixed time step~$\Delta t>0$ for two reasons:
\begin{enumerate}[label=(\roman*)]
  \item one of the two numerical methods used to solve the two implicit problems in~\eqref{eq:GSV_intro} may not yield a solution;
  \item even if solutions to the two implicit problems are found, the fact that there may be multiple solutions to the implicit problems may lead to situations where the $S$-reversibility~\eqref{eq:S_reversibility_Stormer_Verlet_explicit} is not satisfied.
\end{enumerate}
The main objective of this paper is to propose a modification of the numerical flow in order to satisfy the two properties~\eqref{eq:lebesgue_measure_preservation} and~\eqref{eq:S_reversibility_Stormer_Verlet_explicit} for any choice of the time step. The idea is to introduce a~$S$-reversibility check in the algorithm, in the spirit of what has recently been proposed to unbias HMC algorithms to sample measures on submanifolds describing mechanically constrained systems~\cite{zappa_2018,lelievre_2019}. In this setting, the Hamiltonian dynamics is constrained to live on the cotangent space associated with the submanifold, which requires to solve an implicit problem to project configurations back to the cotangent space. Popular methods to integrate these Hamiltonian dynamics are the RATTLE and SHAKE methods~\cite{ryckaert_1977, andersen_1983}. The HMC algorithm using the RATTLE or the SHAKE method without a~$S$-reversibility check leads to biased results, see~\cite{zappa_2018,lelievre_2019,lelievre_2022}; similarly, the HMC algorithm using the GSV dynamics without a~$S$-reversibility check exhibits a bias, as numerically demonstrated in Section~\ref{sec:numerical_results}.

%---------------------------------------------------------
%---------------------------------------------------------
%---------------- MAIN RESULTS ---------------------------
%---------------------------------------------------------
%---------------------------------------------------------

\paragraph{Main results.} The main results of this work are the following.
\begin{itemize}
  \item First, we construct a map~$\varphi_{\Delta t}^{\GSV,\rev}$ which corresponds to one step of the GSV dynamics with~$S$-reversibility checks (see~\eqref{eq:B_k}--\eqref{eq:psi_rev} in Section~\ref{subsec:enforcing_S_reversibility} below for a precise definition). We then show that the HMC algorithm described in Algorithm~\ref{alg:hmc_standard} with~$\varphi_{\Delta t}=\varphi_{\Delta t}^{\GSV,\rev}$ in step~\ref{step:hmc_standard_2} is unbiased. This leads to Algorithm~\ref{alg:hmc_scheme_rev_check}. We also show that the GHMC algorithm can be unbiased in a similar way, see Algorithm~\ref{alg:ghmc_scheme_rev_check}.

  In practice, the output of~$\varphi_{\Delta t}^{\GSV,\rev}$ provides a proposal move when three successive conditions are met:
  \begin{itemize}
    \item the GSV scheme~\eqref{eq:GSV_intro} can be solved from an initial condition~$(q^{0},p^{0})$, in which case the obtained configuration is denoted by~$(q^{1},p^{1})$;
    \item the GSV scheme can be solved from~$S(q^{1},p^{1})=(q^{1},-p^{1})$, in which case the obtained configuration is denoted by~$(q^{2},p^{2})$;
    \item $S$-reversibility is observed. For GSV, this reduces to the condition~$q^{2}=q^{0}$ (see Lemma~\ref{lem:solver_rev_from_flow_rev_gsv} below).
  \end{itemize}
  If one of the conditions is not met, the proposal move is rejected. We make precise practical implementations of the algorithm using Newton's method in Sections~\ref{subsec:practical_numerial_solver_newton_method}, see Algorithms~\ref{alg:newton_practice} and~\ref{alg:psi_dt_rev}.
  \item Second, we discuss the consistency of the obtained algorithms with Langevin and overdamped Langevin dynamics. In particular, we prove in Section~\ref{subsec:RMHMC_overdamped_Langevin} that RMHMC yields a very efficient consistent discretization of the overdamped Langevin dynamics with position-dependent diffusion coefficient:
  \begin{equation*}
    \rmd q_t = \left(
        -D(q_t)\nabla V(q_t)+\div D(q_t)
        \right)\rmd t + \sqrt{2D(q_t)}~\rmd W_t.
  \end{equation*}
  Indeed, the rejection probability due to the Metropolis--Hastings procedure scales as~$\rmO(h^{3/2})$ (where $h$ is the time step discretization of the overdamped Langevin dynamics) instead of the usual~$\rmO(h^{1/2})$ obtained when using a Metropolized Euler--Maruyama scheme (\emph{i.e.}~the MALA algorithm)~\cite{fathi_2017}. As already mentionned above, we believe this is interesting since such dynamics arise in many contexts, for numerical and physical reasons. Indeed, position-dependent diffusion coefficients are used to precondition the overdamped Langevin dynamics, which leads to better mixing properties~\cite{leimkuhler_2018,garbuno_2020,liu_2022}. Overdamped Langevin dynamics with position-dependent diffusion coefficients also appear naturally to model various physical and biological processes~\cite{saffman_1975,jardat_1999} and efficient numerical methods to integrate these dynamics have been recently proposed in~\cite{phillips_2023} for example.
\end{itemize}

In the subsequent sections, we present the results in a more general setting than the one used in this introduction. Our approach offers a framework to sample without bias the invariant probability measure of a stochastic process, provided that the numerical scheme satisfies certain reversibility properties. In particular, we also discuss the Implicit Midpoint Rule (IMR) integrator of the Hamiltonian dynamics~\eqref{eq:hamiltonian_dynamics} which is also commonly used in the literature~\cite{brofos_2021,brofos_2021_b,kook_2022,kook_2023}, and adapt the argument to other types of phase-space dynamics, such as for the Generalized Hybrid Metropolis--Adjusted Langevin Algorithm (GHMALA) introduced in~\cite{poncet_2017}, see Section~\ref{subsec:ghmala} below.

Let us emphasize that this work focuses on unbiasedness, namely establishing that the target probability measure~$\mu$ is invariant for the (G)HMC algorithms we introduce in Section~\ref{sec:unbiased_HMC_schemes_implicit_integrators} with general Hamiltonian functions. To prove ergodicity, one needs to show that the Markov chains generated by these algorithms are in addition irreducible. This has been studied for the (G)HMC algorithms with separable Hamiltonian functions in~\cite{cances_2007,livingstone_2019,durmus_2020}, but irreducibility for nonseparable Hamiltonian functions will not be further discussed in this work.

Before presenting the outline of our work, let us mention a related independent recent preprint by Noble \textit{et al.}~\cite{noble_2023} where the authors also correct the bias of RMHMC using a reversibility check, focusing on the Generalized Störmer Verlet numerical scheme, and using fixed-point iterations schemes to define the numerical flows. In our work, we consider more general nonseparable Hamiltonian functions, and more general numerical schemes (IMR, Newton's method to define the numerical flow, etc.). In particular, our analysis can be used to study other dynamics than RMHMC, as illustrated in Section~\ref{subsec:ghmala}.

%---------------------------------------------------------
%---------------------------------------------------------
%---------------- OUTLINE --------------------------------
%---------------------------------------------------------
%---------------------------------------------------------

\paragraph{Outline of the work.} This paper is organized as follows. We make precise in Section~\ref{sec:numerical_flow_s_reversibility} the definitions of numerical schemes, solvers and flows, and the notions of~$S$-reversibility when~$S$ is a general involution. We then show how to enforce~$S$-reversibility and state our main results (whose proofs are postponed to Section~\ref{sec:proofs}). In Section~\ref{sec:unbiased_HMC_schemes_implicit_integrators}, (G)HMC algorithms that sample the target probability measure without bias are introduced. This is where we also discuss the consistency of (G)HMC algorithms with Langevin and overdamped Langevin dynamics. We then construct in Section~\ref{sec:construction_practical_numerical_flows} ideal numerical flows for the usual HMC integrators, and use Newton's method to provide practical implementations. The analysis of Sections~\ref{sec:numerical_flow_s_reversibility},~\ref{sec:unbiased_HMC_schemes_implicit_integrators} and~\ref{sec:construction_practical_numerical_flows} are illustrated on two running examples: the Generalized Störmer--Verlet (GSV) and the Implicit Midpoint Rule (IMR) numerical schemes. We finally illustrate in Section~\ref{sec:numerical_results} on two numerical examples the interest of the~$S$-reversibility check as well as using a position-dependent diffusion coefficient using the Hamiltonian function~\eqref{eq:H_RMHMC}.

%---------------------------------------------------------
%---------------------------------------------------------
%---------------------------------------------------------
%---------------------------------------------------------
%---------------------------------------------------------
%---------------- NUMERICAL FLOWS ------------------------
%---------------------------------------------------------
%---------------------------------------------------------
%---------------------------------------------------------
%---------------------------------------------------------
%---------------------------------------------------------
\section{Numerical flows and \texorpdfstring{$S$}{S}-reversibility}
\label{sec:numerical_flow_s_reversibility}
We describe in this section the framework for numerical schemes and numerical flows which allows to properly define~$S$-reversibility. The definitions of numerical solvers and numerical flows are given in Section~\ref{subsec:numerical_scheme_solver_flow}. We then discuss in Section~\ref{subsec:regularity_numerical_solver_local_preservation_lebesgue_measure} the regularity and the Lebesgue measure preservation of a numerical flow. The notion of~$S$-reversibility is introduced in Section~\ref{subsec:S_reversibility}. We finally show how to enforce this property using a numerical flow with~$S$-reversibility check in Section~\ref{subsec:enforcing_S_reversibility}.

Consider a time step~$\Delta t>0$ and denote by~$\calX$ the configuration space. We assume that~$\calX=\calO\times\bbR^{m}$ with~$\calO$ an open subset of~$\bbR^{m}$ (it is easy to generalize all the results to domains with periodic boundary conditions in position, namely~$\calX=\bbT^{m}\times\bbR^{m}$ with~$\bbT$ the one-dimensional torus). Notice that the space~$\calX$ is an open subset of~$\bbR^{d}$ with~$d=2m$. It is equipped with the Euclidean norm~$\left\lVert\,\cdot\,\right\rVert$ in~$\bbR^{d}$. The matrix norm is the usual operator norm associated with~$\left\lVert\,\cdot\,\right\rVert$. If~$k\geqslant 1$ and~$\Phi\colon\calX\times\calX^k\to\calX^k$ is a~$\calC^{1}$ map, then~$\nabla_{x}\Phi$ (respectively~$\nabla_y \Phi$) corresponds to the gradient of~$\Phi$ with respect to the first coordinate of the argument of~$\Phi$ (respectively the last~$k$ coordinates).

%---------------------------------------------------------
%---------------------------------------------------------
%---------------- NUMERICAL SCHEME -----------------------
%---------------------------------------------------------
%---------------------------------------------------------

\subsection{Numerical scheme, numerical solver and numerical flow}
\label{subsec:numerical_scheme_solver_flow}
We present in this section the adaptation of the strategy developed to analyze schemes for constrained systems~\cite{zappa_2018,lelievre_2019} to the case of implicit integrators. For a given initial condition~$x^{n}$ at time~$t_n=n\Delta t$, let us consider an implicit \textit{numerical scheme}~$\Phi_{\Delta t}$ defining the approximation~$x^{n+1}$ of the solution to an ordinary differential equation (ODE) at time~$t_{n+1}$ as a solution to an equation of the form
\begin{equation*}
  \Phi_{\Delta t}\left(x^{n},x^{n+1}_{1},\dots,x^{n+1}_{k-1},x^{n+1}_{k}\right)=0,  
\end{equation*}
where~$\Phi_{\Delta t}\colon\calX\times\calX^{k}\to\calX^{k}$ is a map depending on the time step~$\Delta t$ and the underlying ODE. The configurations~$\left(x^{n+1}_{i}\right)_{1\leqslant i\leqslant k-1}$ are intermediate configurations depending on the structure of the numerical scheme, while~$x^{n+1}_{k}=x^{n+1}$ is the new configuration. 

Let us introduce the two running examples of numerical schemes~$\Phi_{\Delta t}$ that will be used in the following.
\begin{definition}[Implicit Midpoint Rule]
  \label{def:IMR}
  Fix~$\Delta t>0$. Starting from~$x^{n}=(q^{n},p^{n})\in\calX$, the Implicit Midpoint Rule (IMR) numerical scheme discretizing the Hamiltonian dynamics~\eqref{eq:hamiltonian_dynamics} reads:
  \begin{equation}
    \label{eq:IMR}
    \left\lbrace
    \begin{aligned}
      q^{n+1}&= q^{n}+\Delta t~\nabla_{p}H\left(\dfrac{q^{n}+q^{n+1}}{2},\dfrac{p^{n}+p^{n+1}}{2}\right),\\
      p^{n+1}&= p ^{n}-\Delta t~\nabla_{q}H\left(\dfrac{q^{n}+q^{n+1}}{2},\dfrac{p^{n}+p^{n+1}}{2}\right),
    \end{aligned}
    \right.
  \end{equation}
  with~$H$ the Hamiltonian function of the system. In that case, one has~$k=1$ and 
  \begin{equation}
    \label{eq:Phi_IMR}
    \Phi_{\Delta t}^{\IMR}(q,p,q_1,p_1)=
    \begin{pmatrix}
      q_1-q-\Delta t~\nabla_{p}H\left(\dfrac{q+q_1}{2},\dfrac{p+p_1}{2}\right)\\[0.3cm]
      p_1-p+\Delta t~\nabla_{q}H\left(\dfrac{q+q_1}{2},\dfrac{p+p_1}{2}\right)
    \end{pmatrix}.
  \end{equation}
  Except for very specific Hamiltonian functions,~$x^{n+1}=(q^{n+1},p^{n+1})$ cannot be uniquely obtained in an analytical form as a function of~$x^{n}$: the scheme is implicit.
\end{definition}

\begin{definition}[Generalized Störmer--Verlet]
  \label{def:GSV}
  Fix~$\Delta t>0$. Starting from~$x^{n}=(q^{n},p^{n})\in\calX$, the Generalized Störmer--Verlet (GSV) scheme discretizing the Hamiltonian dynamics~\eqref{eq:hamiltonian_dynamics} reads
  \begin{equation}
    \label{eq:GSV}
    \left\lbrace
      \begin{aligned}
        \widetilde{p}^{n+1/2} &= p^{n}-\frac{\Delta t}{2}\nabla_q H(q^{n},\widetilde{p}^{n+1/2}),\\
        q^{n+1}&= q^{n}+\frac{\Delta t}{2}\left(
          \nabla_pH(q^{n},\widetilde{p}^{n+1/2})+\nabla_pH(q^{n+1},\widetilde{p}^{n+1/2})
        \right),\\
        p^{n+1}&= \widetilde{p}^{n+1/2}-\frac{\Delta t}{2}\nabla_qH(q^{n+1},\widetilde{p}^{n+1/2}),
      \end{aligned}
    \right.
  \end{equation}
  where~$H$ is the Hamiltonian function of the system. This scheme can be seen as the composition of the symplectic Euler B numerical scheme with time step~$\Delta t/2$ and the symplectic Euler A numerical scheme with time step~$\Delta t/2$, which respectively read: 
  \begin{equation*}
    \left\lbrace
    \begin{aligned}
      p^{n+1} &= p^{n}-\frac{\Delta t}{2}~\nabla_q H(q^{n},p^{n+1}),\\
      q^{n+1} &= q^{n}+\frac{\Delta t}{2}~\nabla_p H(q^{n},p^{n+1}),
    \end{aligned}
    \right.\qquad 
    \left\lbrace
    \begin{aligned}
      q^{n+1} &= q^{n}+\frac{\Delta t}{2}~\nabla_p H(q^{n+1},p^{n}),\\
      p^{n+1} &= p^{n}-\frac{\Delta t}{2}~\nabla_q H(q^{n+1},p^{n}).
    \end{aligned}
    \right.
  \end{equation*}
  For both numerical schemes,~$k=1$ and
  \begin{equation*}
    %\label{eq:Phi_EulerB}
    \Phi_{\Delta t/2}^{\EulerB}(q,p,q_1,p_1)=\begin{pmatrix}
      q_1-q-\dfrac{\Delta t}{2}~\nabla_p H(q,p_1)\\[0.2cm]
      p_1-p+\dfrac{\Delta t}{2}~\nabla_q H(q,p_1)
    \end{pmatrix},
  \end{equation*}
  \begin{equation*}
    %\label{eq:Phi_EulerA}
    \Phi^{\EulerA}_{\Delta t/2}(q,p,q_1,p_1)=\begin{pmatrix}
      q_1-q-\dfrac{\Delta t}{2}~\nabla_p H(q_1,p)\\[0.2cm]
      p_1-p+\dfrac{\Delta t}{2}~\nabla_q H(q_1,p)
    \end{pmatrix}.
  \end{equation*}
  As a consequence, the numerical scheme for GSV is defined with~$k=2$ as
  \begin{align}
    \Phi_{\Delta t}^{\GSV}\left(q,p,q_1,p_1,q_2,p_2\right)
    &=\left(\Phi_{\Delta t/2}^{\EulerB}(q,p,q_1,p_1),\Phi_{\Delta t/2}^{\EulerA}(q_1,p_1,q_2,p_2)\right),\nonumber\\
    &=\label{eq:Phi_GSV}
    \begin{pmatrix}
      q_1-q-\dfrac{\Delta t}{2}~\nabla_p H(q,p_1)\\[0.2cm]
      p_1-p+\dfrac{\Delta t}{2}~\nabla_q H(q,p_1)\\[0.2cm]
      q_2-q_1-\dfrac{\Delta t}{2}~\nabla_p H(q_2,p_1)\\[0.2cm]
      p_2-p_1+\dfrac{\Delta t}{2}~\nabla_q H(q_2,p_1)
    \end{pmatrix}.
  \end{align}
  For a general nonseparable Hamiltonian function,~$x^{n+1}=(q^{n+1},p^{n+1})$ cannot be uniquely obtained in an analytical form as a function of~$x^{n}$: the scheme is implicit.
\end{definition}

The numerical schemes we are interested in are implicit: this implies that for a given configuration~$x^{n}$, it may be the case that there is no solution~$x^{n+1}$ to the implicit problem. For instance the algorithm implemented in practice to solve the implicit problem does not converge to any solution. Besides, there may be more than one solution to the implicit problem, so that, for example, the algorithm may converge to different solutions depending on numerical parameters. In order to precisely define the numerical procedure, let us then introduce the notion of \emph{numerical solver}, which includes the notion of \emph{numerical flow}.

\begin{definition}[Numerical solver and numerical flow associated with a numerical scheme~$\Phi_{\Delta t}$]
  \label{def:numerical_solver}
  Let~$k\geqslant1$,~$\Delta t>0$ and~$\Phi_{\Delta t}\colon\calX\times\calX^{k}\to\calX^{k}$ be a~$\calC^{1}$ numerical scheme. A \emph{numerical solver} is a continuous map~$\chi_{\Delta t}\colon\calA_{\Delta t}\to\calX^{k}$ defined on a nonempty open set~$\calA_{\Delta t}\subset\calX$ such that
  \begin{equation}
    \label{eq:phi_chi_relation}
    \forall x\in\calA_{\Delta t},\qquad 
    \left\lbrace
    \begin{aligned}
      &\Phi_{\Delta t}\left(x,\chi_{\Delta t}(x)\right)=0,\\
      &\nabla_{y}\Phi_{\Delta t}\left(x,\chi_{\Delta t}(x)\right)\text{ is an invertible }dk\times dk\text{-matrix}.
    \end{aligned}
    \right.
  \end{equation}
  The associated \emph{numerical flow}~$\varphi_{\Delta t}$ is then defined as the~$k$-th coordinate of~$\chi_{\Delta t}$:
  \begin{equation*}
    \forall x\in\calA_{\Delta t},\qquad \varphi_{\Delta t}(x)=\chi_{\Delta t,k}(x),\text{ where }\chi_{\Delta t}(x)=\left(\chi_{\Delta t,1}(x),\dots,\chi_{\Delta t,k}(x)\right).
  \end{equation*}
  In particular, if~$k=1$, the notions of numerical solver and numerical flow coincide.
\end{definition}
The set~$\calA_{\Delta t}$ denotes the possible starting configurations for which the numerical solver~$\chi_{\Delta t}$ gives a solution to the numerical scheme. In particular, the numerical flow updates the configuration~$x\in\calA_{\Delta t}$ to~$\chi_{\Delta t,k}(x)=\varphi_{\Delta t}(x)$ via the intermediate steps~$(\chi_{\Delta t,1}(x),\dots,\chi_{\Delta t,k-1}(x))$. In practical terms, a configuration~$x$ belongs to~$\calA_{\Delta t}$ if the algorithm implemented converges to a unique solution.

\begin{remark}
  In Definition~\ref{def:numerical_solver}, the univalued maps~$\chi_{\Delta t}$ and~$\varphi_{\Delta t}$ encode the fact that the implemented algorithm is able to find a unique solution of the implicit problem. It should be possible to extend the results of this work to algorithms which are able to find multiple solutions to the implicit problem, in the spirit of~\cite{lelievre_2022}.
\end{remark}

To conclude this section, we illustrate on a simple example that implicit problems lead in general to multiple possible solutions.
\begin{remark}
  \label{rem:problem_S_rev_introduction}
  Consider the GSV numerical scheme in the case of a RMHMC algorithm in dimension~$m=1$, \emph{i.e.}~the scheme defined in Definition~\ref{def:GSV} with the Hamiltonian function~\eqref{eq:H_RMHMC}. Consider the diffusion coefficient~$D(q)=1+q^{2}$. Starting from~$(q^{n},p^{n})\in\calX$, the GSV numerical scheme then reads
  \begin{equation*}
  \left\lbrace
    \begin{aligned}
      \widetilde{p}^{n+1/2} &= p^{n} - \frac{\Delta t}{2}\left(V'(q^{n})-\dfrac{q^{n}}{1+\left(q^{n}\right)^{2}}+q^{n}\left(\widetilde{p}^{n+1/2}\right)^{2}\right),\\
      q^{n+1} &= q^{n}+\frac{\Delta t}{2}\left(2+\left(q^{n}\right)^{2}+\left(q^{n+1}\right)^{2}\right)\widetilde{p}^{n+1/2},\\
      p^{n+1} &= \widetilde{p}^{n+1/2} - \frac{\Delta t}{2}\left(V'(q^{n+1})-\dfrac{q^{n+1}}{1+\left(q^{n+1}\right)^{2}}+q^{n+1}\left(\widetilde{p}^{n+1/2}\right)^{2}\right).
    \end{aligned}
  \right.
\end{equation*}
Note that only the first two steps are implicit. When~$q^{n}\neq 0$, the first equation is a quadratic equation in~$\widetilde{p}^{n+1/2}$ that can be solved analytically. This means that two solutions exist when the discriminant of the quadratic equation is strictly positive (which is the case when the time step~$\Delta t$ is small enough). The second equation is also a quadratic equation in~$q^{n+1}$ if~$\widetilde{p}^{n+1/2}\neq0$, which again admits two different solutions when the time step is small enough. This means that there are generically four possible outputs for the configuration~$\left(q^{n+1},p^{n+1}\right)$.
\end{remark}

%---------------------------------------------------------
%---------------------------------------------------------
%---------------- REGULARITY -----------------------------
%---------------------------------------------------------
%---------------------------------------------------------

\subsection{Regularity of the numerical solver and local preservation of the Lebesgue measure}
\label{subsec:regularity_numerical_solver_local_preservation_lebesgue_measure}

In Definition~\ref{def:numerical_solver}, the requirement that, for all~$x\in\calA_{\Delta t}$, the~$dk\times dk$-matrix~$\nabla_{y}\Phi_{\Delta t}(x,\chi_{\Delta t}(x))$ is invertible is used from a theoretical viewpoint to prove that the numerical flow with $S$-reversibiliy checks is measurable, see Lemma~\ref{lem:B_open_set}. It is a natural assumption in order to prove that the numerical solver~$\chi_{\Delta t}$ is a~$\calC^{1}$ map, see Lemma~\ref{lem:numerical_solver_C1}. This regularity is needed to check that the numerical flow preserves the Lebesgue measure, see~\eqref{eq:lebesgue_measure_preservation} and Proposition~\ref{prop:numerical_flow_symplectic}. In practice, this is also needed when using Newton's method to solve the implicit problem, see Remark~\ref{rem:case_k_2} and Algorithm~\ref{alg:newton_idealized}. It is assumed that the set~$\calA_{\Delta t}$ is open, which is convenient to discuss the regularity of~$\chi_{\Delta t}$, and will also be needed in order to construct a measurable numerical flow that is defined globally on~$\calX$ and satisfies the two fundamental properties~\eqref{eq:lebesgue_measure_preservation} and~\eqref{eq:S_reversibility_Stormer_Verlet_explicit}, see Proposition~\ref{prop:psi_involution}. We refer to Section~\ref{sec:construction_practical_numerical_flows} where we discuss theoretical constructions and practical algorithms to define numerical flows.

We first show that a numerical solver is in fact~$\calC^{1}$ and not just continuous, as required in Definition~\ref{def:numerical_solver}.
\begin{lemma}
  \label{lem:numerical_solver_C1}
  Let~$k\geqslant 1$,~$\Delta t>0$ and~$\chi_{\Delta t}$ be a numerical solver associated with a numerical scheme~$\Phi_{\Delta t}$. Then~$\chi_{\Delta t}$ is~$\calC^{1}$ and it holds
  \begin{equation*}
    \forall x\in\calA_{\Delta t},\qquad
    \nabla\chi_{\Delta t}(x)=-\Bigl(\nabla_{y}\Phi_{\Delta t}\bigl(x,\chi_{\Delta t}(x)\bigr)\Bigr)^{-1}\nabla_{x}\Phi_{\Delta t}\bigl(x,\chi_{\Delta t}(x)\bigr).
  \end{equation*}
\end{lemma}
\begin{proof}
  The fact that~$\chi_{\Delta t}$ is~$\calC^{1}$ is a direct consequence of the implicit function theorem and the continuity of~$\chi_{\Delta t}$. The expression of the gradient of~$\chi_{\Delta t}$ is then obtained by differentiating~$\Phi_{\Delta t}(x,\chi_{\Delta t}(x))=0$ with respect to~$x$.
\end{proof}

We now discuss the preservation of the Lebesgue measure by the numerical flow. Hamiltonian dynamics enjoy nice structural properties as their exact flows are themselves symplectic. It is natural in this case to choose a numerical scheme and a numerical solver which, when implemented, yield a symplectic numerical flow~$\varphi_{\Delta t}$ (\emph{i.e.}~\eqref{eq:symplecticity} is satisfied). In this case, the numerical flow satisfies property~\eqref{eq:lebesgue_measure_preservation} since~$\left\lvert \det\nabla\varphi_{\Delta t}\right\rvert=1$. Examples of symplectic integrators include IMR and GSV, as the following proposition states. The result is a straightforward adaptation of~\cite[Theorems~VI.3.4 and~VI.3.5]{hairer_2006} in our context, see Section~\ref{subsec:proof_symplectic} for the proof.

\begin{proposition}
\label{prop:numerical_flow_symplectic}
  Let~$H$ be a~$\calC^{2}$ Hamiltonian function, and fix~$\Delta t>0$. Assume that~$\chi_{\Delta t}\colon\calA_{\Delta t}\to\calX$ (respectively~$\chi_{\Delta t}\colon\calA_{\Delta t}\to\calX^{2}$) is a numerical solver for the IMR (respectively the GSV) numerical scheme. Then its associated numerical flow~$\varphi_{\Delta t}$ is symplectic on~$\calA_{\Delta t}$. In particular, for all~$x\in\calA_{\Delta t}$, it holds~$\left\lvert\det\nabla\varphi_{\Delta t}(x)\right\rvert=1$.
\end{proposition}

This motivates the following assumption.

\begin{assumption}
  \label{ass:numerical_flow_symplectic}
  Let~$\Delta t>0$ and~$\chi_{\Delta t}\colon\calA_{\Delta t}\to\calX$ be a numerical solver associated with the numerical scheme~$\Phi_{\Delta t}$. The associated numerical flow~$\varphi_{\Delta t}\colon\calA_{\Delta t}\to\calX$ is such that~$\left\lvert\det\nabla\varphi_{\Delta t}(x)\right\rvert=1$ for any~$x\in\calA_{\Delta t}$.
\end{assumption}

Note that it makes sense to write~$\nabla\varphi_{\Delta t}$ since a numerical flow is~$\calC^1$ on the open set~$\calA_{\Delta t}$ by Lemma~\ref{lem:numerical_solver_C1}. Assumption~\ref{ass:numerical_flow_symplectic} implies that~$\varphi_{\Delta t}$ locally preserves the Lebesgue measure on~$\calA_{\Delta t}$.

\begin{remark}
  If the numerical flow does not satisfy Assumption~\ref{ass:numerical_flow_symplectic}, the Metropolis procedure can be modified to lead to unbiased HMC schemes, see~\cite{fang_2014}.
\end{remark}

%---------------------------------------------------------
%---------------------------------------------------------
%---------------- S-REVERSIBILITY ------------------------
%---------------------------------------------------------
%---------------------------------------------------------
\subsection{\texorpdfstring{$S$-reversibility}{S-reversibility}}
\label{subsec:S_reversibility}

Denoting by~$S\colon\calX\to\calX$ an involution, we now define the concept of~$S$-reversibility for a numerical scheme~$\Phi_{\Delta t}$ and a numerical solver~$\chi_{\Delta t}$. Let us introduce
\begin{equation}
    \label{eq:psi_dt}
    \psi_{\Delta t}=S\circ\varphi_{\Delta t},    
\end{equation}
where~$\varphi_{\Delta t}$ is the numerical flow associated with~$\chi_{\Delta t}$. For Hamiltonian dynamics, the involution~$S$ is typically the momentum reversal map~\eqref{eq:S_momentum_reversal}.

\begin{definition}[$S$-reversibility for a numerical scheme, a numerical solver and a numerical flow]
  \label{def:S_rev_numerical_scheme}
  Fix~$\Delta t>0$. A numerical scheme~$\Phi_{\Delta t}$ is said to be S-reversible if
  \begin{equation*}
    \forall(x,x_1,\dots x_k)\in\calX\times\calX^{k},\qquad\Phi_{\Delta t}(x,x_1,\dots,x_k)=0\Longleftrightarrow \Phi_{\Delta t}(S(x_k), \dots, S(x_1), S(x))=0.
  \end{equation*}
  A numerical solver~$\chi_{\Delta t}\colon\calA_{\Delta t}\to\calX^{k}$ is~$S$-reversible if, for any~$x\in \calA_{\Delta t}\cap\psi_{\Delta t}^{-1}(\calA_{\Delta t})$, it holds
  \begin{align}
    \chi_{\Delta t}\circ\psi_{\Delta t}(x)
    &=\left(
    \chi_{\Delta t,1}(\psi_{\Delta t}(x)),\dots,\chi_{\Delta t,k}(\psi_{\Delta t}(x))
    \right),\nonumber\\
    &=\label{eq:s_rev_solver}
    \left(
      S\circ\chi_{\Delta t,k-1}(x),\dots,S\circ\chi_{\Delta t,1}(x),S(x)
    \right).
  \end{align}
  A numerical flow~$\varphi_{\Delta t}\colon\calA_{\Delta t}\to\calX$ is~$S$-reversible if for any~$x\in\calA_{\Delta t}\cap\psi_{\Delta t}^{-1}(\calA_{\Delta t})$, it holds
  \begin{equation*}
    \psi_{\Delta t}\circ\psi_{\Delta t}(x)=x.
  \end{equation*}
\end{definition}

To better understand what it means for a numerical solver to be~$S$-reversible, we refer to Figure~\ref{fig:understand_s_rev_solver}: if a sequence of configurations solves the numerical scheme, then, after applying~$S$ along the sequence, the time reversed sequence also solves the numerical scheme.

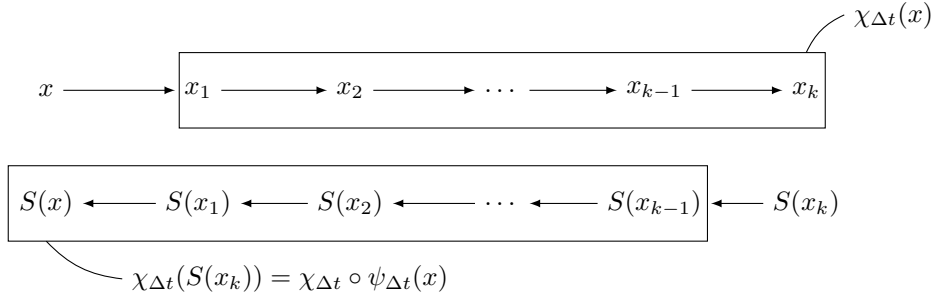
\begin{figure}
  \centering
  \begin{tikzpicture}
    \node (x) at (0,0) {$x$};
    \node (x1) at (2,0) {$x_1$};
    \node (x2) at (4,0) {$x_2$};
    \node (dots) at (6,0) {$\dots$};
    \node (xk1) at (8,0) {$x_{k-1}$};
    \node (xk) at (10,0) {$x_k$};
    
    \node (Sxk) at (10,-1.5) {$S(x_{k})$};
    \node (Sxk1) at (8,-1.5) {$S(x_{k-1})$};
    \node (dot) at (6,-1.5) {$\dots$};
    \node (Sx2) at (4,-1.5) {$S(x_{2})$};
    \node (Sx1) at (2,-1.5) {$S(x_{1})$};
    \node (Sx) at (0,-1.5) {$S(x)$};

    \draw [-latex] (x) edge (x1) (x1) edge (x2) (x2) edge (dots) (dots) edge (xk1) (xk1) edge (xk);
    \draw (1.75,-0.5) rectangle (10.25,0.5);
    \draw (10,0.5) to[bend left=20] (10.5,1) node [right] {$\chi_{\Delta t}(x)$};

    \draw [-latex] (Sxk) edge (Sxk1) (Sxk1) edge (dot) (dot) edge (Sx2) (Sx2) edge (Sx1) (Sx1) edge (Sx);
    \draw (-0.5,-2) rectangle (8.7,-1);
    \draw (0,-2) to[bend right=20] (1,-2.5) node [right] {$\chi_{\Delta t}(S(x_k))=\chi_{\Delta t}\circ\psi_{\Delta t}(x)$};
    
  \end{tikzpicture}
  \caption[]{Illustration of the~$S$-reversibility: the two sequences should contain the same elements up to applying~$S$ and inverting the order.}
  \label{fig:understand_s_rev_solver}
\end{figure}

Note that, if a numerical solver~$\chi_{\Delta t}$ is~$S$-reversible, then for all~$x\in\calA_{\Delta t}\cap\psi_{\Delta t}^{-1}\left(\calA_{\Delta t}\right)$,
\begin{equation}
    \label{eq:varphi_s_rev}
  \varphi_{\Delta t}\circ S\circ\varphi_{\Delta t}(x)
  =
  \chi_{\Delta t,k}\circ \psi_{\Delta t}(x)
  =
  S(x),
\end{equation}
so that~$\psi_{\Delta t}\circ\psi_{\Delta t}(x)=x$. Thus, if~$\chi_{\Delta t}$ is a~$S$-reversible numerical solver, then its associated numerical flow~$\varphi_{\Delta t}$ is~$S$-reversible, and~$\psi_{\Delta t}$ defined by~\eqref{eq:psi_dt} is an involution on~$\calA_{\Delta t}\cap\psi_{\Delta t}^{-1}\left(\calA_{\Delta t}\right)$. We require in~\eqref{eq:s_rev_solver} that the reversibility property is satisfied by the intermediate configurations (and not only the starting and ending configurations). This is crucial in the proof of Lemma~\ref{lem:B_open_set} below as well as for the construction of practical numerical flows, see Proposition~\ref{prop:GSV_FP}. In particular, when~$k\geqslant 2$, the fact that~$\psi_{\Delta t}^{2}(x)=x$ for all~$x\in\calA_{\Delta t}\cap\psi_{\Delta t}^{-1}\left(\calA_{\Delta t}\right)$ is usually not enough to satisfy~\eqref{eq:s_rev_solver}, except in very specific situations such as in Lemma~\ref{lem:solver_rev_from_flow_rev_gsv} below.

\begin{lemma}
  \label{lem:solver_rev_from_flow_rev_gsv}
  Fix~$\Delta t>0$ and let~$\Phi_{\Delta t}^{\GSV}$ be the map defined in~\eqref{eq:Phi_GSV}. Let~$S$ be the momentum reversal map~\eqref{eq:S_momentum_reversal}, and consider the Hamiltonian function~\eqref{eq:H_RMHMC} where~$V$ and~$D$ are~$\calC^{1}$ maps. Then checking the~$S$-reversibility of the numerical flow ($\psi_{\Delta t}\circ\psi_{\Delta t}=\id$ on~$\calA_{\Delta t}\cap\psi_{\Delta t}^{-1}\left(\calA_{\Delta t}\right)$) is enough to ensure the~$S$-reversibility of the numerical solver (\emph{i.e.}~\eqref{eq:s_rev_solver}). In fact, it is even enough to check~$S$-reversibility on positions only, \emph{i.e.}~if~$(q,p,q_1,p_1,q_2,p_2)\in\calX^{3}$ are such that 
  \begin{equation}
    \label{eq:solver_rev_flow_rev_1}
    \Phi_{\Delta t}^{\GSV}(q,p,q_1,p_1,q_2,p_2)=0,
  \end{equation}
  and if there exist~$(\tilde{q}_1,\tilde{p}_1)\in\calX$ and~$\tilde{p}\in\bbR^{m}$ such that 
  \begin{equation}
    \label{eq:solver_rev_flow_rev_2}
    \Phi_{\Delta t}^{\GSV}(S(q_2,p_2),\tilde{q}_1,\tilde{p}_1,q,\tilde{p})=0,
  \end{equation}
  then~$(\tilde{q}_1,\tilde{p}_1)=S(q_1,p_1)$ and~$\tilde{p}=-p$ so that~$(q,\tilde{p})=S(q,p)$, and~\eqref{eq:s_rev_solver} is thus satisfied.
\end{lemma}
  
\begin{proof}
  For more readability, we write explicitly~\eqref{eq:solver_rev_flow_rev_1} and~\eqref{eq:solver_rev_flow_rev_2} using the definition of the GSV numerical scheme~\eqref{eq:Phi_GSV}:
  \begin{equation*}
    \begin{pmatrix}
      q_1-q-\dfrac{\Delta t}{2}~\nabla_p H(q,p_1)\\[0.2cm]
      p_1-p+\dfrac{\Delta t}{2}~\nabla_q H(q,p_1)\\[0.2cm]
      q_2-q_1-\dfrac{\Delta t}{2}~\nabla_p H(q_2,p_1)\\[0.2cm]
      p_2-p_1+\dfrac{\Delta t}{2}~\nabla_q H(q_2,p_1)
    \end{pmatrix}=0,\qquad
    \begin{pmatrix}
      \tilde{q}_1-q_2-\dfrac{\Delta t}{2}~\nabla_p H(q_2,\tilde{p}_1)\\[0.2cm]
      \tilde{p}_1+p_2+\dfrac{\Delta t}{2}~\nabla_q H(q_2,\tilde{p}_1)\\[0.2cm]
      q-\tilde{q}_1-\dfrac{\Delta t}{2}~\nabla_p H(q,\tilde{p}_1)\\[0.2cm]
      \tilde{p}-\tilde{p}_1+\dfrac{\Delta t}{2}~\nabla_q H(q,\tilde{p}_1)
    \end{pmatrix}=0.
  \end{equation*}

  By adding the first and third components of~\eqref{eq:solver_rev_flow_rev_1}, one obtains
  \begin{equation*}
    q_2-q=\frac{\Delta t}{2}\left(\nabla_pH(q_2,p_1)+\nabla_pH(q,p_1)\right)=\Delta t\frac{D(q)+D(q_2)}{2}p_1,
  \end{equation*}
  while~\eqref{eq:solver_rev_flow_rev_2} likewise yields
  \begin{equation*}
    q-q_2=\Delta t\frac{D(q)+D(q_2)}{2}\tilde{p}_{1}.
  \end{equation*}
  The matrix~$D(q)+D(q_2)$ being invertible, one obtains~$\tilde{p}_{1}=-p_1$. Then, subtracting the fourth component of~\eqref{eq:solver_rev_flow_rev_2} from the second component of~\eqref{eq:solver_rev_flow_rev_1} and using the fact that~$\nabla_{q}H(q,-p)=\nabla_{q}H(q,p)$, it holds~$\tilde{p}=-p$. Finally, adding the first component of~\eqref{eq:solver_rev_flow_rev_1} and the third from~\eqref{eq:solver_rev_flow_rev_2} then implies that~$q_1=\tilde{q}$, using again that the Hamiltonian function is an even function of the momenta. One therefore obtains~$(\tilde{q}_1,\tilde{p}_1)=S(q_1,p_1)$ and~$\tilde{p}=-p$.
\end{proof}

Let us now check that the numerical schemes of the two running examples are~$S$-reversible with~$S$ the momentum reversal map~\eqref{eq:S_momentum_reversal}. At the continuous level, the exact Hamiltonian flow is~$S$-reversible when the Hamiltonian function is an even function of the momenta. This means that, if~$\phi_t$ denotes the Hamiltonian flow for~\eqref{eq:hamiltonian_dynamics}, then (compare with~\eqref{eq:varphi_s_rev})
\begin{equation*}
  \phi_t\circ S\circ\phi_t=S.
\end{equation*}
Numerical discretizations such as the Implicit Midpoint Rule~\eqref{eq:IMR} or the Generalized Störmer--Verlet~\eqref{eq:GSV} schemes also satisfy~$S$-reversibility.

\begin{proposition}
  \label{prop:IMR_GSV_S_reversibility}
  Let~$H$ be a~$\calC^{1}$ Hamiltonian function even in the momentum variable. Let~$S$ be the momentum reversal map~\eqref{eq:S_momentum_reversal}, and fix~$\Delta t>0$. Then both~$\Phi_{\Delta t}^{\IMR}$ defined by~\eqref{eq:Phi_IMR} and~$\Phi_{\Delta t}^{\GSV}$ defined by~\eqref{eq:Phi_GSV} are~$S$-reversible numerical schemes.
\end{proposition}
  
\begin{proof}
  Let~$(q,p,q_1,p_1)\in\calX^{2}$ such that~$\Phi_{\Delta t}^{\IMR}(q,p,q_1,p_1)=0$. Since~$\nabla_{p}H(q,-p)=-\nabla_p H(q,p)$ and~$\nabla_{q}H(q,-p)=\nabla_q H(q,p)$, it holds 
  \begin{align*}
    \Phi_{\Delta t}^{\IMR}(S(q_1,p_1), S(q,p))
    &=
    \Phi_{\Delta t}^{\IMR}(q_1,-p_1,q,-p),\\
    &=
    \begin{pmatrix}
      q-q_1-\Delta t~\nabla_{p}H\left(\dfrac{q+q_1}{2},-\dfrac{p_1+p}{2}\right)\\
      -p+p_1+\Delta t~\nabla_{q}H\left(\dfrac{q+q_1}{2},-\dfrac{p_1+p}{2}\right)
    \end{pmatrix}=0.
  \end{align*}
  Likewise, if~$(q,p,q_1,p_1,q_2,p_2)\in\calX^{3}$ satisfy~$\Phi_{\Delta t}^{\GSV}(q,p,q_1,p_1,q_2,p_2)=0$, then 
  \begin{align*}
    \Phi_{\Delta t}^{\GSV}(S(q_2,p_2),S(q_1,p_1),S(q,p))
    &=
    \Phi_{\Delta t}^{\GSV}(q_2,-p_2,q_1,-p_1,q,-p),\\
    &=
    \begin{pmatrix}
      q_1-q_2-\dfrac{\Delta t}{2}~\nabla_p H(q_2,-p_1)\\[0.2cm]
      -p_1+p_2+\dfrac{\Delta t}{2}~\nabla_q H(q_2,-p_1)\\[0.2cm]
      q-q_1-\dfrac{\Delta t}{2}~\nabla_p H(q,-p_1)\\[0.2cm]
      -p+p_1+\dfrac{\Delta t}{2}~\nabla_q H(q,-p_1)
    \end{pmatrix}=0,
  \end{align*}
  which gives the claimed result.
\end{proof}

It is of paramount importance to note that~$S$-reversibility for a numerical scheme does not imply~$S$-reversibility for an associated numerical solver. Let us list the three difficulties one may encounter in practice when building a numerical solver for a~$S$-reversible numerical scheme, choosing~$k=1$ for simplicity of exposition. For a given configuration~$x\in\calX$,
\begin{itemize}
  \item it may happen that~$x\notin\calA_{\Delta t}$, in which case the output~$\varphi_{\Delta t}(x)$ is not defined;
  \item if~$x\in\calA_{\Delta t}$, it may happen that~$\psi_{\Delta t}(x)\notin\calA_{\Delta t}$, \emph{i.e.}~$\varphi_{\Delta t}(\psi_{\Delta t}(x))$ is not defined;
  \item if~$x\in\calA_{\Delta t}$ and~$\psi_{\Delta t}(x)\in\calA_{\Delta t}$, it may happen that~$\psi_{\Delta t}^{2}(x)\neq x$, that is the numerical flow does not satisfy~$S$-reversibility for this configuration. Indeed, multiple solutions can be available at each integration steps since they are solution to an implicit problem (see \emph{e.g.}~Remark~\ref{rem:problem_S_rev_introduction}), and the implemented numerical flow may not choose the one compatible with~$S$-reversibility. Therefore, the numerical flow may not be~$S$-reversible while the numerical scheme is.
\end{itemize}

%---------------------------------------------------------
%---------------------------------------------------------
%---------------- ENFORCING S-REV ------------------------
%---------------------------------------------------------
%---------------------------------------------------------
\subsection{Enforcing \texorpdfstring{$S$}{S}-reversibility}
\label{subsec:enforcing_S_reversibility}
In order to correct for the lack of~$S$-reversibility arising from the use of implicit methods, we define, for a numerical solver~$\chi_{\Delta t}$,
\begin{equation}
  \label{eq:B_k}
  \calB_{\Delta t}=\left\lbrace 
  x\in\calA_{\Delta t}\cap\left(\psi_{\Delta t}\right)^{-1}(\calA_{\Delta t}),~
  \chi_{\Delta t}\circ\psi_{\Delta t}(x)=
  \Bigl(
    S\circ\chi_{\Delta t,k-1}(x),\dots,S\circ\chi_{\Delta t,1}(x),S(x)
  \Bigr)
  \right\rbrace.
\end{equation}
The set~$\calB_{\Delta t}$ is thus the set of configurations~$x$ for which:
\begin{itemize}
    \item each of the steps of the integrator starting from~$x$ has successfully converged, so that~$\chi_{\Delta t}(x)$ (hence~$\varphi_{\Delta t}(x)$) is well-defined;
    \item each of the steps of the integrator starting from~$\psi_{\Delta t}(x)=S\circ\varphi_{\Delta t}(x)$ has successfully converged, so that~$\chi_{\Delta t}\circ\psi_{\Delta t}(x)=\chi_{\Delta t}\circ S\circ\varphi_{\Delta t}(x)$ is well-defined, and hence also the configuration~$\varphi_{\Delta t}\circ S\circ\varphi_{\Delta t}(x)$;
    \item the map~$\chi_{\Delta t}$ is~$S$-reversible starting from the configuration~$x$, in the sense of Definition~\ref{def:S_rev_numerical_scheme}.
\end{itemize}
Let us then introduce the maps
\begin{equation}
  \label{eq:psi_rev}
  \forall x\in\calX,\qquad 
  \left\lbrace
  \begin{aligned}
    \varphi_{\Delta t}^{\rev}(x)&=\varphi_{\Delta t}(x)\,\mathbb{1}_{x\in \calB_{\Delta t}}+S(x)\,\mathbb{1}_{x\notin \calB_{\Delta t}},\\
    \psi_{\Delta t}^{\rev}(x)&=S\circ\varphi_{\Delta t}^{\rev}(x)=\psi_{\Delta t}(x)\mathbb{1}_{x\in\calB_{\Delta t}}+x\,\mathbb{1}_{x\notin\calB_{\Delta t}}.
  \end{aligned}
  \right.
\end{equation}
This is the natural way to define the flow ``with~$S$-reversibility check''~\cite{zappa_2018,lelievre_2019} based on the numerical scheme: integrations of the numerical scheme for which the numerical solver does not satisfy~$S$-reversibility are discarded, and replaced by the map~$S$, so that~$\psi_{\Delta t}^{\rev}$ is a global involution, see Lemma~\ref{prop:psi_rev_involution}. 

The following result shows that the numerical flow with~$S$-reversibility check~$\varphi_{\Delta t}^{\rev}$ is~$S$-reversible on the whole configuration space.
\begin{proposition}
    \label{prop:psi_rev_involution}
    Fix~$\Delta t>0$. Let~$\chi_{\Delta t}\colon\calA_{\Delta t}\to\calX^{k}$ be a numerical solver and~$\varphi_{\Delta t}\colon\calA_{\Delta t}\to\calX$ be the associated numerical flow. Then the numerical flow with~$S$-reversibility check~$\varphi_{\Delta t}^{\rev}$ defined by~\eqref{eq:B_k}--\eqref{eq:psi_rev} is~$S$-reversible.
\end{proposition}

\begin{proof}
  Fix~$x\in\calX$. If~$x\notin\calB_{\Delta t}$,~$\psi_{\Delta t}^{\rev}(x)=x\notin\calB_{\Delta t}$ so that~$\psi_{\Delta t}^{\rev}\circ\psi_{\Delta t}^{\rev}(x)=x$. If~$x\in\calB_{\Delta t}$, we claim that~$\psi_{\Delta t}^{\rev}(x)\in\calB_{\Delta t}$. Indeed, it holds~$\psi_{\Delta t}^{\rev}(x)=\psi_{\Delta t}(x)\in\calA_{\Delta t}$ by the definition of~$\calB_{\Delta t}$. Moreover,
  \begin{equation}
    \label{eq:tmp}
    \psi_{\Delta t}\circ\psi_{\Delta t}^{\rev}(x)
    =
    \psi_{\Delta t}\circ\psi_{\Delta t}(x)
    =
    S\circ\chi_{\Delta t,k}\circ \psi_{\Delta t}(x)
    =
    S\circ S(x)=x\in\calA_{\Delta t},
  \end{equation}
  where we used~\eqref{eq:B_k} for the third equality. Using~\eqref{eq:B_k} again, it holds 
  \begin{equation*}
    \forall 1\leqslant i\leqslant k-1,\qquad
    \chi_{\Delta t,k-i}\left(
      \psi_{\Delta t}^{\rev}(x)
    \right)=\chi_{\Delta t,k-i}\left(
      \psi_{\Delta t}(x)
    \right)=S\circ\chi_{\Delta t,i}(x),
  \end{equation*}
  so that
  \begin{align*}
    \MoveEqLeft[12]
    \Bigl(
      S\circ\chi_{\Delta t,k-1}(\psi_{\Delta t}^{\rev}(x)),\dots,S\circ\chi_{\Delta t,1}(\psi_{\Delta t}^{\rev}(x)),S(\psi_{\Delta t}^{\rev}(x))
    \Bigr)\\
    &=
    \Bigl(
      S\circ S\circ\chi_{\Delta t,1}(x),\dots,S\circ S\circ\chi_{\Delta t,k-1}(x),S\circ S\circ\varphi_{\Delta t}(x)
    \Bigr),\\
    &=
    \chi_{\Delta t}(x)
    =
    \chi_{\Delta t}\circ\psi_{\Delta t}(\psi_{\Delta t}^{\rev}(x)),
  \end{align*}
  where we used~\eqref{eq:tmp} for the last equality. Hence, if~$x\in\calB_{\Delta t}$, then~$\psi_{\Delta t}^{\rev}(x)\in\calB_{\Delta t}$, and it holds~$\psi_{\Delta t}^{\rev}(\psi_{\Delta t}^{\rev}(x))=\psi_{\Delta t}(\psi_{\Delta t}(x))=x$. This concludes the proof.
\end{proof}

\noindent Note that, even if~$\calA_{\Delta t}$ is nonempty,~$\calB_{\Delta t}$ can be empty. We make precise in Section~\ref{subsec:numerical_solvers_fixed_point} a setting where~$\calB_{\Delta t}$ is nonempty under reasonable assumptions for both running examples.

To show that the map~$\varphi_{\Delta t}^{\rev}$ satisfies the two fundamental properties~\eqref{eq:lebesgue_measure_preservation} and~\eqref{eq:S_reversibility_Stormer_Verlet_explicit}, we proceed as follows: first, we show that~$\calB_{\Delta t}$ is an open set so that~$\varphi_{\Delta t}^{\rev}$ is measurable; second, we prove that~$\varphi_{\Delta t}\colon\calB_{\Delta t}\to\calB_{\Delta t}$ is a~$\calC^{1}$-diffeomorphism. Together with the fact that~$\left\lvert\det\nabla\varphi_{\Delta t}\right\rvert=1$ on~$\calB_{\Delta t}\subset\calA_{\Delta t}$ (see Assumption~\ref{ass:numerical_flow_symplectic}), this implies that~$\psi_{\Delta t}^{\rev}\colon\calX\to\calX$ is globally well-defined and preserves the Lebesgue measure on~$\calX$, as stated in Proposition~\ref{prop:psi_involution} thanks to the following additional assumption.

\begin{assumption}
  \label{ass:S_C1_involution}
  Fix~$\Delta t>0$. The numerical scheme~$\Phi_{\Delta t}$ is~$S$-reversible where~$S\colon\calX\to\calX$ is a~$\calC^{1}$ involution such that~$\left\lvert\det\nabla S\right\rvert=1$.
\end{assumption}
For instance, both the IMR and GSV numerical schemes satisfy Assumption~\ref{ass:S_C1_involution} with~$S$ the momentum reversal map~\eqref{eq:S_momentum_reversal}, see Proposition~\ref{prop:IMR_GSV_S_reversibility}.

Following the proofs of~\cite{lelievre_2019}, we show in Lemma~\ref{lem:B_open_set} below that~$\calB_{\Delta t}$ is an open set of~$\calX$ as the union of path connected components of the open set~$\calA_{\Delta t}\bigcap\psi_{\Delta t}^{-1}\left(\calA_{\Delta t}\right)$. Since this set is a topological manifold, it is locally path connected, which implies that path connected components of any open subset of~$\calA_{\Delta t}\bigcap\psi_{\Delta t}^{-1}\left(\calA_{\Delta t}\right)$ are actually its connected components, and that these connected components are open subsets of~$\calA_{\Delta t}\bigcap\psi_{\Delta t}^{-1}(\calA_{\Delta t})$ (see for instance~\cite[Theorem~25.5]{munkres_2000}). The proof of Lemma~\ref{lem:B_open_set} is postponed to Section~\ref{subsec:proof:B_open_set}.

\begin{lemma}
  \label{lem:B_open_set}
  Let~$\Phi_{\Delta t}$ a numerical scheme such that Assumption~\ref{ass:S_C1_involution} holds. Let~$\chi_{\Delta t}\colon\calA_{\Delta t}\to\calX$ be a numerical solver associated with~$\Phi_{\Delta t}$, and consider a path connected component~$C$ of the open set~$\calA_{\Delta t}\bigcap\psi_{\Delta t}^{-1}(\calA_{\Delta t})$. If there exists~$x_{\star}\in C$ such that
  \begin{equation*}
    \chi_{\Delta t}\circ\psi_{\Delta t}(x_{\star})=
    \Bigl(
        S\circ\chi_{\Delta t,k-1}(x_{\star}),\dots,S\circ\chi_{\Delta t,1}(x_{\star}),S(x_{\star})
    \Bigr),
  \end{equation*}
  then 
  \begin{equation}
      \label{eq:reversibility_open_set}
      \forall x\in C,\qquad 
        \chi_{\Delta t}\circ\psi_{\Delta t}
      (x)=
      \Bigl(
        S\circ\chi_{\Delta t,k-1}(x),\dots,S\circ\chi_{\Delta t,1}(x),S(x)
      \Bigr).
  \end{equation}
  The set~$\calB_{\Delta t}$ defined by~\eqref{eq:B_k} is therefore the union of path connected components of the open set~$\calA_{\Delta t}\bigcap\psi_{\Delta t}^{-1}(\calA_{\Delta t})$, hence an open set of~$\calX$.
\end{lemma}

Lemma~\ref{lem:B_open_set} is the main ingredient for the following proposition to hold. This is one of the main theoretical results of this work as it implies the unbiasedness of usual HMC schemes using the numerical flow with~$S$-reversibility check~\eqref{eq:psi_rev}, see Propositions~\ref{prop:HMC_mu_invariance} and~\ref{prop:GHMC_mu_invariance} below.

\begin{proposition}
  \label{prop:psi_involution}
  Under Assumptions~\ref{ass:numerical_flow_symplectic} and~\ref{ass:S_C1_involution}, the map~$\psi_{\Delta t}^{\rev}\colon\calX\to\calX$ defined by~\eqref{eq:B_k}--\eqref{eq:psi_rev} is well-defined and preserves the Lebesgue measure on~$\calX$.
\end{proposition}

Proposition~\ref{prop:psi_involution} is proved as~\cite[Proposition~2.4]{lelievre_2019}. We nonetheless recall the proof here for the sake of completeness.

\begin{proof}
  Since~$\calB_{\Delta t}$ is open,~$\varphi_{\Delta t}^{\rev}$ and~$\psi_{\Delta t}^{\rev}$ are measurable maps. From Assumptions~\ref{ass:numerical_flow_symplectic} and~\ref{ass:S_C1_involution}, one obtains~$\left\lvert\det\nabla\psi_{\Delta t}^{\rev}\right\rvert=1$ on~$\calB_{\Delta t}$. To show that~$\psi_{\Delta t}^{\rev}$ globally preserves the Lebesgue measure, we decompose any measurable set~$\calE\subset \calX$ as its intersection with~$\calB_{\Delta t}$ and its complementary, noting that the sets~$\left(\psi_{\Delta t}^{\rev}\right)^{-1}\left(\calE\cap\calB_{\Delta t}\right)\subset\calB_{\Delta t}$ and~$\left(\psi_{\Delta t}^{\rev}\right)^{-1}\left(\calE\cap\left(\calX\setminus\calB_{\Delta t}\right)\right)\subset\calX\setminus\calB_{\Delta t}$ are disjoint. Therefore, using the notation~$\left\lvert \calE\right\rvert=\int_{\calX}\mathbb{1}_{\calE}(x)\,\rmd x$,
  \begin{align*}
      \left\lvert\left(\psi_{\Delta t}^{\rev}\right)^{-1}(\calE)\right\rvert
      &=
      \left\lvert
        \left(
          \psi_{\Delta t}^{\rev}
        \right)^{-1}
        \left(
          \calE\cap \calB_{\Delta t}
        \right)
      \right\rvert
      +
      \left\lvert
        \left(
          \psi_{\Delta t}^{\rev}
        \right)^{-1}
        \left(
          \calE\cap 
          \left(
            \calX\setminus\calB_{\Delta t}
          \right)
        \right)
      \right\rvert,\\
      &=
      \left\lvert
        \left(
          \psi_{\Delta t}
        \right)^{-1}
        \left(
          \calE\cap\calB_{\Delta t}
        \right)
      \right\rvert
      +
      \left\lvert
        \calE\cap
        \left(
          \calX\setminus\calB_{\Delta t}
        \right)
      \right\rvert
      =
      \left\lvert
        \calE\cap \calB_{\Delta t}
      \right\rvert
      +
      \left\lvert 
        \calE\cap
        \left(
          \calX\setminus\calB_{\Delta t}
        \right)
      \right\rvert
      =
      \left\lvert \calE\right\rvert,
  \end{align*}
  where we used the definition of~$\psi_{\Delta t}^{\rev}$ to obtain the second equality, and the fact that~$\psi_{\Delta t}$ is an involution hence a~$\calC^{1}$ diffeomorphism on~$\calB_{\Delta t}$ for the third one. This shows that~$\psi_{\Delta t}^{\rev}$ is a globally defined, Lebesgue measure preserving map on~$\calX$.
\end{proof}

Proposition~\ref{prop:psi_involution} along with Proposition~\ref{prop:psi_rev_involution} show that, starting from a~$S$-reversible numerical scheme satisfying Assumption~\ref{ass:S_C1_involution} along with a numerical solver whose numerical flow satisfies Assumption~\ref{ass:numerical_flow_symplectic}, one can always construct a numerical flow with~$S$-reversibility check~$\varphi_{\Delta t}^{\rev}$ defined by~\eqref{eq:B_k}--\eqref{eq:psi_rev} which satisfies the two fundamental properties~\eqref{eq:lebesgue_measure_preservation} and~\eqref{eq:S_reversibility_Stormer_Verlet_explicit} on the whole configuration space for any time step~$\Delta t$.

%---------------------------------------------------------
%---------------------------------------------------------
%---------------------------------------------------------
%---------------------------------------------------------
%---------------------------------------------------------
%---------------- UNBIASED HMC ---------------------------
%---------------------------------------------------------
%---------------------------------------------------------
%---------------------------------------------------------
%---------------------------------------------------------
%---------------------------------------------------------
\section{Unbiased HMC schemes for implicit integrators}
\label{sec:unbiased_HMC_schemes_implicit_integrators}
We now present (G)HMC schemes whose proposals are computed using implicit integrators, such as the IMR or the GSV numerical schemes, see Definitions~\ref{def:IMR} and~\ref{def:GSV}. We propose in Section~\ref{subsec:hmc} an unbiased one-step HMC algorithm, and in Section~\ref{subsec:ghmc} an unbiased GHMC algorithm to sample the phase-space Boltzmann--Gibbs measure~$\mu$ defined by~\eqref{eq:measure_non_separable} when~$H$ is an even function the momenta that can be nonseparable. It is shown in Section~\ref{subsec:RMHMC_overdamped_Langevin} that the one-step HMC algorithm combined with the Hamiltonian function~\eqref{eq:H_RMHMC} can be used to obtain a weakly consistent discretization of order 1 of the overdamped Langevin dynamics with a position-dependent diffusion coefficient, improving upon the weak consistency order 1/2 of the standard MALA algorithm~\cite{fathi_2017}. In Section~\ref{subsec:ghmala}, we propose an unbiased version of the Generalized Hybrid Metropolis--Adjusted Langevin Algorithm (GHMALA) introduced in~\cite{poncet_2017}, whose aim is to sample the Gibbs measure~$\pi$ defined by~\eqref{eq:invariant_measure_position_RMHMC} using a discretization of the overdamped Langevin dynamics with nonreversible drifts.

%---------------------------------------------------------
%---------------------------------------------------------
%---------------- HMC ------------------------------------
%---------------------------------------------------------
%---------------------------------------------------------
\subsection{Hamiltonian Monte Carlo}
\label{subsec:hmc}

\begin{algorithm}
  \caption{One-step HMC algorithm with~$S$-reversibility checks}
  \label{alg:hmc_scheme_rev_check}
  Consider an initial condition~$(q^{0},p^{0})\in\calX$, and set~$n=0$.
  \begin{enumerate}[label={[\thealgorithm.\roman*]}, align=left]
      \item \label{step:hmc_1} Sample~$\tilde{p}^{n}$ according to~$Z_{q^n}^{-1}\rme^{-H(q^{n},\cdot)}$ with~$Z_{q^n}=\int_{\bbR^{m}}\rme^{-H(q^{n},p)}\rmd p$, the conditional probability density of~$p$ given~$q^{n}$;
      \item \label{step:hmc_2} Apply one step of the Hamiltonian dynamics with momentum reversal and~$S$-reversibility check:
      \begin{equation*}
          (\widetilde{q}^{n+1}, \widetilde{p}^{n+1})=\psi_{\Delta t}^{\rev}(q^{n},\widetilde{p}^{n}),
      \end{equation*}
      where~$\psi_{\Delta t}^{\rev}$ is defined by~\eqref{eq:B_k}--\eqref{eq:psi_rev};
      \item \label{step:hmc_3} Draw a random variable~$U^{n}$ with uniform law on~$[0,1]$:
      \begin{itemize}[label=$\bullet$]
          \item if~$U^{n}\leqslant \exp(-H(\widetilde{q}^{n+1},\widetilde{p}^{n+1})+H(q^{n},\widetilde{p}^{n}))$, accept the proposal and set~$(q^{n+1},p^{n+1})=(\widetilde{q}^{n+1},\widetilde{p}^{n+1})$;
          \item else reject the proposal and set~$(q^{n+1},p^{n+1})=(q^{n},\widetilde{p}^{n})$;
      \end{itemize}
      \item Increment~$n$ and go back to~\ref{step:hmc_1}.
  \end{enumerate}
\end{algorithm}

The one-step unbiased HMC method we propose corresponds to Algorithm~\ref{alg:hmc_scheme_rev_check}. The HMC algorithm yields actually a Markov chain in~$q^{n}$. The GHMC algorithm is a variant where in step~\ref{step:hmc_1}, the momenta are partially refreshed starting from the previous momenta~$p^{n}$, see Algorithm~\ref{alg:ghmc_scheme_rev_check} below. Notice that it is assumed that one can sample exactly the conditional probability measure~$\mu(\rmd p|q^{n})$ in step~\ref{step:hmc_1}: this is typically the case in practice when this measure is a centered Gaussian probability measure with a covariance matrix which depends on~$q$, see in particular the Riemann Manifold HMC algorithm in Section~\ref{subsec:RMHMC_overdamped_Langevin}.

The Markov chain~$(q^{n},p^{n})_{n\geqslant0}$ generated by Algorithm~\ref{alg:hmc_scheme_rev_check} admits~$\mu$ as an invariant probability measure since the algorithm is the composition of two steps which leave~$\mu$ invariant:
\begin{itemize}
  \item the resampling of momenta in step~\ref{step:hmc_1};
  \item the Metropolis--Hastings steps~\ref{step:hmc_2}--\ref{step:hmc_3}.
\end{itemize}
The invariance of~$\mu$ by the Metropolis--Hastings step follows from the reversibility of the associated Markov chain, which crucially relies on the fact that~$\psi_{\Delta t}^{\rev}$ satisfies the two properties~\eqref{eq:lebesgue_measure_preservation} and~\eqref{eq:S_reversibility_Stormer_Verlet_explicit}. This is shown in the following proposition.
\begin{proposition}
  \label{prop:HMC_mu_invariance}
  Suppose that Assumptions~\ref{ass:numerical_flow_symplectic} and~\ref{ass:S_C1_involution} hold. Then the probability measure~$\mu$ is invariant by Algorithm~\ref{alg:hmc_scheme_rev_check} with~$\psi_{\Delta t}^{\rev}$ defined by~\eqref{eq:B_k}--\eqref{eq:psi_rev}.
\end{proposition}

\begin{proof}
  This result is standard and is proven for instance in~\cite[Section~2.1.4]{lelievre_2010}. The only computation that differs in the proof is the one associated with the Metropolis--Hastings ratio in steps~\ref{step:hmc_2}--\ref{step:hmc_3} of Algorithm~\ref{alg:hmc_scheme_rev_check}. Let us prove that, for a proposed move from~$(q,p)$ to~$(q',p')$, this ratio writes
  \begin{equation}
    \label{eq:ratio_invariant_measure}
    \frac{\delta_{\psi_{\Delta t}^{\rev}(q',p')}\left(\rmd q\,\rmd p\right)\exp\left(-H(q',p')\right)\rmd q'\,\rmd p'}{\delta_{\psi_{\Delta t}^{\rev}(q,p)}\left(\rmd q'\,\rmd p'\right)\exp\left(-H(q,p)\right)\rmd q\,\rmd p}=
    \exp\left(-H(q',p')+H(q,p)\right).
  \end{equation}
  This computation is performed as in the proof of~\cite[Lemma~2.11]{lelievre_2019}. Equation~\eqref{eq:ratio_invariant_measure} is equivalent to 
  \begin{equation*}
    \delta_{\psi_{\Delta t}^{\rev}(q',p')}\left(\rmd q\,\rmd p\right)\rmd q'\,\rmd p'=\delta_{\psi_{\Delta t}^{\rev}(q,p)}\left(\rmd q'\,\rmd p'\right)\rmd q\,\rmd p.
  \end{equation*}
  To prove that this identity holds, we fix a bounded continuous function~$f\colon\calX\times\calX\to\bbR$. Then,
  \begin{align*}
    &\int_{(q,p)\in\calX}\int_{(q',p')\in\calX}f(q,p,q',p')\delta_{\psi_{\Delta t}^{\rev}(q',p')}\left(\rmd q\,\rmd p\right)\rmd q'\,\rmd p'\\
    \quad&=
    \int_{(q',p')\in\calX}f(\psi_{\Delta t}^{\rev}(q',p'),q',p')\,\rmd q'\,\rmd p',\\
    \quad&=
    \int_{(q',p')\in\calB_{\Delta t}}f(\psi_{\Delta t}(q',p'),q',p')\,\rmd q'\,\rmd p'
    +
    \int_{(q',p')\in\calX\setminus\calB_{\Delta t}}f(q',p',q',p')\,\rmd q'\,\rmd p',\\
    \quad&=
    \int_{(q,p)\in\calB_{\Delta t}}f\left(q,p,\left(\psi_{\Delta t}\right)^{-1}(q,p)\right)\,\rmd q\,\rmd p
    +
    \int_{(q,p)\in\calX\setminus\calB_{\Delta t}}f(q,p,q,p)\,\rmd q\,\rmd p,\\
    &=
    \int_{(q,p)\in\calB_{\Delta t}}f(q,p,\psi_{\Delta t}(q,p))\,\rmd q\,\rmd p+\int_{(q,p)\in\calX\setminus\calB_{\Delta t}}f(q,p,q,p)\,\rmd q\,\rmd p,\\
    &=
    \int_{(q,p)\in\calB_{\Delta t}}f(q,p,\psi_{\Delta t}^{\rev}(q,p))\,\rmd q\,\rmd p+\int_{(q,p)\in\calX\setminus\calB_{\Delta t}}f(q,p,\psi_{\Delta t}^{\rev}(q,p))\,\rmd q\,\rmd p,\\
    &=
    \int_{(q,p)\in\calX}f(q,p,\psi_{\Delta t}^{\rev}(q,p))\,\rmd q\,\rmd p,\\
    &=
    \int_{(q,p)\in\calX}\int_{(q',p')\in\calX}f(q,p,q',p')\delta_{\psi_{\Delta t}^{\rev}(q,p)}(\rmd q'\,\rmd p')\,\rmd q\,\rmd p,
  \end{align*}
  where we used 
  \begin{itemize}
    \item for the second and fifth equalities that~$\psi_{\Delta t}^{\rev}(x)=\psi_{\Delta t}(x)$ on~$\calB_{\Delta t}$ and~$\psi_{\Delta t}^{\rev}(x)=x$ on~$\calX\setminus\calB_{\Delta t}$ (see~\eqref{eq:B_k}--\eqref{eq:psi_rev});
    \item for the third equality, a change of variables~$(q,p)=\psi_{\Delta t}(q',p')$ on~$\calB_{\Delta t}$ (which is an open set thanks to Lemma~\ref{lem:B_open_set}) using the fact that~$\psi_{\Delta t}$ is a~$\calC^{1}$-diffeomorphism on~$\calB_{\Delta t}$ and that~$\left\lvert\det\nabla\psi_{\Delta t}\right\rvert=1$ on~$\calB_{\Delta t}$ (see Assumptions~\ref{ass:numerical_flow_symplectic} and~\ref{ass:S_C1_involution});
    \item for the fourth equality that~$\psi_{\Delta t}$ is an involution on~$\calB_{\Delta t}$ (see~\eqref{eq:B_k}--\eqref{eq:psi_rev}).
  \end{itemize}
  This concludes the proof.
\end{proof}

Note that Algorithm~\ref{alg:hmc_scheme_rev_check} in fact only generates a Markov chain~$(q^{n})_{n\geqslant0}$ since momenta are resampled at each iteration. The Markov chain~$(q^{n})_{n\geqslant0}$ samples without bias the marginal of~$\mu$ in position. This marginal does not have a simple analytical form in general, unless choosing a specific Hamiltonian function as for RMHMC. Indeed, in this case, the Hamiltonian function is~\eqref{eq:H_RMHMC}, and the marginal in position is the probability measure~$\pi$ defined by~\eqref{eq:invariant_measure_position_RMHMC}. For RMHMC, step~\ref{step:hmc_1} simply consists in sampling the Gaussian distribution~$\calN\left(0,D(q^{n})^{-1}\right)$, see~\cite{girolami_2011}. We will come back to the RMHMC algorithm in Section~\ref{subsec:RMHMC_overdamped_Langevin}, where it will be proven that it yields a consistent discretization of the overdamped Langevin dynamics with position-dependent diffusion coefficients.

%---------------------------------------------------------
%---------------------------------------------------------
%---------------- GHMC -----------------------------------
%---------------------------------------------------------
%---------------------------------------------------------
\subsection{Generalized Hamiltonian Monte Carlo}
\label{subsec:ghmc}

The GHMC method is a popular variant of HMC which relies on a discretization of the Langevin dynamics:
\begin{equation}
  \label{eq:langevin_dynamics}
  \left\lbrace
  \begin{aligned}
    \rmd q_t &= \nabla_{p}H(q_t,p_t)\,\rmd t,\\
    \rmd p_t &= -\nabla_{q}H(q_t,p_t)\,\rmd t-\gamma\nabla_pH(q_t,p_t)\,\rmd t+\sqrt{2\gamma}\,\rmd W_t,
  \end{aligned}
  \right.
\end{equation}
where~$\gamma>0$ is the friction parameter, and~$(W_t)_{t\geqslant 0}$ is a standard~$d$-dimension Brownian motion. We present in Section~\ref{subsubsec:ghmc_algorithm} the GHMC algorithm, and formally discuss in Section~\ref{subsubsec:consistency_GHMC_Langevin_dynamics} its weak consistency with the Langevin dynamics.

\subsubsection{The GHMC algorithm}
\label{subsubsec:ghmc_algorithm}

\begin{algorithm}
  \caption{GHMC algorithm with~$S$-reversibility checks}
  \label{alg:ghmc_scheme_rev_check}
  Consider an initial condition~$(q^{0},p^{0})\in\calX$, and set~$n= 0$.
  \begin{enumerate}[label={[\thealgorithm.\roman*]}, align=left]
      \item\label{step:ghmc_1} Evolve the momenta by integrating the fluctuation-dissipation part of the Langevin dynamics with time increment~$\Delta t/2$:~$(q^{n},p^{n+1/4})=\varphi_{\Delta t/2}^{\FD}(q^{n},p^{n},\Xi^n)$.
      \item\label{step:ghmc_2} Apply one step of the Hamiltonian dynamics with momentum reversal and~$S$-reversibility check:
      \begin{equation*}
          (\widetilde{q}^{n+1}, \widetilde{p}^{n+3/4})=\psi_{\Delta t}^{\rev}(q^{n},p^{n+1/4}),
      \end{equation*}
      where~$\psi_{\Delta t}^{\rev}$ is defined in~\eqref{eq:B_k}--\eqref{eq:psi_rev};
      \item\label{step:ghmc_3} Draw a random variable~$U^{n}$ with uniform law on~$(0,1)$:
      \begin{itemize}[label=$\bullet$]
          \item if~$U^{n}\leqslant \exp(-H(\widetilde{q}^{n+1},\widetilde{p}^{n+3/4})+H(q^{n},p^{n+1/4}))$, accept the proposal and set~$(q^{n+1},p^{n+3/4})=(\widetilde{q}^{n+1},\widetilde{p}^{n+3/4})$;
          \item else reject the proposal and set~$(q^{n+1},p^{n+3/4})=(q^{n},p^{n+1/4})$.
      \end{itemize}
      \item\label{step:ghmc_4} Reverse the momenta as~$\widetilde{p}^{n+1}=-p^{n+3/4}$.
      \item\label{step:ghmc_5} Evolve the momenta by integrating the fluctuation-dissipation part with time increment~$\Delta t/2$:~$(q^{n+1},p^{n+1})=\varphi_{\Delta t/2}^{\FD}(q^{n+1},\widetilde{p}^{n+1},\Xi^{n+1/2})$.
      \item Increment~$n$ and go back to~\ref{step:ghmc_1}.
  \end{enumerate}
\end{algorithm}

The GHMC algorithm is obtained thanks to a splitting technique, considering separately the Hamiltonian and fluctuation-dissipation parts of the dynamics~\eqref{eq:langevin_dynamics}. The latter part of the dynamics induces a partial refreshment of the momenta instead of a full resampling as in HMC. A Strang splitting based on these elements leads to Algorithm~\ref{alg:ghmc_scheme_rev_check}. To integrate the fluctuation-dissipation part of the dynamics for a fixed position~$q^{n}$, \emph{i.e.}~the elementary dynamics
\begin{equation}
  \label{eq:FD_dynamics}
  \rmd p_t = -\gamma\nabla_pH(q^{n},p_t)\,\rmd t+\sqrt{2\gamma}\,\rmd W_t,
\end{equation}
we use an explicit numerical scheme which preserves the invariant measure~$\mu$. The associated numerical flow, for a time step~$\Delta t$, is denoted by~$\varphi_{\Delta t}^{\FD}$. In the case when~$H$ is quadratic in~$p$, equation~\eqref{eq:FD_dynamics} is simply an Ornstein--Uhlenbeck process and one can rely on a mid-point scheme or on an analytic integration to define~$\varphi_{\Delta t}^{\FD}$. This is for example what can be done for the Hamiltonian function~\eqref{eq:H_RMHMC}, see~\eqref{eq:OU_update_rmhmc} below. Alternatively,~$\varphi_{\Delta t}^{\FD}$ can be defined by a one-step Euler--Maruyama discretization corrected by a Metropolis--Hastings acceptance-rejection step: this is the celebrated MALA algorithm~\cite{rossky_1978,roberts_1998}. Typically, the map~$\varphi_{\Delta t}^{\FD}$ takes as input a configuration~$(q,p)\in\calX$ as well as random numbers (for instance a Gaussian increment when discretizing the Brownian motion or a uniformly drawn number in~$[0,1]$ for the Metropolis--Hastings procedure) which we regroup into a vector~$\Xi$.

Step~\ref{step:ghmc_4} is only needed for the consistency of GHMC with respect to the Langevin dynamics, see Section~\ref{subsubsec:consistency_GHMC_Langevin_dynamics}. When accepting the proposal in step~\ref{step:ghmc_3}, the sign of the momenta of the output is not changed over the iteration.

Since the momenta updates~\ref{step:ghmc_1} and~\ref{step:ghmc_5}, the Metropolis--Hastings procedure~\ref{step:ghmc_2}--\ref{step:ghmc_3} (see the proof of Proposition~\ref{prop:HMC_mu_invariance}) and the momentum reversal~\ref{step:ghmc_4} are all reversible with respect to~$\mu$, the Markov chain obtained from this algorithm admits~$\mu$ as an invariant probability measure, as the following proposition states.

\begin{proposition}
  \label{prop:GHMC_mu_invariance}
  Suppose that Assumptions~\ref{ass:numerical_flow_symplectic} and~\ref{ass:S_C1_involution} hold. The measure~$\mu$ is invariant by Algorithm~\ref{alg:ghmc_scheme_rev_check} with~$\psi_{\Delta t}^{\rev}$ defined by~\eqref{eq:B_k}--\eqref{eq:psi_rev}.
\end{proposition}

Let us finally make precise the flow~$\varphi_{\Delta t}^{\FD}$ for the Hamiltonian function~\eqref{eq:H_RMHMC}. In this case, the diffusion process associated with the steps~\ref{step:ghmc_1} and~\ref{step:ghmc_5} of Algorithm~\ref{alg:ghmc_scheme_rev_check} reads
\begin{equation}
  \label{eq:Ornstein_Uhlenbeck_dynamics}
  \rmd p_t = -\gamma D(q^{n})p_t\,\rmd t+\sqrt[]{2\gamma}\,\rmd W_t.
\end{equation}
One can exactly integrate~\eqref{eq:Ornstein_Uhlenbeck_dynamics}, at least for matrices~$D(q^{n})$ for which~$\exp(-\gamma D(q^{n})\Delta t)$ is easy to compute (\emph{e.g.}~if~$D(q^{n})$ is diagonal). In a more general case, it is preferable to use a mid-point Euler scheme to update the momenta, which requires only solving a linear problem. Using a time increment~$\Delta t/2$, this scheme reads
\begin{equation*}
  p^{n+1/4}=p^{n}-\frac{\Delta t}{4}\gamma D\left(q^{n}\right)\left(p^{n}+p^{n+1/4}\right)+\sqrt{\gamma\Delta t}\,G^{n}.
\end{equation*}
This requires solving the linear system
\begin{equation*}
    \left[
        \rmI_{m} + \frac{\Delta t}{4}\gamma D\left(q^{n}\right)
    \right] p^{n+1/4}
    = \left[\rmI_{m}-\frac{\Delta t}{4}\gamma D\left(q^{n}\right)\right]p^{n}+\sqrt{\gamma\Delta t}\,G^{n},
\end{equation*}
which has a unique solution:
\begin{equation}
  \label{eq:OU_update_rmhmc}
    p^{n+1/4}= \left[
        \rmI_{m} + \frac{\Delta t}{4}\gamma D\left(q^{n}\right)
    \right]^{-1}\left[
        \left(\rmI_{m}-\frac{\Delta t}{4}\gamma D\left(q^{n}\right)\right)p^{n}+\sqrt{\gamma\Delta t}\,G^{n}
    \right].
\end{equation}

\subsubsection{Formal consistency of GHMC with Langevin dynamics} 
\label{subsubsec:consistency_GHMC_Langevin_dynamics}
We now formally discuss the weak consistency of the GHMC algorithm with the Langevin dynamics~\eqref{eq:langevin_dynamics} in the limit~$\Delta t\to0$. Such results have been formalized for separable Hamiltonian functions, see for instance~\cite{bou-rabee_2010}.

We first give the definition of the weak consistency order of a numerical discretization of a general diffusion process on a family of sets~$(B_{\Delta t})_{\Delta t>0}$ for a class of functions~$\scrF$.

\begin{definition}[Weak consistency order on a family of sets~$(B_{\Delta t})_{\Delta t>0}$ for a class of functions~$\scrF$]
  \label{def:weak_consistency_order}
  Let~$(X_t)_{t\geqslant 0}$ be a diffusion process with infinitesimal generator~$\calL$ whose domain includes~$\scrF$. Consider a numerical discretization with time step~$\Delta t$ of this diffusion process, which generates a Markov chain~$(X^n)_{n\geqslant 0}$ with transition operator~$P_{\Delta t}$. The numerical discretization is a weakly consistent discretization of order~$r>0$ of the diffusion process~$(X_t)_{t\geqslant0}$ for a class of functions~$\scrF$ if there exists~$\Delta t_{\star}>0$ and a family of sets~$(B_{\Delta t})_{0<\Delta t\leqslant\Delta t^{\star}}$ such that, for any function~$f\in\scrF$, there exists~$K\in\bbR_{+}$ for which
  \begin{equation}
    \label{eq:upper_bound_weak_consistency}
    \forall\Delta t\in(0,\Delta t^{\star}],\quad \forall x\in B_{\Delta t},\qquad 
    P_{\Delta t}f(x)= \rme^{\Delta t\calL}f(x)+\Delta t^{r+1}\kappa_{\Delta t}(x),\quad
    \left\lvert \kappa_{\Delta t}(x)\right\rvert\leqslant K.
  \end{equation}
\end{definition}
Notice that this definition also makes sense for deterministic processes (\emph{i.e.}~solutions of ordinary differential equations). Sometimes, to alleviate the notation, we simply write~$P_{\Delta t}f(x)= \rme^{\Delta t\calL}f(x)+\rmO(\Delta t^{r+1})$. Note also that it would be possible to consider more general definitions where the upper bound in~\eqref{eq:upper_bound_weak_consistency} is given by a function~$\calK(x)$ which has certain integrability properties.

Let us make the following assumptions: there is a class of functions~$\scrF$ such that
\begin{enumerate}[label={[H.\roman*]}, align=left]
  \item \label{hyp:ghmc_1} the map~$\varphi_{\Delta t}^{\FD}$ yields a weakly consistent discretization of the stochastic differential equation~\eqref{eq:Ornstein_Uhlenbeck_dynamics} of order~$r_1\geqslant0$ on~$\bbR^{m}$ for~$\scrF$;
  \item \label{hyp:ghmc_2} the map~$\varphi_{\Delta t}$ yields a consistent discretization of the ordinary differential equation~\eqref{eq:hamiltonian_dynamics} of order~$r_2\geqslant0$ on~$\calA_{\Delta t}$ for~$\scrF$.
\end{enumerate}
Moreover,
\begin{enumerate}[label={[H.\roman*]}, resume, align=left]
  \item \label{hyp:ghmc_3} The initial distribution~$\mu^{0}$ is such that~$\mu^{0}(\calX\setminus\calB_{\Delta t})\leqslant C\Delta t^{r_3+1}$ for a constant~$C\geqslant0$.
\end{enumerate}
Define~$\alpha=\min(2,r_1,r_2,r_3)$. Under assumptions~\ref{hyp:ghmc_1} to~\ref{hyp:ghmc_3}, GHMC is formally a weakly consistent discretization of order~$\alpha$ of the Langevin dynamics~\eqref{eq:langevin_dynamics}. This typically implies that the following one-step weak estimate holds, for some constant~$C_f\in\bbR_{+}$:
\begin{equation}
  \label{eq:one_step_horizon_weak_estimate}
  \left\lvert\bbE^{\mu^{0}}\left[ f(q^{1},p^{1})\right]-\bbE^{\mu^{0}}\left[f(q_{\Delta t},p_{\Delta t})\right]\right\rvert
  \leqslant C_{f}\Delta t^{\alpha+1},
\end{equation}
where~$(q^{1},p^{1})$ is the output of Algorithm~\ref{alg:ghmc_scheme_rev_check} using a time step~$\Delta t$ starting from~$(q^{0},p^{0})\sim\mu^{0}$ and~$(q_{\Delta t},p_{\Delta t})$ is the exact solution of~\eqref{eq:langevin_dynamics} at time~$\Delta t$ starting from~$(q_{0},p_{0})\sim\mu^{0}$. The constant~$C_{f}$ a priori depends on~$f$ in~\eqref{eq:one_step_horizon_weak_estimate} but may be uniform in~$f$ for some class of functions~$\scrF$ (\emph{e.g.} when~$\scrF$ is a set of functions with uniformly bounded derivatives). Using standard techniques such the ones described in~\cite[Chapter 2]{milstein_2004} or~\cite{talay_1990,kopec_2015,stoltz_2018}, the estimate~\eqref{eq:one_step_horizon_weak_estimate} can be extended to a finite time horizon estimate, as
\begin{equation*}
  \forall n\in\left\lbrace 1,\dots,\left\lfloor\frac{T}{\Delta t}\right\rfloor\right\rbrace,\qquad
  \left\lvert\bbE^{\mu^{0}}\left[ f(q^{n},p^{n})\right]-\bbE^{\mu^{0}}\left[f(q_{n\Delta t},p_{n\Delta t})\right]\right\rvert
  \leqslant C_{f,T}\Delta t^{\alpha},
\end{equation*}
where~$T>0$ is a fixed time horizon and~$C_{f,T}\in\bbR_{+}$, upon extending assumption~\ref{hyp:ghmc_3} as
\begin{equation*}
  \exists C_T\geqslant 0,\quad
  \forall n\in\left\lbrace 0,\ldots,\left\lfloor\frac{T}{\Delta t}\right\rfloor-1\right\rbrace,\quad\mu^{n}\left(\calX\setminus\calB_{\Delta t}\right)\leqslant C_T\Delta t^{r_3+1},
\end{equation*}
where~$\mu^{n}$ is the law of~$(q^{n},p^{n})$. One way to establish such an inequality is to prove a pointwise estimate similar to~\eqref{eq:rmhmc_rejection_probability} below, and to control some moments of~$\mu^{n}$ to integrate the pointwise bound (with an upper bound~$\rmO\left(\Delta t^{r_3+1}\right)$).

Let us now comment on the consistency result~\eqref{eq:one_step_horizon_weak_estimate}.
\begin{itemize}
  \item The momentum reversal step~\ref{step:ghmc_4} is needed for the result to hold (see for instance~\cite[Section~2.2.3.2]{lelievre_2010}). Upon acceptance of the proposal, the momenta are not reversed, so that a consistent one-step discretization of the Hamiltonian dynamics is obtained.
  \item The upper bound~$\alpha\leqslant 2$ comes from the Strang splitting underlying GHMC: one cannot generally expect more than a second order weak consistency according to the BCH formula (see for instance~\cite[Section~III.4]{hairer_2006} as well as~\cite{leimkuhler_2015} for the application to Langevin dynamics).
  \item The formalization of~\eqref{eq:one_step_horizon_weak_estimate} for a general Hamiltonian function requires to identify a class of functions~$\scrF$ in Definition~\ref{def:weak_consistency_order}. Typical classes of functions used to prove similar results in the literature are for instance the set of smooth bounded functions with uniformly bounded derivatives or the set of functions growing at most polynomially with derivatives growing at most polynomially. The class~$\scrF$ should be stable by the discrete and continuous transition operators of the dynamics~\eqref{eq:langevin_dynamics} and~\eqref{eq:Ornstein_Uhlenbeck_dynamics}. Establishing such stability properties is out of the scope of this work, and we refer for example to~\cite{stoltz_2018,monmarche_2022} for estimates allowing to make precise the functional setting for separable Hamiltonian functions.
\end{itemize}

Let us finally discuss the relevance of the hypotheses~\ref{hyp:ghmc_1}--\ref{hyp:ghmc_3}:
\begin{itemize}
  \item Item~\ref{hyp:ghmc_1} formally holds when using MALA (in which case~$r_1=1$). For specific Hamiltonian functions such as~\eqref{eq:H_RMHMC}, the midpoint scheme~\eqref{eq:OU_update_rmhmc} yields a second-order consistent discretization (\emph{i.e.}~$r_1=2$, see for example~\cite[Lemma~3]{fathi_2017}). In order to make precise statements, one should identify sufficient assumptions on the Hamiltonian function~$H$ for remainder terms to be controlled.
  \item Item~\ref{hyp:ghmc_2} holds for the IMR or the GSV numerical schemes with~$r_2=2$ (see~\cite[Theorems~VI.3.4 and~VI.3.5]{hairer_2006}). This implies that the rejection probability due to the Metropolis--Hastings acceptance/rejection procedure (step~\ref{step:ghmc_3}) scales as~$\rmO\left(\Delta t^{r_2+1}\right)$ since the Hamiltonian is conserved approximately at order~$\rmO\left(\Delta t^{r_2+1}\right)$ over one time step.
  \item We observe in the numerical simulations we performed that assumption~\ref{hyp:ghmc_3} holds when using the Hamiltonian function~\eqref{eq:H_RMHMC} (see Figure~\ref{fig:rejection_probabilities} below). 
\end{itemize}
The combination of these considerations suggests that, when using a second order scheme to discretize~\eqref{eq:Ornstein_Uhlenbeck_dynamics} ($r_1=2$) and the IMR or the GSV numerical scheme to numerically integrate~\eqref{eq:hamiltonian_dynamics}, GHMC is a discretization of weak order~$2$ of the Langevin dynamics on~$\calB_{\Delta t}$ for the class of functions~$\scrF$.

Such consistency results can be made rigorous, as examplified in the next section for RMHMC.

%---------------------------------------------------------
%---------------------------------------------------------
%---------------- RMHMC ----------------------------------
%---------------------------------------------------------
%---------------------------------------------------------

\subsection{RMHMC as a discretization of the overdamped Langevin dynamics with position-dependent diffusion coefficient}
\label{subsec:RMHMC_overdamped_Langevin}
The overdamped Langevin dynamics with a position-dependent diffusion coefficient is the process
\begin{equation*}
  \label{eq:overdamped_langevin_diffusion}
  \rmd q_t = \left(
      -D(q_t)\nabla V(q_t)+\div D(q_t)
      \right)\rmd t + \sqrt[]{2D(q_t)}~\rmd W_t,
\end{equation*}
where~$(W_t)_{t\geqslant0}$ is a Brownian motion,~$V$ a potential energy function and~$D$ a diffusion coefficient (\emph{i.e.}~a map whose output is a positive definite symmetric matrix). The divergence of the matrix valued function~$D$ is the vector field whose~$i$-th component is the divergence of the~$i$-th column of~$D$:
\begin{equation*}
  \left(\div D\right)_{i}=\sum_{j=1}^{m}\partial_{q_j}D_{i,j}. 
\end{equation*}
Denote by~$\calL$ the generator of the dynamics~\eqref{eq:overdamped_langevin_diffusion}. It acts on test functions~$f$ as
\begin{equation}
  \label{eq:generator_overdamped_langevin}
  \calL f=\left(-D\nabla V+\div D\right)^{\sfT}\nabla f+D\colon\nabla^{2} f.
\end{equation}
A straightforward computation shows that, for any smooth test functions~$f,g$,
\begin{equation*}
  \left\langle \calL f, g\right\rangle_{L^{2}(\mu)}=-\int_{\bbR^{m}}\left(\nabla f\right)^{\sfT}D\nabla g\,\rmd\mu.
\end{equation*}
Choosing~$g=1$ shows that~$\int_{\bbR^{m}}\calL f\rmd\mu=0$, so that the process~\eqref{eq:overdamped_langevin_diffusion} admits~\eqref{eq:invariant_measure_position_RMHMC} as an invariant probability measure, whatever the diffusion coefficient~$D$ (provided that the dynamics~\eqref{eq:overdamped_langevin_diffusion} is well-defined). These dynamics (along with a Metropolis--Hastings procedure) are used for example in the molecular dynamics literature to simulate ionic solutions and estimate macroscopic properties~\cite{jardat_1999}. In this situation, the diffusion coefficient~$D$ encodes some physical information corresponding to degrees of freedom which have not been retained in a coarse graining procedure. The diffusion coefficient~$D$ can also be added for numerical reasons to accelerate convergence towards equilibrium and to reduce the asymptotic variance of MCMC estimators~\cite{roberts_2002,rey-bellet_2016,abdulle_2019,lelievre_2023}.

A natural and consistent discretization of the process~\eqref{eq:overdamped_langevin_diffusion} is provided by the Euler--Maruyama scheme with time step~$h$, whose bias is corrected by a Metropolis--Hastings procedure. When the diffusion coefficient~$D$ is constant (\emph{i.e.}~it does not depend on the position~$q_t$), the rejection probability of the Metropolis--Hastings acceptance/rejection step scales as~$\rmO(h^{3/2})$ in the limit~$h\to0$. However, this rejection probability scales as~$\rmO(h^{1/2})$ when~$D$ is not constant (see~\cite[Lemma~4]{fathi_2017}). This high rejection probability constrains the time step to be smaller, which strongly reduces, or even cancels, the benefits from using a non-constant diffusion coefficient. 

We show in this section that, under natural assumptions, using the Hamiltonian function~\eqref{eq:H_RMHMC} in Algorithm~\ref{alg:hmc_scheme_rev_check} with either the IMR or GSV numerical scheme yields a weakly consistent discretization of weak order 1 of~\eqref{eq:overdamped_langevin_diffusion} with an effective time step~$h=\Delta t^{2}/2$. A key ingredient is that the Metropolis--Hastings rejection probability scales as~$\rmO(h^{3/2})$ for such numerical schemes even for non-constant diffusion coefficients (since the rejection probability for IMR or GSV is expected to be of order $\rmO(\Delta t^3)$, for instance following similar computations as in the proof of~\cite[Lemma 3.1]{stoltz_2018}). The objective of this section is to provide one possible setting of sufficient conditions under which this weak order can be obtained. The results heavily rely on the fact that when using the Hamiltonian function~\eqref{eq:H_RMHMC}, the conditional probability distribution of the momenta knowing the position~$q^{n}$ appearing in step~\ref{step:hmc_1} is simply the Gaussian distribution with zero mean and covariance matrix~$D(q^{n})^{-1}$, see~\cite{girolami_2011}.

In order to state the result, we make the following assumption.
\begin{assumption}
  \label{ass:RMHMC}
  The Hamiltonian function~$H$ is defined by~\eqref{eq:H_RMHMC}, where~$V$ and~$D$ are~$\calC^{\infty}$ functions. The map~$D$ has values in the space of symmetric positive definite matrices and is uniformly bounded from above and below: there exist~$0<a\leqslant b<+\infty$ such that
  \begin{equation}
    \label{eq:ass_D_bound}
    \forall q\in\calO,\qquad
    a\rmI_m\leqslant D(q)\leqslant b\rmI_m.  
  \end{equation}
  The map~$V$ is uniformly lower bounded: there exists~$c\in\bbR$ such that
  \begin{equation*}
    \forall q\in\calO,\qquad
    V(q)\geqslant c.
  \end{equation*}
  Furthermore, the map~$V$ and all its derivatives, as well as all the derivatives of the map~$D$, grow at most polynomially.

  The map~$\Phi_{\Delta t}$ is either the IMR or the GSV numerical scheme. The associated numerical solver~$\chi_{\Delta t}$ returns a solution~$y=(y_1,\dots,y_k)$ close to the initial point~$(q^{n},p^{n})\in\calA_{\Delta t}$ of the equation~$\Phi_{\Delta t}((q^{n},p^{n}),y_1,\ldots,y_k)=0$ in the following sense: there exist~$\Delta t_{\star}>0$ and~$C\in\bbR_{+}$ (which depend on~$(q^{n},p^{n})$) such that
  \begin{equation}
    \label{eq:solution_close}
    \forall 1\leqslant i\leqslant k,\quad \forall\Delta t\in(0,\Delta t_{\star}],\qquad \left\lVert y_{i}-(q^{n},p^{n})\right\rVert\leqslant C\Delta t. 
  \end{equation}
\end{assumption}
The last assumption, whose relevance is discussed more thoroughly in Section~\ref{subsec:numerical_solvers_fixed_point}, is used to obtain estimates on~$\left\lVert \widetilde{q}^{n+1}-q^{n}\right\rVert$ (where~$\widetilde{q}^{n+1}$ is the proposal in step~\ref{step:hmc_2}) using the implicit function theorem, as done for instance in~\cite{hairer_2006}.

We give in Lemma~\ref{lem:one_step_hmc_weak_consistency_overdamped_langevin} an expansion in powers of the time step of the output of a numerical flow associated either with the IMR or GSV numerical schemes. The proof is postponed to Section~\ref{subsec:proof:HMC_weak_consistency}.

\begin{lemma}
  \label{lem:one_step_hmc_weak_consistency_overdamped_langevin}
  Suppose that Assumption~\ref{ass:RMHMC} holds. Fix~$\Delta t>0$. Let~$(q^{n},p^{n})\in\calA_{\Delta t}$ and denote by~$(\widetilde{q}^{n+1},\widetilde{p}^{n+1})=\varphi_{\Delta t}(q^{n},p^{n})$. Then the following expansion holds:
  \begin{equation}
    \label{eq:expansion_position_rmhmc}
      \widetilde{q}^{n+1}= q^{n} + \Delta t D(q^{n})p^{n} + \dfrac{\Delta t^{2}}{2}\calF_{2}(q^{n}, p^{n})+\Delta t^{3}\calF_{3}(q^{n},p^{n})+\Delta t^{4}\calR_{\Delta t}(q^{n},p^{n}),
  \end{equation}
  for some functions~$\calF_2$,~$\calF_3$ and~$\calR_{\Delta t}$ that satisfy:
  \begin{itemize}
    \item There exist~$\Delta t_{\star}>0$ and~$K\in\bbR_{+}$ (which depend on~$(q^{n},p^{n})$) such that
    \begin{equation}
      \label{eq:expansion_position_rmhmc_bounds}
      \forall\Delta t\in(0,\Delta t_{\star}],\qquad
      \left\lVert\calR_{\Delta t}(q^{n},p^{n})\right\rVert\leqslant K.
    \end{equation}
    \item For~$\rmG^{n}$ a random variable following a standard normal distribution, it holds
    \begin{equation}
      \label{eq:computation_f2}
        \bbE_{\rmG^{n}}\left[\calF_2(q^{n},D(q^{n})^{-1/2}\rmG^{n})\right]=-D(q^{n})\nabla V(q^{n})+\div D(q^{n}),
    \end{equation}
    and~$\bbE_{\rmG^{n}}\left[\calF_3(q^{n},D(q^{n})^{-1/2}\rmG^{n})\right]=0$, where the subscript~$\rmG^{n}$ indicates that the expectation is over~$\rmG^{n}$.
  \end{itemize}
\end{lemma}

The following proposition states that HMC is a weakly consistent discretization of~\eqref{eq:overdamped_langevin_diffusion} under the assumption that the probability to reject the proposal due to~$S$-reversibility checks is controlled in the limit~$\Delta t\to0$. The proof can be read in Section~\ref{subsec:proof:HMC_weakly_consistent}. We denote by~$\scrF_{\pol}$ the set of smooth functions of the position variable~$q\in\calO$ that grow at most polynomially, with derivatives that grow at most polynomially.
\begin{proposition}
  \label{prop:one_step_hmc_weakly_consistent_overdamped_langevin}
  Suppose that Assumption~\ref{ass:RMHMC} holds and fix~$f\in\scrF_{\pol}$. Assume that there exist~$\Delta t_{\star}>0$ and~$K,\delta\in\bbR_{+},$ such that
  \begin{equation}
    \label{eq:stability_continous_process}
    \forall\Delta t\in(0,\Delta t_{\star}],\quad
    \forall q\in\calO,\qquad
    \left\lvert \left(
      \rme^{\Delta t\calL}f-(\id+\Delta t\calL)f
    \right)(q)\right\rvert
    \leqslant K\Delta t^{2}\left(1+\left\lVert q\right\rVert^{\delta}\right),
  \end{equation}
  and
  \begin{equation}
    \label{eq:remainder_term_uniform_bound}
    \forall\Delta t\in(0,\Delta t_{\star}],\quad
    \forall(q,p)\in\calA_{\Delta t},\qquad
    \left\lVert \varphi_{\Delta t}(q,p)-\phi_{\Delta t}(q,p)\right\rVert\leqslant K\Delta t^{3}\left(1+\left\lVert q\right\rVert^{\delta}+\left\lVert p\right\rVert^{\delta}\right),
  \end{equation}
  where~$\phi_{\Delta t}$ is the flow of the continuous Hamiltonian dynamics~\eqref{eq:hamiltonian_dynamics}.
  Assume furthermore that there exists~$r_3> 0$ such that, for all~$a\in\bbN$, there are~$C_{a},b_{a}\in\bbR_{+}$ for which
  \begin{equation}
    \label{eq:rmhmc_rejection_probability}
    \forall \Delta t\in(0,\Delta t_{\star}],\quad
    \forall q\in\calO,
    \qquad
    \bbE_{\rmG}\left[
      \left\lVert \rmG\right\rVert^{a}\mathbb{1}_{\left(q,D(q)^{-1/2}\rmG\right)\in\calX\setminus\calB_{\Delta t}}
    \right]\leqslant 
    C_a\left(1+\left\lVert q\right\rVert^{b_a}\right)\Delta t^{2(r_3+1)},
  \end{equation}
  where the set~$\calB_{\Delta t}$ is defined by~\eqref{eq:B_k} and~$\rmG\sim\calN(0,\rmI_m)$. Then the following one step weak estimate holds for~$f\in\scrF_{\pol}$: there exists~$C\in\bbR_{+}$ such that
  \begin{equation}
    \label{eq:rmhmc_weak_estimate}
    \forall \Delta t\in(0,\Delta t_{\star}],\quad\forall q^{0}\in\calO,\qquad
    \left\lvert
    \bbE^{q^{0}}\left[
       f(q^{1})
    \right]-\bbE^{q^{0}}\left[f(q_{\Delta t^{2}/2})\right]\right\rvert
    \leqslant C\left(1+\left\lVert q^{0}\right\rVert^{\alpha}\right)\Delta t^{2(\min(1,r_3)+1)},
  \end{equation}
  where~$q^{1}$ is the position after one step of Algorithm~\ref{alg:hmc_scheme_rev_check} starting from~$(q^{0},D(q^{0})^{-1/2}G^{0})$ with~$G^{0}\sim\calN(0,\rmI_m)$, and~$q_{\Delta t^{2}/2}$ is the exact solution of the overdamped Langevin dynamics~\eqref{eq:overdamped_langevin_diffusion} at time~$\Delta t^{2}/2$ starting from the initial condition~$q_{0}=q^{0}$.
\end{proposition}

Defining~$h=\Delta t^{2}/2$, Proposition~\ref{prop:one_step_hmc_weakly_consistent_overdamped_langevin} states that the one-step RMHMC is a weakly consistent discretization of the overdamped Langevin dynamics on~$\calO$ for the class of functions~$\scrF_{\pol}$, with effective time step~$h$ and order~$\min(1,r_3)$. In order to get this result, note that we have formulated the bound~\eqref{eq:rmhmc_rejection_probability} directly in terms of the effective time step, \emph{i.e.}~in powers of~$\Delta t^{2}$. As discussed after~\eqref{eq:one_step_horizon_weak_estimate}, one can extend these results to obtain weak estimates over finite time horizons.

The error estimate~\eqref{eq:stability_continous_process} is obtained by proving that~$\scrF_{\pol}$ is stable by~$\calL$ and~$\rme^{t\calL}$, which is typically done by proving moment estimates of the solution of~\eqref{eq:overdamped_langevin_diffusion}. Assumption~\eqref{eq:remainder_term_uniform_bound} can be obtained by performing Taylor expansions of~$\varphi_{\Delta t}$ and~$\phi_{\Delta t}$ in powers of~$\Delta t$: first, the factor~$\Delta t^{3}$ follows from the fact that the IMR or the GSV numerical scheme is a consistent discretization of~\eqref{eq:hamiltonian_dynamics} of order~2 (see for instance~\cite{hairer_2006}, and~\cite[Lemma 3.1]{stoltz_2018} for precise computations in the specific case of the GSV numerical scheme for separable Hamiltonian functions, similar to the ones performed in Lemma~\ref{lem:one_step_hmc_weak_consistency_overdamped_langevin}); second, the polynomial growth in the remainder term can be obtained using the implicit function theorem (as in~\cite{hairer_2006}) coupled with moment estimates on the numerical solution (see~\cite{kopec_2015} for the specific case when~$D(q)=\rmI_m$ for any~$q\in\calO$). 

Let us finally comment on the validity of assumption~\eqref{eq:rmhmc_rejection_probability}. Upon integrating in~$q$ with respect to~$\pi$ the inequality~\eqref{eq:rmhmc_rejection_probability} for~$a=0$ (assuming that this probability measure admits a moment of order~$b_0$),~\eqref{eq:rmhmc_rejection_probability} implies that~$\mu(\calX\setminus\calB_{\Delta t})=\rmO\left(\Delta t^{2(r_3+1)}\right)$. This bound will be numerically investigated in Section~\ref{sec:numerical_results}. In fact, the convergence towards~0 seems to be faster than any polynomial function of the time step, see Figure~\ref{fig:rejection_probabilities} below. Thus, in practice, it is observed that~$r_3\geqslant 1$ so that RMHMC yields a discretization of~\eqref{eq:overdamped_langevin_diffusion} of order~$1$ as announced at the beginning of this section.

%---------------------------------------------------------
%---------------------------------------------------------
%---------------- GHMALA ---------------------------------
%---------------------------------------------------------
%---------------------------------------------------------
\subsection{A fully unbiased version of Generalized Hybrid Metropolis--Adjusted Langevin Algorithm}
\label{subsec:ghmala}

In the context of the overdamped Langevin dynamics, considering nonreversible drifts has been shown to improve the convergence towards the Gibbs measure~$\pi$ defined in~\eqref{eq:invariant_measure_position_RMHMC} (in terms of the spectral gap of the generator of the dynamics), and to lead to estimators with lower variance~\cite{rey-bellet_2015,rey-bellet_2015_ii,duncan_2016,zhang_2022}. These nonreversible dynamics are typically of the form
\begin{equation}
  \label{eq:OL_dynamics_nonreversible}
  \rmd q_t=-\nabla V(q_t)\,\rmd t+\gamma(q_t)\,\rmd t+\sqrt[]{2}\,\rmd W_t,
\end{equation}
where~$V$ is a~$\calC^{1}$ potential energy function,~$\gamma$ is any smooth vector field satisfying~$\div (\gamma\pi)=0$, and~$(W_t)_{t\geqslant 0}$ is a standard~$m$-dimensional Brownian motion. To correct for the bias when discretizing these dynamics, for instance with the Euler--Maruyama scheme, one cannot directly use a Metropolis--Hastings acceptance/rejection step on a discretization of~\eqref{eq:OL_dynamics_nonreversible} since this would cancel the nonreversible term~$\gamma(q_t)$, negating the interest of adding this nonreversible part~\cite{duncan_2016}. The Generalized Hybrid Metropolis--Adjusted Langevin Algorithm (GHMALA) introduced in~\cite{poncet_2017} aims at Metropolizing the nonreversible dynamics~\eqref{eq:OL_dynamics_nonreversible} while preserving the nice properties they enjoy. In particular, it performs a bias-free sampling of the target distribution for sufficiently small time steps using a discretization of~\eqref{eq:OL_dynamics_nonreversible} (see~\cite[Lemma~3.2]{poncet_2017}). In this section, we adapt the GHMALA algorithm by implementing~$S$-reversibility checks in order to perform unbiased sampling \emph{whatever the time step} (see Algorithm~\ref{alg:GHMALA_rev_check}) under additional assumptions on~$\gamma$, see Proposition~\ref{prop:GHMALA_pi_invariance}. 

The GHMALA algorithm considers discretizations of the process~\eqref{eq:OL_dynamics_nonreversible} where the non-gradient term can change sign, which allows to ensure some form of reversibility. The discretization is based on a splitting procedure. First, the reversible dynamics
\begin{equation}
  \label{eq:PONCET_reversible_part}
  \rmd q_t=-\nabla V(q_t)\,\rmd t+\sqrt{2}\,\rmd W_t  
\end{equation}
is integrated using the MALA algorithm. Then, the nonreversible dynamics is rewritten as
\begin{equation}
  \label{eq:PONCET_nonreversible_part}
  \rmd q_t=\xi\gamma(q_t)\,\rmd t,  
\end{equation}
where~$\xi\in\bbR$ specifies the direction and the intensity of the non reversibility. The elementary dynamics~\eqref{eq:PONCET_nonreversible_part} is then integrated using a one-step HMC method in an augmented space of configurations~$(q,\xi)$. The sign of~$\xi$ can be reversed depending on the outcome of a Metropolis--Hastings acceptance/rejection step. While the integration of~\eqref{eq:PONCET_reversible_part} with MALA is explicit, the proposal move for~\eqref{eq:PONCET_nonreversible_part} is constructed using an implicit numerical integrator (IMR for GHMALA) to ensure some form of reversibility. We make precise in Definition~\ref{def:GHMALA} this second step. We denote by~$\calX=\bbR^{m}\times\bbR$ the augmented space of dimension~$d=m+1$.

\begin{definition}[IMR for Generalized Hybrid MALA]
    \label{def:GHMALA}
    The proposal move for~\eqref{eq:PONCET_nonreversible_part} is
    \begin{equation*}
        \left\lbrace
        \begin{aligned}
            q^{n+1}&=q^{n}+\Delta t\xi^{n}\gamma\left(\frac{q^{n}+q^{n+1}}{2}\right),\\
            \xi^{n+1}&=\xi^{n}.
        \end{aligned}
        \right.
    \end{equation*}
    In that case,~$k=1$ and
    \begin{equation}
    \label{eq:Phi_GHMALA}
        \Phi_{\Delta t}^{\GHMALA}(q,\xi,q_1,\xi_1)=
        \begin{pmatrix}
            q_1-q-\Delta t\xi\gamma\left(\dfrac{q+q_1}{2}\right)\\
            \xi_1-\xi
        \end{pmatrix}.
    \end{equation}
\end{definition}

The following two propositions are the counterparts of Propositions~\ref{prop:numerical_flow_symplectic} and~\ref{prop:IMR_GSV_S_reversibility} for the IMR numerical scheme in this non-Hamiltonian setting. Denote by~$S$ the direction reversal map defined by
\begin{equation}
  \label{eq:S_direction_reversal_Poncet}
  \forall(q,\xi)\in\calX,\qquad S(q,\xi)=(q,-\xi).
\end{equation}

\begin{proposition}
  \label{prop:GHMALA_s_rev}
  Fix~$\Delta t>0$ and let~$\Phi_{\Delta t}^{\GHMALA}$ be the numerical scheme defined by~\eqref{eq:Phi_GHMALA}. Let~$S$ be the direction reversal map defined by~\eqref{eq:S_direction_reversal_Poncet}. Then~$\Phi_{\Delta t}^{\GHMALA}$ is~$S$-reversible. 
\end{proposition}

\begin{proof}
    Let~$(q,\xi,q_1,\xi_{1})\in\calX^{2}$ such that~$\Phi_{\Delta t}^{\GHMALA}(q,\xi,q_1,\xi_{1})=0$. Then 
  \begin{align*}
    \Phi_{\Delta t}^{\GHMALA}(S(q_1,\xi_1),S(q,\xi))
    &=
    \Phi_{\Delta t}^{\GHMALA}(q_1,-\xi_1,q,-\xi),\\
    &=
    \begin{pmatrix}
      q-q_1-\Delta t(-\xi_1) \gamma\left(\dfrac{q+q_1}{2}\right)\\
      -\xi+\xi_1
    \end{pmatrix}=0,
  \end{align*}
  from which the conclusion follows.
\end{proof}

\begin{proposition}
  \label{prop:GHMALA_symplectic_num_flow}
  Fix~$\Delta t>0$. Let~$\gamma$ be a~$\calC^{1}$ function which satisfies
  \begin{equation}
    \label{eq:det_relation}
    \forall\theta\in\bbR,\quad
    \forall q\in\bbR^{m},\qquad
    \det(\rmI_m-\theta\nabla\gamma(q))=\det(\rmI_m+\theta\nabla\gamma(q)).
  \end{equation}
  Assume that~$\chi_{\Delta t}^{\GHMALA}\colon\calA^{\GHMALA}_{\Delta t}\to\calX$ is a numerical solver for the GHMALA numerical scheme. Then, for all~$(q,\xi)\in\calA_{\Delta t}^{\GHMALA}$, the numerical flow~$\varphi_{\Delta t}^{\GHMALA}$ satisfies
  \begin{equation*}
    \left\lvert\det\nabla\varphi_{\Delta t}^{\GHMALA}(q,\xi)\right\rvert=1.
  \end{equation*}
\end{proposition}

The relation~\eqref{eq:det_relation} is satisfied for instance when the matrix~$\nabla\gamma$ is antisymmetric (since the determinant of a matrix and its transpose coincide), or when~$\gamma=\calJ\nabla V$ with~$\calJ$ an antisymmetric matrix. The latter statement is a consequence of the fact that~$\det(\rmI_m+B^{\sfT}A^{\sfT})=\det(\rmI_m+AB)=\det(\rmI_m+BA)$ for any square matrices~$A,B$.

\begin{proof}
  Using the definition of the GHMALA numerical scheme~\eqref{eq:Phi_GHMALA}, we readily compute
  \begin{equation}
    \label{eq:nabla_GHMALA}
    \forall(x,x_1)\in\calX^{2},\qquad
    \left\lbrace
    \begin{aligned}
      &\nabla_{x}\Phi_{\Delta t}^{\GHMALA}(x,x_1)=
      \begin{pmatrix}
        -\rmI_m-\dfrac{\Delta t}{2}\xi\nabla\gamma\left(\dfrac{q+q_1}{2}\right) & -\Delta t\gamma\left(\dfrac{q+q_1}{2}\right)\\[0.2cm]
        0_{1,m} & -1
      \end{pmatrix},\\  
      &\nabla_{y}\Phi_{\Delta t}^{\GHMALA}(x,x_1)=
      \begin{pmatrix}
        \rmI_m-\dfrac{\Delta t}{2}\xi\nabla\gamma\left(\dfrac{q+q_1}{2}\right) & 0_{m,1}\\[0.2cm]
        0_{1,m} & 1
      \end{pmatrix}.
    \end{aligned}
    \right. 
  \end{equation}
  Let~$x=(q,\xi)\in\calA_{\Delta t}^{\GHMALA}$ and denote by~$x_1=(q_1,\xi_1):=\varphi_{\Delta t}^{\GHMALA}(x)$. By definition of a numerical solver (see Definition~\ref{def:numerical_solver}), the matrix~$\nabla_{y}\Phi_{\Delta t}^{\GHMALA}\left(x,x_1\right)$ is invertible, which is equivalent to~$\rmI_m-\frac{\Delta t}{2}\xi\nabla\gamma\left(\frac{q+q_1}{2}\right)$ being an invertible matrix as seen by the second equality in~\eqref{eq:nabla_GHMALA}. Using Lemma~\ref{lem:numerical_solver_C1}, we can then compute the gradient of~$\varphi_{\Delta t}^{\GHMALA}$. Not explicitly indicating the argument~$\frac{q+q_1}{2}$ for the maps~$\gamma$ and~$\nabla\gamma$ for readability, it holds
  \begin{align*}
    \nabla\varphi_{\Delta t}^{\GHMALA}(x)
    &=\begin{pmatrix}
      \left(\rmI_m-\dfrac{\Delta t}{2}\xi\nabla\gamma\right)^{-1} & 0_{m,1}\\[0.2cm]
      0_{1,m} & 1
    \end{pmatrix}
    \begin{pmatrix}
      \rmI_m+\dfrac{\Delta t}{2}\xi\nabla\gamma & \Delta t\gamma\\[0.2cm]
      0_{1,m} & 1
    \end{pmatrix},\\[0.2cm]
    &=\begin{pmatrix}
      \left(\rmI_m-\dfrac{\Delta t}{2}\xi\nabla\gamma\right)^{-1}\left(\rmI_m+\dfrac{\Delta t}{2}\xi\nabla\gamma\right)& \Delta t\left(\rmI_m-\dfrac{\Delta t}{2}\xi\nabla\gamma\right)^{-1}\gamma\\[0.2cm]
      0_{1,m} & 1
    \end{pmatrix}.
  \end{align*}
  The condition~\eqref{eq:det_relation} immediately implies that~$\left\lvert\det\nabla\varphi_{\Delta t}^{\GHMALA}(x)\right\rvert=1$ for any~$x\in\calA_{\Delta t}^{\GHMALA}$. 
\end{proof}

We now describe a fully unbiased GHMALA algorithm. Let~$\chi_{\Delta t}^{\GHMALA}\colon\calA_{\Delta t}^{\GHMALA}\to\calX$ be a numerical solver associated with~$\Phi_{\Delta t}^{\GHMALA}$, and denote by~$\psi_{\Delta t}^{\GHMALA,\rev}$ the map defined by~\eqref{eq:B_k}--\eqref{eq:psi_rev} with~$\varphi_{\Delta t}=\varphi_{\Delta t}^{\GHMALA}$ and~$\calA_{\Delta t}=\calA_{\Delta t}^{\GHMALA}$ the numerical flow associated with~$\chi_{\Delta t}^{\GHMALA}$ (which is simply equal to~$\chi_{\Delta t}^{\GHMALA}$ in this particular case) and~$S$ the direction reversal map defined by~\eqref{eq:S_direction_reversal_Poncet}. The GHMALA algorithm with~$S$-reversibility checks is then given in Algorithm~\ref{alg:GHMALA_rev_check}.

\begin{algorithm}
  \caption{GHMALA algorithm with~$S$-reversibility checks}
  \label{alg:GHMALA_rev_check}
  Consider an initial condition~$(q^{0},\xi^{0})\in\calX$, and set~$n= 0$.
  \begin{enumerate}[label={[\thealgorithm.\roman*]}, align=left]
      \item Integrate the reversible part of the dynamics~\eqref{eq:PONCET_reversible_part} using the MALA algorithm on the position~$q^{n}$. The output is denoted by~$(q^{n+1/3},\xi^n)$.
      \item Integrate the nonreversible part of the dynamics~\eqref{eq:PONCET_nonreversible_part} using the map~$\psi_{\Delta t}^{\GHMALA,\rev}$:
      \begin{equation*}
          (q^{n+2/3}, -\xi^n)=\psi_{\Delta t}^{\GHMALA,\rev}(q^{n+1/3},\xi^n);
      \end{equation*}
      \item Draw a random variable~$U^{n}$ with uniform law on~$(0,1)$:
      \begin{itemize}[label=$\bullet$]
          \item if~$U^{n}\leqslant \exp(-V(q^{n+2/3})+V(q^{n+1/3}))$, accept the proposal and set~$(q^{n+1},\xi^{n+1})=(q^{n+2/3},\xi^n)$;
          \item else reject the proposal and set~$(q^{n+1},\xi^{n+1})=(q^{n+1/3},-\xi^{n})$.
      \end{itemize}
      \item Increment~$n$ and go back to~\ref{step:ghmc_1}.
  \end{enumerate}
\end{algorithm}

The following proposition shows that Algorithm~\ref{alg:GHMALA_rev_check} is unbiased whatever the time step~$\Delta t$, by which we mean that there is an invariant probability measure whose marginal in position is~$\pi$.
\begin{proposition}
  \label{prop:GHMALA_pi_invariance}
  Assume that~$\gamma$ is a~$\calC^{1}$ function which satisfies~\eqref{eq:det_relation}. Then the measure~$\pi\otimes(\delta_{-\xi^0}/2+\delta_{\xi^0}/2)$ is an invariant measure for the Markov chain generated by Algorithm~\ref{alg:GHMALA_rev_check}.
\end{proposition}

\begin{proof}
  Denote by~$\pi_{\xi^0}(\rmd q\,\rmd \xi)=\pi(\rmd q)\otimes(\delta_{-\xi^0}(\rmd\xi)/2+\delta_{\xi^0}(\rmd\xi)/2)$. By Lemma~\ref{prop:GHMALA_s_rev}, the numerical scheme~$\Phi_{\Delta t}^{\GHMALA}$ is~$S$-reversible with~$S$ defined by~\eqref{eq:S_direction_reversal_Poncet}, which is a~$\calC^{1}$ involution. This shows that Assumption~\ref{ass:S_C1_involution} is satisfied. Furthermore, Proposition~\ref{prop:GHMALA_symplectic_num_flow} shows that the map~$\varphi_{\Delta t}^{\GHMALA}$ satisfies Assumption~\ref{ass:numerical_flow_symplectic}. The result then follows from~\cite[Section~2.1.4]{lelievre_2010} by adapting the computation of the Metropolis--Hastings ratio. Here, for a proposed move from~$(q,\xi)$ to~$(q',\xi')$, this ratio writes
  \begin{equation*}
    \frac{
      \delta_{
        \psi_{\Delta t}^{\GHMALA,\rev}(q',\xi')
        }
        (\rmd q\, \rmd\xi)
        \pi_{\xi^0}(\rmd q'\, \rmd\xi')
    }{
      \delta_{
        \psi_{\Delta t}^{\GHMALA,\rev}(q,\xi)
        }
        (\rmd q'\, \rmd\xi')
        \pi_{\xi^0}(\rmd q\, \rmd\xi)
    }
    =
    \exp\left(
      -V(q')+V(q)
    \right),
  \end{equation*}
  which is readily verified by following the same computations as in the proof of Proposition~\ref{prop:HMC_mu_invariance}.
\end{proof}

Theoretical constructions of the map~$\chi_{\Delta t}^{\GHMALA}$ can be performed using arguments similar to the ones used for the construction of numerical solvers for the IMR numerical scheme associated with the Hamiltonian dynamics, see Proposition~\ref{prop:IMR_FP} below. Practical implementations can be readily adapted from the ones presented in Section~\ref{subsec:practical_numerial_solver_newton_method} below.

%---------------------------------------------------------
%---------------------------------------------------------
%---------------------------------------------------------
%---------------------------------------------------------
%---------------------------------------------------------
%---------------- PRACTICAL NUM FLOWS --------------------
%---------------------------------------------------------
%---------------------------------------------------------
%---------------------------------------------------------
%---------------------------------------------------------
%---------------------------------------------------------
\section{Construction of practical numerical flows}
\label{sec:construction_practical_numerical_flows}

In practice, an iterative method is used to solve the implicit problem~$\Phi_{\Delta t}(x,y)=0$ for a fixed~$x\in\calX$ with~$y\in\calX^k$ (\emph{e.g.}~fixed point algorithm or Newton's method). Therefore, one needs to check that such a method yields a numerical flow in the sense of Definition~\ref{def:numerical_solver}. However, usual implemented methods typically do not fit into the framework described in Section~\ref{sec:numerical_flow_s_reversibility}, as they do not solve \textit{exactly}~$\Phi_{\Delta t}(x,y)=0$ since only a finite number of steps are used to find a solution with the iterative method at hand (and because of floating point arithmetic). We therefore present idealized versions of such methods, and show that they yield numerical solvers which satisfy Assumption~\ref{ass:numerical_flow_symplectic}, so that, combined with Assumption~\ref{ass:S_C1_involution}, unbiasedness sampling is guaranteed using the (G)HMC algorithms introduced in Section~\ref{subsec:hmc} and~\ref{subsec:ghmc}. In particular, we use the Banach fixed-point theorem in Section~\ref{subsec:numerical_solvers_fixed_point} to show that, for sufficiently small time steps, a numerical flow can be constructed for both the IMR and GSV numerical schemes, the set~$\calB_{\Delta t}$ defined in~\eqref{eq:B_k} being nonempty. We then focus on an idealized Newton's method in Section~\ref{subsec:numerical_solver_newton_method_ift}. A practical implementation of Newton's method is provided in Section~\ref{subsec:practical_numerial_solver_newton_method} (see Algorithm~\ref{alg:newton_practice}) along with a practical implementation of the associated numerical flows with~$S$-reversibility check~$\varphi_{\Delta t}^{\rev}$ (defined in~\eqref{eq:B_k}--\eqref{eq:psi_rev}) in Algorithm~\ref{alg:psi_dt_rev}.

%---------------------------------------------------------
%---------------------------------------------------------
%---------------- FIXED-POINT ----------------------------
%---------------------------------------------------------
%---------------------------------------------------------
\subsection{Numerical solvers using the Banach fixed-point theorem}
\label{subsec:numerical_solvers_fixed_point}

We construct in this section examples of numerical solvers~$\chi_{\Delta t}^{\FP}\colon\calA_{\Delta t}^{\FP}\to\calX^k$ for the two running examples we consider, based on the Banach fixed-point (FP) theorem, such that the set~$\calA_{\Delta t}^{\FP}$ is open and nonempty. The associated set~$\calB_{\Delta t}^{\FP}$, defined by~\eqref{eq:B_k} with~$\calA_{\Delta t}=\calA_{\Delta t}^{\FP}$, is then shown to be nonempty, so that the associated numerical flow with~$S$-reversibility checks and~$S$-reversal~$\psi_{\Delta t}^{\rev,\FP}$ as defined in~\eqref{eq:psi_rev} (with~$\calB_{\Delta t}=\calB_{\Delta t}^{\FP}$ and~$\varphi_{\Delta t}^{\rev}=\varphi_{\Delta t}^{\rev,\FP}$) is nontrivial.

We first show in Proposition~\ref{prop:IMR_FP} how to construct a numerical flow in the case~$k=1$ for the IMR numerical scheme (see Definition~\ref{def:IMR}). A similar result is then stated for the GSV numerical scheme (see Definition~\ref{def:GSV}) in Proposition~\ref{prop:GSV_FP}.

We make the following assumption.
\begin{assumption}
  \label{ass:H_C2_S_momentum_reversal}
 The Hamiltonian function~$H$ is~$\calC^{2}$ and even in the momentum variable, and~$S$ is the momentum reversal map~\eqref{eq:S_momentum_reversal}.
\end{assumption}

\begin{proposition}
  \label{prop:IMR_FP}
  Suppose that Assumption~\ref{ass:H_C2_S_momentum_reversal} holds. Then, there exists~$\Delta t_{\star}>0$ such that, for any~$\Delta t\in(0,\Delta t_{\star}]$, there are a nonempty open set~$\calA_{\Delta t}^{\IMR,\FP}\subset\calX$ and a~$\calC^{1}$ map~$\chi_{\Delta t}^{\IMR,\FP}\colon\calA_{\Delta t}^{\IMR,\FP}\to\calX$ such that
  \begin{equation}
    \label{eq:imr_fp}
    \forall x\in\calA_{\Delta t}^{\IMR,\FP},\quad \Phi_{\Delta t}^{\IMR}\left(x,\chi_{\Delta t}^{\IMR,\FP}(x)\right)=0,\quad\nabla_{y}\Phi_{\Delta t}^{\IMR}\left(x,\chi_{\Delta t}^{\IMR,\FP}(x)\right)\text{ is invertible.}
  \end{equation}
  Furthermore, defining~$\psi_{\Delta t}^{\IMR,\FP}=S\circ\chi_{\Delta t}^{\IMR,\FP}$, the set
  \begin{equation}
    \label{eq:B_IMR_FP}
    \calB_{\Delta t}^{\IMR,\FP}=\left\lbrace
      x\in\calA_{\Delta t}^{\IMR,\FP}\bigcap\left(\psi_{\Delta t}^{\IMR,\FP}\right)^{-1}\left(\calA_{\Delta t}^{\IMR,\FP}\right)
      ,\psi_{\Delta t}^{\IMR,\FP}\circ\psi_{\Delta t}^{\IMR,\FP}(x)=x
    \right\rbrace
  \end{equation}
  is open and nonempty.
\end{proposition}

In the proof, the construction of the numerical flow~$\chi_{\Delta t}^{\IMR,\FP}$ is performed so that it provides, for~$x\in\calA_{\Delta t}^{\IMR,\FP}$, the nearest configuration from~$x$ (for the norm on~$\calX$) that solves~$\Phi_{\Delta t}^{\IMR}(x,x')=0$.

\begin{proof}
  We first construct the numerical flow~$\chi_{\Delta t}^{\IMR,\FP}$ on a nonempty open set~$\calA_{\Delta t}^{\IMR,\FP}\subset\calX$. Let~$x=(q,p)\in\calX$. The configurations~$x_1\in\calX$ such that~$\Phi_{\Delta t}^{\IMR}(x,x_1)=0$ satisfy
  \begin{equation*}
    x_1=x+\Delta tJ\nabla H\left(\frac{x+x_1}{2}\right),
  \end{equation*}
  where the matrix~$J$ is defined in~\eqref{eq:hamiltonian_dynamics}. Fix~$R>0$ and denote by~$B_{x}^{R}:=\overline{\calB\left(x,R\right)}$ the closed ball centered at~$x\in\calX$ with radius~$R$. We denote by~$\left\lVert\cdot\right\rVert_{\calC^{0}(E)}$ the sup norm on the set~$E\subset\calX$. For~$\Delta t> 0$ such that~$\Delta t\left\lVert\nabla H\right\rVert_{\calC^0\left(B_x^{R/2}\right)}\leqslant R$, define
  \begin{equation*}
    \calF_{\Delta t}^{x}\colon\left\lbrace
    \begin{array}[]{ccl}
      B_x^R &\to&B_x^R \\
      x_1&\mapsto&x+\Delta t J\nabla H\left(\dfrac{x+x_1}{2}\right).
    \end{array}
    \right.
  \end{equation*}
  Note that, for all~$x_1\in B_x^R$,~$\Phi_{\Delta t}^{\IMR}(x,x_1)=0$ if and only if~$x_1$ is a fixed point of~$\calF_{\Delta t}^{x}$. A simple computation shows that~$\calF_{\Delta t}^{x}$ is a contraction when
  \begin{equation*}
    \Delta t\frac{\left\lVert\nabla^{2}H\right\rVert_{\calC^0\left(B_x^{R/2}\right)}}{2}< 1.
  \end{equation*}
  Define~$F\colon\calX\to(0,+\infty]$ as
  \begin{equation*}
    F(x)=
    \left\lbrace
    \begin{array}[]{ll}
      \min\left(\dfrac{R}{\left\lVert\nabla H\right\rVert_{\calC^0\left(B_x^{R/2}\right)}},\dfrac{2}{\left\lVert\nabla^{2}H\right\rVert_{\calC^0\left(B_x^{R/2}\right)}}\right)&\text{if }\left\lVert\nabla H\right\rVert_{\calC^0\left(B_x^{R/2}\right)}\left\lVert\nabla^{2}H\right\rVert_{\calC^0\left(B_x^{R/2}\right)}\neq0,\\
      +\infty&\text{otherwise}.
    \end{array}
    \right.
  \end{equation*}
  Note that~$F$ is continuous on~$\dom F=\left\lbrace x\in\calX,F(x)<+\infty\right\rbrace$. If~$\Delta t< F(x)$, there exists by the Banach fixed point theorem~\cite[Section~1.6]{zeidler_1995} a unique~$x_1\in B_x^R$ such that~$\Phi_{\Delta t}^{\IMR}(x,x_1)=0$. This configuration is denoted by~$x_1=:\chi_{\Delta t}^{\IMR,\FP}(x)$. This construction naturally defines the numerical flow~$\chi_{\Delta t}^{\IMR,\FP}$ on the open set 
  \begin{equation*}
    \calA_{\Delta t}^{\IMR,\FP}:=F^{-1}(\Delta t,+\infty)\cup \Int F^{-1}\left(\left\lbrace +\infty\right\rbrace\right),
  \end{equation*}
  where~$\Int A$ denotes the interior of a set~$A$. For~$\Delta t>0$ small enough, this set is nonempty. Indeed, if for all~$x\in\calX$,~$F(x)=+\infty$, then~$\calA_{\Delta t}^{\IMR,\FP}=\calX$. Otherwise, there exists~$x_0\in\calX$ such that~$F(x_0)$ is finite. In this case, we choose~$\Delta t_{\star}=F(x_0)/2>0$ (which will later be reduced but will still remain positive), and~$\calA_{\Delta t_{\star}}^{\IMR,\FP}$ is nonempty. Moreover, for any~$\Delta t\in(0,\Delta t_{\star}]$, it holds~$\calA_{\Delta t}^{\IMR,\FP}\subset\calA_{\Delta t_{\star}}^{\IMR,\FP}$.
  
  We next construct the set~$\calB_{\Delta t}^{\IMR,\FP}$ and show that it is nonempty for~$\Delta t$ sufficiently small. First, let us shows that~$\chi_{\Delta t}^{\IMR,\FP}$ is~$S$-reversible (see Definition~\ref{def:S_rev_numerical_scheme}) for a sufficiently small time step. Define~$G\colon\calX\to(0,+\infty]$ as
  \begin{equation*}
    G(x)=\min\left(F(x),\min\limits_{y\in B_x^R}\left\lbrace F(S(y))\right\rbrace\right).
  \end{equation*}
  Note that the value~$G(x)$ is positive (and can be infinite). Assume that~$\Delta t< G(x)$ for some~$x\in\calX$. Then~$x_1=\chi_{\Delta t}^{\IMR,\FP}(x)$ is well-defined. Moreover, since in this case~$\Delta t<F(S(x_1))$, there exists a unique~$x_1'\in B_{S(x_1)}^{R}$ such that~$\Phi_{\Delta t}^{\IMR}(S(x_1),x_1')=0$, which is~$x_1'=\chi_{\Delta t}^{\IMR,\FP}(S(x_1))$. Recall that the map~$\Phi_{\Delta t}^{\IMR}$ is~$S$-reversible by Proposition~\ref{prop:IMR_GSV_S_reversibility} so that~$\Phi_{\Delta t}^{\IMR}(x,x_1)=0$ if and only if~$\Phi_{\Delta t}^{\IMR}(S(x_1),S(x))=0$. To conclude that~$x_1'=S(x)$, it therefore suffices to show that~$S(x)\in B_{S(x_1)}^R$, in view of the uniqueness of the fixed point of~$\calF_{\Delta t}^{S(x_1)}$ in~$B_{S(x_1)}^{R}$. Since~$S$ is the momentum reversal map, it is an isometry, so that~$\left\lVert S(x)-S(x_1)\right\rVert=\left\lVert x-x_1\right\rVert\leqslant R$. Therefore, one obtains that~$x_1'=S(x)$, and it holds
  \begin{equation*}
    x=\left(S\circ\chi_{\Delta t}^{\IMR,\FP}\circ S \circ\chi_{\Delta t}^{\IMR,\FP}\right)(x).
  \end{equation*}
  Following the same reasoning as above, there exists~$\Delta t_{\star}>0$ , for any~$\Delta t\in(0,\Delta t_{\star}]$, the map~$\chi_{\Delta t}^{\IMR,\FP}$ is~$S$-reversible. Then the set~$\calB_{\Delta t}^{\IMR,\FP}$ defined by~\eqref{eq:B_IMR_FP} is open by Lemma~\ref{lem:B_open_set} and is nonempty as it contains~$G^{-1}(\Delta t,+\infty]$.

  To conclude, it remains to prove that~$\chi_{\Delta t}^{\IMR,\FP}\colon\calA_{\Delta t}^{\IMR,\FP}\to\calX$ is~$\calC^{1}$ for~$\Delta t\in(0,\Delta t_{\star}]$ and that the second statement of~\eqref{eq:imr_fp} holds. Note that for~$x\in\calA_{\Delta t}^{\IMR,\FP}$, one has~$\Phi_{\Delta t}^{\IMR}\left(x,\chi_{\Delta t}^{\IMR}(x)\right)=0$ and the matrix
  \begin{equation*}
    \nabla_{y}\Phi_{\Delta t}^{\IMR}\left(x,\chi_{\Delta t}^{\IMR,\FP}(x)\right)=\rmI_d-\frac{\Delta t}{2}J\nabla^{2}H\left(\frac{x+\chi_{\Delta t}^{\IMR,\FP}(x)}{2}\right)
  \end{equation*}
  is invertible since~$\Delta t<2/\left\lVert\nabla^{2}H\right\rVert_{\calC^{0}\left(B_x^{R/2}\right)}$. Therefore, the map~$\chi_{\Delta t}^{\IMR,\FP}$ is~$\calC^{1}$ as a result of the implicit function theorem. This concludes the proof.
\end{proof}

A similar result holds for the GSV numerical scheme. The result is stated here and proved for completeness in Section~\ref{subsec:proof_GSV_FP}. The main idea is to repeat the construction of contracting maps for each implicit step in the splitting scheme. Note that the construction provided in the proof can be readily adapted to other numerical schemes whatever the number of intermediate configurations.

\begin{proposition}
  \label{prop:GSV_FP}
  Suppose that Assumption~\ref{ass:H_C2_S_momentum_reversal} holds. Then, there exists~$\Delta t_{\star}>0$ such that, for any time step~$\Delta t\in(0,\Delta t_{\star}]$, there are a nonempty open set~$\calA_{\Delta t}^{\GSV,\FP}\subset\calX$ and a~$\calC^{1}$ map~$\chi_{\Delta t}^{\GSV,\FP}\colon\calA_{\Delta t}^{\GSV,\FP}\to\calX^{2}$ such that
  \begin{equation}
    \label{eq:gsv_fp}
    \forall x\in\calA_{\Delta t}^{\GSV,\FP},\quad \Phi_{\Delta t}^{\GSV}(x,\chi_{\Delta t}^{\GSV,\FP}(x))=0,\quad\nabla_{y}\Phi_{\Delta t}^{\GSV}(x,\chi_{\Delta t}^{\GSV,\FP}(x))\text{ is invertible.}
  \end{equation}
  Furthermore, defining~$\varphi_{\Delta t}^{\GSV,\FP}=\chi_{\Delta t,2}^{\GSV,\FP}$ and~$\psi_{\Delta t}^{\GSV,\FP}=S\circ\varphi_{\Delta t}^{\GSV,\FP}$, the set
  \begin{equation}
    \label{eq:B_GSV_FP}
    \calB_{\Delta t}^{\GSV,\FP}=\left\lbrace
    \begin{aligned}
      &x\in\calA_{\Delta t}^{\GSV,\FP}\bigcap\left(\psi_{\Delta t}^{\GSV,\FP}\right)^{-1}\left(\calA_{\Delta t}^{\GSV,\FP}\right),\\
      &\qquad\chi_{\Delta t}^{\GSV,\FP}\circ\psi_{\Delta t}^{\GSV,\FP}(x)=
      \Bigl(
        S\circ\chi_{\Delta t,1}^{\GSV,\FP}(x),S(x)
      \Bigr)
    \end{aligned}
    \right\rbrace,
  \end{equation}
  is open and nonempty.
\end{proposition}

These results show that a numerical solver can be theoretically constructed for the IMR or the GSV numerical scheme and that it can be used in Algorithms~\ref{alg:hmc_scheme_rev_check} and~\ref{alg:ghmc_scheme_rev_check}, leading to unbiased sampling. This is stated here for the IMR numerical scheme, a similar result holding for the GSV numerical scheme. Denote by~$\psi_{\Delta t}^{\IMR,\FP}=S\circ\varphi_{\Delta t}^{\IMR,\FP}$ and
\begin{equation*}
  \psi_{\Delta t}^{\rev,\IMR,\FP}=\psi_{\Delta t}^{\IMR,\FP}\mathbb{1}_{\calB_{\Delta t}^{\IMR,\FP}}+\id\mathbb{1}_{\calX\setminus\calB_{\Delta t}^{\IMR,\FP}}.
\end{equation*}

\begin{corollary}
  Suppose that Assumption~\ref{ass:H_C2_S_momentum_reversal} holds. Consider the IMR numerical scheme~$\Phi_{\Delta t}^{\IMR}$. Let~$\chi_{\Delta t}^{\IMR,\FP}$ and~$\varphi_{\Delta t}^{\IMR,\FP}$ be the maps defined in Proposition~\ref{prop:IMR_FP}. Then Propositions~\ref{prop:psi_involution},~\ref{prop:HMC_mu_invariance} and~\ref{prop:GHMC_mu_invariance} hold with~$\psi_{\Delta t}^{\rev}$ replaced by~$\psi_{\Delta t}^{\rev,\IMR,\FP}$ where~$\psi_{\Delta t}^{\rev,\IMR,\FP}$ is the map defined by~\eqref{eq:B_k}--\eqref{eq:psi_rev} with~$\varphi_{\Delta t}=\varphi_{\Delta t}^{\IMR,\FP}$ and~$\calB_{\Delta t}=\calB_{\Delta t}^{\IMR,\FP}$. In particular, Algorithms~\ref{alg:hmc_scheme_rev_check} and~\ref{alg:ghmc_scheme_rev_check} are unbiased using the map~$\psi_{\Delta t}^{\rev,\IMR,\FP}$.
\end{corollary}
Since the set~$\calB_{\Delta t}^{\IMR,\FP}$ is nonempty for~$\Delta t$ sufficiently small by Proposition~\ref{prop:IMR_FP}, it follows that the map~$\psi_{\Delta t}^{\rev,\IMR,\FP}$ is nontrivial for sufficiently small time steps~$\Delta t$.

\begin{proof}
  Let us check that Assumptions~\ref{ass:numerical_flow_symplectic} and~\ref{ass:S_C1_involution} are satisfied:
  \begin{itemize}
    \item It is clear from Proposition~\ref{prop:IMR_FP} that the maps~$\chi_{\Delta t}^{\IMR,\FP}$ and~$\varphi_{\Delta t}^{\IMR,\FP}$ define a numerical solver and a numerical flow for the IMR numerical scheme~$\Phi_{\Delta t}^{\IMR}$. From Proposition~\ref{prop:numerical_flow_symplectic}, the numerical flow~$\varphi_{\Delta t}^{\IMR,\FP}$ is symplectic so that Assumption~\ref{ass:numerical_flow_symplectic} holds. 
    \item The momentum reversal map~$S$ is a~$\calC^{1}$ involution such that~$\left\lvert\det\nabla S\right\rvert=1$. Moreover, the IMR numerical scheme~$\Phi_{\Delta t}^{\IMR}$ is~$S$-reversible by Proposition~\ref{prop:IMR_GSV_S_reversibility}. This shows that Assumption~\ref{ass:S_C1_involution} is satisfied.
  \end{itemize}  
  Therefore, Propositions~\ref{prop:psi_involution},~\ref{prop:HMC_mu_invariance} and~\ref{prop:GHMC_mu_invariance} hold with~$\psi_{\Delta t}=\psi_{\Delta t}^{\rev,\IMR,\FP}$.
\end{proof}

%---------------------------------------------------------
%---------------------------------------------------------
%---------------- IMPLICIT FUNCTION THM ------------------
%---------------------------------------------------------
%---------------------------------------------------------
\subsection{An idealized numerical solver using Newton's method and the implicit function theorem}
\label{subsec:numerical_solver_newton_method_ift}

We present in this section an idealized version of Newton's algorithm which solves exactly the following implicit problem for a fixed~$x\in\calX$:
\begin{equation}
  \label{eq:implicit_problem}
  \text{find }y\in\calX^k\text{ such that }\Phi_{\Delta t}(x,y)=0.
\end{equation}
A numerical solver~$\chi_{\Delta t}^{\Newton,\id}$ based on this idealized Newton's method is then defined thanks to Algorithm~\ref{alg:newton_idealized}. This solver satisfies all the requirements of Definition~\ref{def:numerical_solver}, see Proposition~\ref{prop:Newton}. A practical implementation of Newton's method is discussed in Section~\ref{subsec:practical_numerial_solver_newton_method}, see Algorithm~\ref{alg:newton_practice}. We discuss at the end of this section the relationship between the objects defined in this section and the numerical solver~$\chi_{\Delta t}^{\FP}$ constructed in Section~\ref{subsec:numerical_solvers_fixed_point}. 

Newton's method starting from an initial condition~$y_0\in\calX^k$ with a number of steps~$N_{\Newton}\geqslant 1$ generates a sequence~$(y^{i})_{1\leqslant i\leqslant N_{\Newton}}$ such that, when a convergence criterion is met, the last element of the sequence~$y^{N_{\Newton}}$ is an approximation of the output of the numerical solver. If~$\left\lVert\Phi_{\Delta t}(x,y^{N_{\Newton}})\right\rVert$ is small enough, then one can think of projecting~$y^{N_{\Newton}}$ onto the zero of~$y\mapsto\Phi_{\Delta t}(x,y)$ which is the closest to~$y^{N_{\Newton}}$. Such a projection can be formalized by an implicit function theorem, see Proposition~\ref{prop:Dimp}.

We first make the following assumption on the existence of a particular solution to~\eqref{eq:implicit_problem}. This ensures that the construction of the numerical solver can be initiated. This assumption is satisfied for our two running examples for sufficiently small time steps, see Propositions~\ref{prop:IMR_FP} and~\ref{prop:GSV_FP}.
\begin{assumption}
  \label{ass:existence_solution_implicit_problem}
  Let~$\Delta t>0$ and~$\Phi_{\Delta t}$ a numerical scheme. There exists~$(x,y)\in\calX\times\calX^k$ such that~$\Phi_{\Delta t}(x,y)=0$ and~$\nabla_{y}\Phi_{\Delta t}(x,y)$ is invertible.
\end{assumption}

\paragraph{Projection on the closest solution.} Assumption~\ref{ass:existence_solution_implicit_problem} allows us to define a map~$\Lambda_{\Delta t}^{\imp}\colon\calD_{\Delta t}^{\imp}\to\calX^k$, where~$\calD_{\Delta t}^{\imp}\subset\calX\times\calX^{k}$. More precisely, for any~$(x,y')\in\calD_{\Delta t}^{\imp}$, the configuration~$y=y'+\Lambda_{\Delta t}^{\imp}(x,y')$ is the closest to~$y'$ (in the sense of the norm in~$\calX^{k}$) for which~$\Phi_{\Delta t}(x,y)=0$ and~$\nabla_{y}\Phi_{\Delta t}(x,y)$ is invertible. This is stated in the following proposition, whose proof is postponed to Section~\ref{subsec:proof:prop_D_imp}. The superscript ``$\imp$'' is motivated by the fact that the implicit function theorem is used to construct the map~$\Lambda_{\Delta t}^{\imp}$.

\begin{proposition}
  \label{prop:Dimp}
  Suppose that Assumption~\ref{ass:existence_solution_implicit_problem} holds. Then there exists an open set~$\calD_{\Delta t}^{\imp}$ and a~$\calC^{1}$ function~$\Lambda_{\Delta t}^{\imp}\colon\calD_{\Delta t}^{\imp}\to\calX^{k}$ such that the following properties hold:
  \begin{itemize}
      \item The set~$\calD_{\Delta t}^{\imp}$ contains~$\calG_{\Delta t}$ and~$\Lambda_{\Delta t}^{\imp}$ is zero on~$\calG_{\Delta t}$, where 
      \begin{equation}
        \label{eq:G_dt}
        \calG_{\Delta t}=\left\lbrace (x,y)\in\calX\times\calX^{k},\,\Phi_{\Delta t}(x,y)=0,\,\nabla_{y}\Phi_{\Delta t}(x,y)\text{ is invertible}\right\rbrace.
      \end{equation}
      \item For any~$(x,y)\in\calD_{\Delta t}^{\imp}$, it holds~$\Phi_{\Delta t}(x,y+\Lambda_{\Delta t}^{\imp}(x,y))=0$ and the matrix~$\nabla_{y}\Phi_{\Delta t}(x,y+\Lambda_{\Delta t}^{\imp}(x,y))$ is invertible.
      \item For any~$(x_0,y_0)\in\calD_{\Delta t}^{\imp}$, there exists a neighborhood~$\calV_0$ of~$(x_0,y_0)$ and a positive real number~$\alpha_0$ such that
      \begin{equation*}
          \forall (x,y)\in\calV_0,\quad
          \left\lbrace
          \begin{aligned}
            &\left\lVert\Lambda_{\Delta t}^{\imp}(x,y)\right\rVert<\alpha_0,\\
            &\forall \lambda\in\calX^{k}\setminus\left\lbrace\Lambda_{\Delta t}^{\imp}(x,y)\right\rbrace,\, \Phi_{\Delta t}(x,y'+\lambda)=0\Longrightarrow\left\lVert\lambda\right\rVert\geqslant\alpha_0.
          \end{aligned}
          \right.
      \end{equation*}
  \end{itemize}
\end{proposition}
In particular, this shows that the map~$y\mapsto y+\Lambda_{\Delta t}^{\imp}(x,y)$ is a projection: if~$(x,y)\in\calD_{\Delta t}^{\imp}$, then the second statement shows that~$(x,y+\Lambda_{\Delta t}^{\imp}(x,y))\in\calG_{\Delta t}$ so that~\eqref{eq:G_dt} implies that
\begin{equation*}
    \Lambda_{\Delta t}^{\imp}\left(x,y+\Lambda_{\Delta t}^{\imp}(x,y)\right)=0.
\end{equation*}

\paragraph{Idealized Newton's algorithm.}
We now introduce an ``idealized'' version of Newton's algorithm to solve~\eqref{eq:implicit_problem}, see~Algorithm~\ref{alg:newton_idealized}. The algorithm takes as inputs a fixed number~$N_{\Newton}$ of Newton's iterations, starting from an initial guess~$y_{0}\in\calX^k$. We denote by~$\calD_{\Delta t}^{\Newton,\id}$ all the couples~$(x,y_0)$ such that Algorithm~\ref{alg:newton_idealized} does not return failure, in which case the output is denoted by~$\Pi_{\Delta t}^{\Newton,\id}(x,y_0)$. The map~$\Pi_{\Delta t}^{\Newton,\id}$ is thus defined from~$\calD_{\Delta t}^{\Newton,\id}$ to~$\calX^k$. Note that this algorithm is an idealized version of Newton's algorithm: it brings the point~$y^{\Newton}$ sufficiently close to the set of zeros of~$y\mapsto\Phi_{\Delta t}(x,y)$ so that it can be transported to an exact solution of~\eqref{eq:implicit_problem} using the map~$\Lambda_{\Delta t}^{\imp}$ defined in Proposition~\ref{prop:Dimp}. The symbol ``$\Newton,\id$'' is to insist that the objects are based on an idealized version of Newton's algorithm. A practical implementation is given in Algorithm~\ref{alg:newton_practice}, where other stopping criteria are used since the map~$\Lambda_{\Delta}^{\imp}$ does not have an analytical expression, see the discussion in Section~\ref{subsec:practical_numerial_solver_newton_method}.

\begin{algorithm}
  \caption{Idealized Newton's algorithm to find a zero of~$y\mapsto\Phi_{\Delta t}(x,y)=0$.}
  \label{alg:newton_idealized}
  Consider an initial condition~$y^{0}=y_{0}\in\calX^k$ and a number of iterations~$N_{\Newton}$, and set~$i=0$.
  \begin{enumerate}[label={[\thealgorithm.\roman*]}, align=left]
    \item \label{step:newton_1} If~$\nabla_{y}\Phi_{\Delta t}(x,y^{i})$ is not invertible, return failure.
    \item \label{step:newton_2} Compute
    \begin{equation}
      \label{eq:update_idealized_Newton}
      y^{i+1}=y^{i}-\left[\nabla_{y}\Phi_{\Delta t}(x,y^{i})\right]^{-1}\Phi_{\Delta t}(x,y^{i}).
    \end{equation}
    \item Increment~$i$. If~$i\leqslant N_{\Newton}-1$, go back to~\ref{step:newton_1}.
    \item \label{step:newton_4} If~$\left(x,y^{N_{\Newton}}\right)\notin\calD_{\Delta t}^{\imp}$, return failure. Else, return~$y^{N_\Newton}+\Lambda_{\Delta t}^{\imp}\left(x,y^{N_{\Newton}}\right)$ where~$\Lambda_{\Delta t}^{\imp}$ is defined in Proposition~\ref{prop:Dimp}.
  \end{enumerate}
\end{algorithm}

To define a numerical solver using the map~$\Pi_{\Delta t}^{\Newton,\id}$, one has to define the initial condition~$y_0\in\calX^{k}$ as a continuous function of~$x$. The initial condition~$y_0$ is a predictor of a solution to~\eqref{eq:implicit_problem}. One can for example choose~$y_0=(x,\dots,x)$ or~$y_0=\ES_{\Delta t}(x)$, namely the configuration obtained after one step of an explicit scheme (such as an explicit Euler scheme) with time step~$\Delta t>0$ applied to~$x$. The latter will be our choice in the following. For instance, for the IMR numerical scheme (see Definition~\ref{def:IMR}), this map can be chosen as
\begin{equation*}
  \ES_{\Delta t}^{\IMR}(q,p)=\begin{pmatrix}
    q+\Delta t\nabla_p H(q,p)\\[0.2cm]
    p-\Delta t\nabla_q H(q,p)
  \end{pmatrix}.
\end{equation*}
For the GSV numerical scheme, one can for instance choose
\begin{equation*}
  \ES_{\Delta t}^{\GSV}(q,p)=\begin{pmatrix}
    \ES^{\IMR}_{\Delta t/2}(q,p),\ES^{\IMR}_{\Delta t/2}\circ\ES^{\IMR}_{\Delta t/2}(q,p)
  \end{pmatrix}.
\end{equation*}
The idealized numerical solver associated with Algorithm~\ref{alg:newton_idealized} is then defined on
\begin{equation*}
  \calA_{\Delta t}^{\Newton,\id}=
  \left\lbrace
    x\in\calX\,\middle|\,
    (x,\ES_{\Delta t}(x))\in\calD_{\Delta t}^{\Newton,\id}
  \right\rbrace
\end{equation*}
by
\begin{equation}
  \label{eq:numerical_solver_newton_idealized}
  \chi_{\Delta t}^{\Newton,\id}(x)=\Pi_{\Delta t}^{\Newton,\id}(x,\ES_{\Delta t}(x)).
\end{equation}
Note that~$\calA_{\Delta t}^{\Newton,\id}$ is open since the map~$\ES_{\Delta t}$ is continuous and the updates~\eqref{eq:update_idealized_Newton} also define continuous mappings~$y^{i}\mapsto y^{i+1}$. The following proposition is then immediate thanks to the results of this section.
\begin{proposition}
  \label{prop:Newton}
  Suppose that Assumption~\ref{ass:existence_solution_implicit_problem} holds. Then the map~$\chi_{\Delta t}^{\Newton,\id}$ defined on the open set~$\calA_{\Delta t}^{\Newton,\id}$ is such that 
  \begin{equation*}
    \forall x\in\calA_{\Delta t}^{\Newton,\id},\qquad \Phi_{\Delta t}\left(x,\chi_{\Delta t}^{\Newton,\id}(x)\right)=0,\qquad \nabla_{y}\Phi_{\Delta t}\left(x,\chi_{\Delta t}^{\Newton,\id}(x)\right)\text{ is invertible}.
  \end{equation*}
\end{proposition}

\paragraph{Relationship between the maps~$\chi_{\Delta t}^{\FP}$ and~$\Lambda_{\Delta t}^{\imp}$.} In contrast with the fixed point method used in Section~\ref{subsec:numerical_solvers_fixed_point}, which constructs a numerical solver that outputs a solution of~\eqref{eq:implicit_problem} close to~$x$ (\emph{i.e.}~each component of the solution of~\eqref{eq:implicit_problem} lies in a neighborhood of~$x$), the numerical flow constructed on Newton's method allows to consider proposal moves~$\widetilde{y}$ which are not close to~$x$. The two numerical solvers should however output the same solution when the last element~$y^{\Newton}$ computed in Algorithm~\ref{alg:newton_idealized} is close to~$\chi_{\Delta t}^{\FP}(x)$. In a sense, the projection constructed in step~\ref{step:newton_4} generalizes the projection~$\chi_{\Delta t}^{\IMR,\FP}$ for all couples~$(x,\tilde{y})\in\calX\times\calX^{k}$ such that there exists a solution~$y$ of~\eqref{eq:implicit_problem} close to~$\tilde{y}$ (and not only close to~$x$) for which~$\nabla_{y}\Phi_{\Delta t}(x,y)$ is invertible. This intuition is formalized in the next proposition, stated for the IMR numerical scheme for concreteness, a similar result holding for the GSV numerical scheme.

\begin{proposition}
  Suppose that Assumption~\ref{ass:H_C2_S_momentum_reversal} holds. Let~$\Lambda_{\Delta t}^{\IMR,\imp}$ be the map defined in Proposition~\ref{prop:Dimp} associated with the IMR numerical scheme. Then there exists~$\Delta t_{\star}>0$ such that, for any~$\Delta t\in(0,\Delta t_{\star}]$ and~$x\in\calA_{\Delta t}^{\IMR,\FP}$, there is a neighborhood~$\calU_x\subset\calX$ of~$\chi_{\Delta t}^{\IMR,\FP}(x)$ for which
  \begin{equation}
    \label{eq:relation_fp_imr}
    \forall y\in\calU_x,\qquad 
    \chi_{\Delta t}^{\IMR,\FP}(x)=y+\Lambda_{\Delta t}^{\IMR,\imp}(x,y).
  \end{equation}
\end{proposition}
\begin{proof}
  Fix~$\Delta t\in(0,\Delta t_{\star}]$ with~$\Delta t_{\star}$ defined in Proposition~\ref{prop:IMR_FP}, and consider~$x\in\calA_{\Delta t}^{\IMR,\FP}$. Equation~\eqref{eq:imr_fp} implies that~$\left(x,\chi_{\Delta t}^{\IMR,\FP}(x)\right)\in\calG_{\Delta t}\subset\calD_{\Delta t}^{\imp}$ where~$\calG_{\Delta t}$ is defined by~\eqref{eq:G_dt}. Using the third property of Proposition~\ref{prop:Dimp} with~$x_0=x$ and~$y_0=\chi_{\Delta t}^{\IMR,\FP}(x)$, there exists an open neighborhood~$\calU_x\subset\calX$ of~$\chi_{\Delta t}^{\IMR,\FP}(x)$ and a positive real number~$\alpha_x$ such that
  \begin{equation}
    \label{eq:alpha_x}
    \forall y\in\calU_x,\quad
    \left\lbrace
      \begin{aligned}
        &\left\lVert\Lambda_{\Delta t}^{\imp}(x,y)\right\rVert<\alpha_x,\\
        &\forall \lambda\in\calX\setminus\left\lbrace\Lambda_{\Delta t}^{\imp}(x,y)\right\rbrace,\, \Phi_{\Delta t}(x,y+\lambda)=0\Longrightarrow\left\lVert\lambda\right\rVert\geqslant\alpha_x.
      \end{aligned}
    \right.
  \end{equation}
  Upon reducing the size of~$\calU_x$, it can be assumed that
  \begin{equation}
    \label{eq:open_ball_alpha_x}
    \forall y\in\calU_x,\qquad \left\lVert \chi_{\Delta t}^{\IMR,\FP}(x)-y\right\rVert<\alpha_x.
  \end{equation}
  Fix~$y\in\calU_x$. If~\eqref{eq:relation_fp_imr} does not hold, let~$\lambda=\chi_{\Delta t}^{\IMR,\FP}(x)-y\neq\Lambda_{\Delta t}^{\imp}(x,y)$. Then, in view of~\eqref{eq:imr_fp},
  \begin{equation*}
    \Phi_{\Delta t}^{\IMR}(x,y+\lambda)=\Phi_{\Delta t}^{\IMR}\left(x,\chi_{\Delta t}^{\IMR,\FP}(x)\right)=0,
  \end{equation*}
  so that~\eqref{eq:alpha_x} implies that
  \begin{equation*}
    \left\lVert\lambda\right\rVert=\left\lVert\chi_{\Delta t}^{\IMR,\FP}(x)-y\right\rVert\geqslant \alpha_x,
  \end{equation*}
  which contradicts~\eqref{eq:open_ball_alpha_x}. This shows that~\eqref{eq:relation_fp_imr} holds.
\end{proof}

%---------------------------------------------------------
%---------------------------------------------------------
%---------------- PRACTICAL NEWTON -----------------------
%---------------------------------------------------------
%---------------------------------------------------------
\subsection{A practical numerical solver using a practical implementation of Newton's method}
\label{subsec:practical_numerial_solver_newton_method}

We give in this section a practical implementation of Algorithm~\ref{alg:newton_idealized}, see Algorithm~\ref{alg:newton_practice}. Denote respectively by~$\eta_{\Newton},\widetilde{\eta}_{\Newton}$ and~$N_{\Newton}$ the convergence thresholds and the maximum number of iteration loops for Newton's algorithm. A practical implementation of Newton's method is provided in Algorithm~\ref{alg:newton_practice} (with convergence checks expressed in terms of relative errors; absolute errors could be considered as well).

\begin{algorithm}
    \caption{Practical Newton's algorithm to find a zero of~$y\mapsto\Phi_{\Delta t}(x,y)=0$.}
    \label{alg:newton_practice}
    Consider an initial condition~$y^{0}=y_0\in\calX^k$, a maximum number of iterations~$N_{\Newton}$, and set~$i=0$.
    \begin{enumerate}[label={[\thealgorithm.\roman*]}, align=left]
      \item\label{step:practical_newton_1} If~$\left(\nabla_{y}\Phi_{\Delta t}\right)(x,y^{i})$ is not invertible, return failure.
      \item Else, compute 
      \begin{equation}
        \label{eq:newton_practice_updates}
        y^{i+1}=y^{i}-\left[
          \left(
            \nabla_{y}\Phi_{\Delta t}
          \right)(x,y^{i})
        \right]^{-1}\Phi_{\Delta t}(x,y^{i}).
      \end{equation}
      \item\label{step:practical_newton_3} If~$\left\lVert\Phi_{\Delta t}(x,y^{i+1})\right\rVert<\eta_{\Newton}\left\lVert\Phi_{\Delta t}(x,y_0)\right\rVert$ or~$\left\lVert y^{i+1}-y^{i}\right\rVert<\widetilde{\eta}_{\Newton}\left\lVert y^{i}\right\rVert$, return~$y^{i+1}$.
      \item Increment~$i$. If~$i\leqslant N_{\Newton}-1$, go back to~\ref{step:practical_newton_1}.
      \item Return failure. 
    \end{enumerate}
\end{algorithm}

Let us first highlight the key differences between Algorithms~\ref{alg:newton_idealized} and~\ref{alg:newton_practice}:
\begin{itemize}
  \item At each iteration, one checks that the matrix~$\nabla_{x'}\Phi_{\Delta t}(x,y^{i})$ is numerically invertible. Typically, this is done by checking its singular values (making sure that the condition number, which is the ratio of the largest and smallest singular values is not too large; and/or making sure that the matrix has full rank by counting the number of singular values whose magnitude is above a given threshold). If the matrix is not numerically invertible, then the algorithm returns failure.
  \item Along with the maximal number of iterations~$N_{\Newton}$, two other stopping criteria of the form~$\left\lVert\Phi_{\Delta t}(x,y^{i})\right\rVert<\eta_{\Newton}\left\lVert\Phi_{\Delta t}(x,y_0)\right\rVert$ and~$\left\lVert y^{i+1}-y^{i}\right\rVert<\widetilde{\eta}_{\Newton}\left\lVert y^{i}\right\rVert$ for some tolerances~$\eta_{\Newton},\widetilde{\eta}_{\Newton}$ are used. If these criteria are not satisfied after~$N_{\Newton}$ iterations, then the algorithm returns failure.
  \item One simply performs Newton's iterations, namely steps~\ref{step:newton_1}--\ref{step:newton_2} of Algorithm~\ref{alg:newton_idealized}, without projecting the final configuration using the map~$\Lambda_{\Delta t}^{\imp}$. It is thus assumed that the obtained configuration~$y$ is sufficiently close to an exact solution (since~$\Phi_{\Delta t}(x,y)$ is small) so that~$\Lambda_{\Delta t}^{\imp}(x,y)=0$.
\end{itemize}

\begin{remark}
  \label{rem:case_k_2}  
  In Algorithm~\ref{alg:newton_practice}, one may either compute the whole gradient~$\nabla_{y}\Phi_{\Delta t}$ or use specific strategies depending on the numerical scheme at hand. Let us illustrate this comment on the specific example of the GSV numerical scheme. Note that~$\nabla_{y}\Phi_{\Delta t}^{\GSV}$ is equivalent to both~$\nabla_{y}\Phi_{\Delta t}^{\EulerB}$ and~$\nabla_{y}\Phi_{\Delta t}^{\EulerA}$ being invertible, see the proof of Proposition~\ref{prop:numerical_flow_symplectic} in Section~\ref{subsec:proof_symplectic}. Therefore, three strategies at least can be considered to build a solution to~\eqref{eq:implicit_problem}: (i) using the full gradient~$\nabla_{y}\Phi_{\Delta t}^{\IMR}(x,x_1,x_2)$; (ii) iterate Newton's method to first obtain~$x_1$ in the Euler~B part of the scheme, and then iterate Newton's method to obtain~$x_2$ in the Euler~A part of the scheme; loop between updates of~$x_1$ and~$x_2$ until both have converged, as done for constrained molecular systems (see~\cite[Section~4.3.3]{leimkuhler_2015} and~\cite[Section~7.2.2]{leimkuhler_2005}).
\end{remark}

We now define the practical numerical solver constructed using Algorithm~\ref{alg:newton_practice}. We denote by~$\calD_{\Delta t}^{\Newton,\pr}$ all the couples~$(x,y_0)\in\calX\times\calX^{k}$ such that Algorithm~\ref{alg:newton_practice} does not return failure, in which case the output is denoted by~$\Pi_{\Delta t}^{\Newton,\pr}(x,y_0)$. The map~$\Pi_{\Delta t}^{\Newton,\pr}$ is thus defined from~$\calD_{\Delta t}^{\Newton,\pr}$ to~$\calX^{k}$. The set~$\calD_{\Delta t}^{\Newton,\pr}$ is open since the set on which the condition in step~\ref{step:practical_newton_1} is satisfied is open, the updates~\eqref{eq:newton_practice_updates} define continuous mappings~$y^{i}\mapsto y^{i+1}$ and the checks in step~\ref{step:practical_newton_3} are based on continuous mappings with strict inequality constraints. The superscript label ``$\Newton,\pr$'' should be contrasted with the previous label ``$\Newton,\id$'', since the objects are now defined based on a practical implementation of Newton's method.

Recall that~$\ES_{\Delta t}$ is defined on~$\calX$ and corresponds to an explicit scheme with time step~$\Delta t$ (such as explicit Euler), see Section~\ref{subsec:numerical_solver_newton_method_ift}. The numerical solver associated with Algorithm~\ref{alg:newton_practice} is defined on
\begin{equation*}
  \calA_{\Delta t}^{\Newton,\pr}=\left\lbrace x\in\calX\,\middle|\,(x,\ES_{\Delta t}(x))\in\calD_{\Delta t}^{\Newton,\pr}\right\rbrace
\end{equation*}
by
\begin{equation*}
  \chi_{\Delta t}^{\Newton,\pr}(x)=\Pi_{\Delta t}^{\Newton,\pr}(x,\ES_{\Delta t}(x)).
\end{equation*}
Note that~$\calA_{\Delta t}^{\Newton,\pr}$ is open since the map~$\ES_{\Delta t}$ is continuous and the set~$\calD_{\Delta t}^{\Newton,\pr}$ is open. Beware that we name the map~$\chi_{\Delta t}^{\Newton,\pr}$ a ``numerical solver'' as it resembles the numerical solver defined by~\eqref{eq:numerical_solver_newton_idealized}, but~$\chi_{\Delta t}^{\Newton,\pr}$ does not satisfy the requirements of Definition~\ref{def:numerical_solver} since, in general, the equality~$\Phi_{\Delta t}\left(x,\chi_{\Delta t}^{\Newton,\pr}(x)\right)= 0$ does not hold.

\paragraph{Practical implementation of the map~$\psi_{\Delta t}^{\rev}$.} We give in Algorithm~\ref{alg:psi_dt_rev} a practical implementation of the map~$\psi_{\Delta t}^{\rev}$ defined in~\eqref{eq:psi_rev} using the practical implementation of Newton's method described Algorithm~\ref{alg:newton_practice}. We introduce to this end a parameter~$\eta_{\rev}$ to quantify the~$S$-reversibility.
\begin{algorithm}
  \caption{Practical implementation of the map~$\psi_{\Delta t}^{\rev}$ to a configuration~$x\in\calX$.}
  \label{alg:psi_dt_rev}
  \begin{enumerate}[label={[\thealgorithm.\roman*]}, align=left]
    \item\label{step:psi_dt_rev_1} Apply Algorithm~\ref{alg:newton_practice} to solve~$\Phi_{\Delta t}(x,y)=0$. If it fails, return~$x$.
    \item\label{step:psi_dt_rev_2} Else, denote by~$(\tilde{y}_1,\dots,\tilde{y}_k)$ the output, and apply Algorithm~\ref{alg:newton_practice} to solve~$\Phi_{\Delta t}(S(\tilde{y}_k),y)=0$. If it fails, return~$x$.
    \item\label{step:psi_dt_rev_3} Else, denote by~$(\bar{y}_1,\dots,\bar{y}_k)$ the output. If
    \begin{equation*}
      \left\lVert (\bar{y}_1,\dots,\bar{y}_{k-1},\bar{y}_k)-(S(\tilde{y}_{k-1}),\dots,S(\tilde{y}_{1}),S(x))\right\rVert<\eta_{\rev}\left\lVert x\right\rVert,  
    \end{equation*}
    return~$S(\tilde{y}_k)$. Else, return~$x$.
  \end{enumerate}
\end{algorithm}

We clearly identify the three successive conditions announced in the introduction and made precise in Section~\ref{subsec:S_reversibility}: step~\ref{step:psi_dt_rev_1} refers to solving the implicit forward problem and rejecting if it fails; step~\ref{step:psi_dt_rev_2} corresponds to solving the implicit backward problem and rejecting if it fails; and step~\ref{step:psi_dt_rev_3} corresponds to the equality~\eqref{eq:s_rev_solver} and rejecting if it does not hold numerically.

%---------------------------------------------------------
%---------------------------------------------------------
%---------------------------------------------------------
%---------------------------------------------------------
%---------------------------------------------------------
%---------------- NUMERICAL RESULTS ----------------------
%---------------------------------------------------------
%---------------------------------------------------------
%---------------------------------------------------------
%---------------------------------------------------------
%---------------------------------------------------------

\section{Numerical results}
\label{sec:numerical_results}
We illustrate in this section through two numerical examples the need to check for~$S$-reversibility to guarantee unbiased sampling. In particular, we show that the Markov chains created via Algorithms~\ref{alg:hmc_scheme_rev_check} and~\ref{alg:ghmc_scheme_rev_check} sample without bias the target measure~$\mu$ defined in~\eqref{eq:measure_non_separable} (and hence~$\pi$ defined in~\eqref{eq:invariant_measure_position_RMHMC} when using the Hamiltonian function~\eqref{eq:H_RMHMC} as in RMHMC, see Section~\ref{subsec:hmc}), while the standard HMC and GHMC algorithms fail to do so for large time steps. We consider sampling from a Boltzmann--Gibbs measure associated with a one-dimensional double well potential in Section~\ref{subsec:one_dimensional_confining_double_well_potential}. We then sample from an anisotropic two-dimensional potential in Section~\ref{subsec:anisotropic_two_dimensional_potential} and show the benefits of using a position-dependent diffusion coefficient to faster reach stationarity. The GSV numerical scheme~\eqref{eq:GSV} is used to integrate the Hamiltonian dynamics, and consider the Hamiltonian function~\eqref{eq:H_RMHMC} and the momentum reversal map~$S$ defined in~\eqref{eq:S_momentum_reversal}. Note that, by Lemma~\ref{lem:solver_rev_from_flow_rev_gsv}, it suffices to check the~$S$-reversibility on the beginning/ending positions. Our implementation of the GSV algorithm consists in using successively two Newton's methods to first fully solve the Euler~A scheme, and then fully solve the Euler B scheme. This corresponds to strategy (ii) in Remark~\ref{rem:case_k_2}. To check the invertibility of the matrix~$\nabla_{y}\Phi_{\Delta t}(x,y^{i})$ in step~\ref{step:practical_newton_1}, we use the \textsc{rank()} method from the Julia programming language, which amounts to counting the number of singular values above a default threshold of~$m\varepsilon\sigma_m$ where~$m$ is the size of the matrix,~$\sigma_m$ its largest singular value and~$\varepsilon=2.22\times 10^{-16}$. The codes which have been used to obtain the numerical results are available at the url~\url{https://github.com/rsantet/RMHMC}.

%---------------------------------------------------------
%---------------------------------------------------------
%---------------- ONE DIM --------------------------------
%---------------------------------------------------------
%---------------------------------------------------------
\subsection{One dimensional confining double well potential}
\label{subsec:one_dimensional_confining_double_well_potential}

In this example, we show by computing histograms that not checking for numerical reversibility induces a bias on the invariant probability measure with respect to the target probability distribution, especially when time steps become large; while using either Algorithms~\ref{alg:hmc_scheme_rev_check} or~\ref{alg:ghmc_scheme_rev_check} with the numerical flow with~$S$-reversibility check~\eqref{eq:psi_rev} (implemented with Algorithm~\ref{alg:psi_dt_rev}) allows performing unbiased numerical sampling. 

We consider a confining double well potential defined on~$\bbR$ as
\begin{equation}
    \label{eq:example_V_1D}
    V_{\sigma,h}(q) = q^{2}-1+\frac{h}{\sqrt{2\pi\sigma^{2}}}\exp\left(-\frac{q^{2}}{2\sigma^{2}}\right),
\end{equation}
and the target probability measure is the usual Gibbs measure~$\pi$ defined by~\eqref{eq:invariant_measure_position_RMHMC}. The parameters~$\sigma$ and~$h$ allows to vary the width and height of the energy barrier at~$q=0$. The potential~\eqref{eq:example_V_1D} exhibits two wells as seen in Figure~\ref{fig:double_well_potential}, which enables us to evaluate the performance of our algorithm for tackling multimodal target distributions. We choose~$\sigma=0.2$,~$h=1$, a friction parameter~$\gamma=1$ for the GHMC algorithm, and the diffusion coefficient (which corresponds to the inverse mass in a Hamiltonian setting)
\begin{equation*}
  D(q)=\left(\dfrac{1.5+\cos(\pi q)}{2}\right)^{2}\geqslant\dfrac{1}{16}>0.
\end{equation*}
We run each simulation for~$N_{\iter} = 10^{7}$ time steps starting from~$q_0=-0.5$ and~$p_{0}$ sampled from~$\calN(0,D(q_0)^{-1})$. Newton's algorithm to solve the implicit steps has convergence tolerance thresholds of~$\eta_{\Newton} = \widetilde{\eta}_{\Newton} = 10^{-12}$, and we perform a maximum of~$N_{\Newton}=100$ iterations for every Newton algorithm loop. We use a tolerance threshold of~$\eta_{\rev}=10^{-8}$ to check for~$S$-reversibility. Observe that~$\eta_{\rev}$ is somewhat larger than~$\eta_{\Newton},\widetilde{\eta}_{\Newton}$ to take into account an error amplification phenomenon due to the iterations in Newton's method. 

\begin{figure}
  \centering
  \includegraphics[scale=0.17]{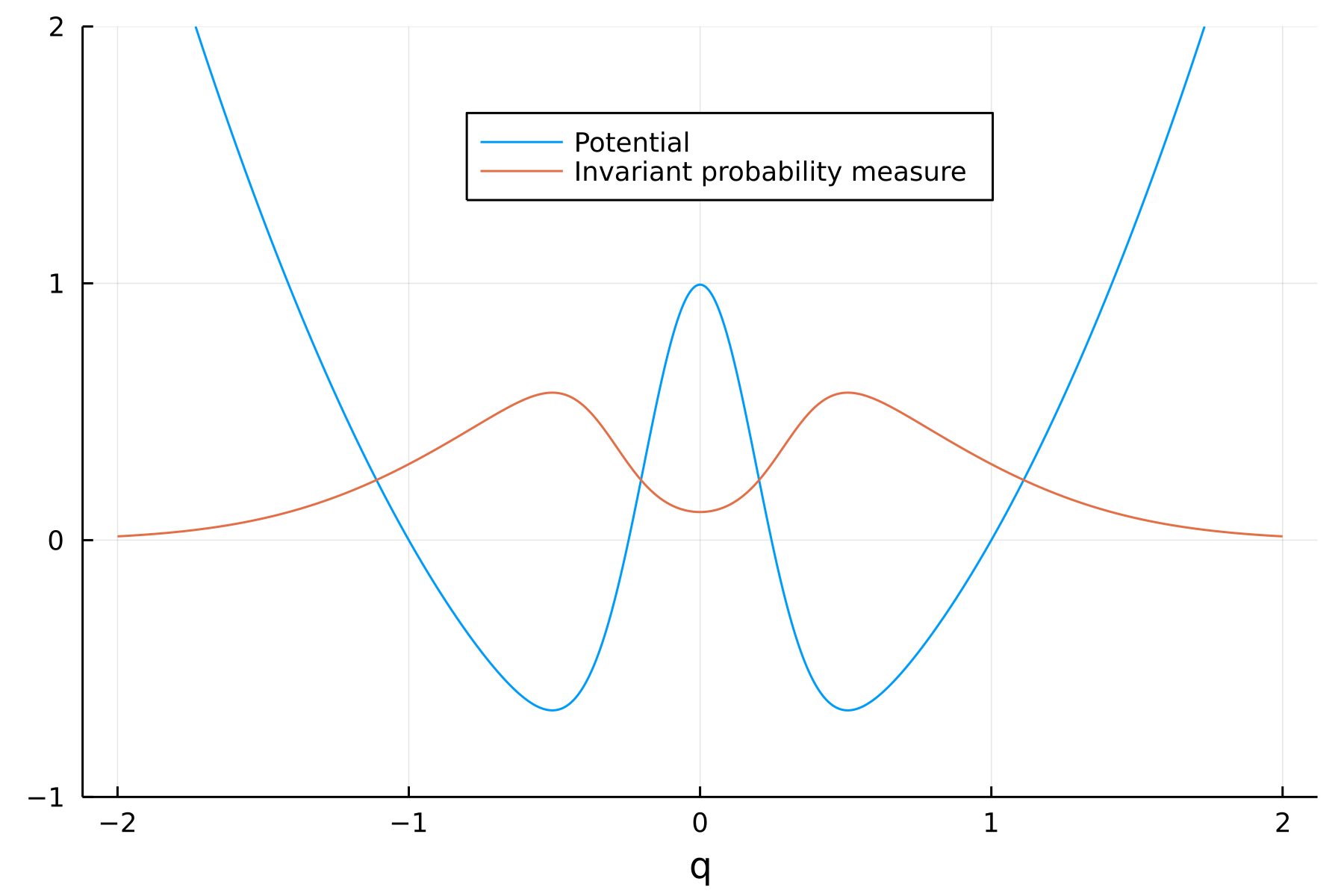}
  \caption[Confining one dimensional double well potential and its associated invariant measure.]{Confining double well potential with~$\sigma=0.2$,~$h=1$ and the associated invariant measure proportional to~$\rme^{-V}$.}
  \label{fig:double_well_potential}
\end{figure}

In our examples, not checking~$S$-reversibility means that we use Algorithm~\ref{alg:ghmc_scheme_rev_check} with~$\psi_{\Delta t}^{\rev}$ replaced by
\begin{equation*}
  \forall x\in\calX,\qquad
  \psi_{\Delta t}^{\fwd}(x)=\psi_{\Delta t}(x)\mathbb{1}_{x\in\calA_{\Delta t}}+x\mathbb{1}_{x\notin\calA_{\Delta t}}.
\end{equation*}
In other words, only forward convergence is checked.

\paragraph{Observing the bias with histograms.} For~$\Delta t$ small enough, checking for~$S$-reversibility is not needed as each implicit problem is well posed: the solution computed by the algorithm satisfies~$S$-reversibility, see Figure~\ref{fig:hist_3} (this is consistent with the results from Section~\ref{subsec:numerical_solvers_fixed_point}). When the time step increases, convergence problems for Newton's method appear and hence~$S$-reversibility issues arise, see Figures~\ref{fig:hist_5} and~\ref{fig:hist_7}. All these figures represent the stationary distributions of the Markov chains: we have indeed checked that the obtained histograms do not vary for a larger number of iterations. The fact that the histograms are even (as expected from the symmetry of the potential) is also a good indication that the stationary state has been reached since the initial condition is not symmetric. These results illustrate the fact that~$S$-reversibility checks are needed to create a Markov chain which yields an unbiased sampling of the Gibbs measure~$\pi$.

\begin{figure}
  \centering
  \includegraphics[scale=0.14]{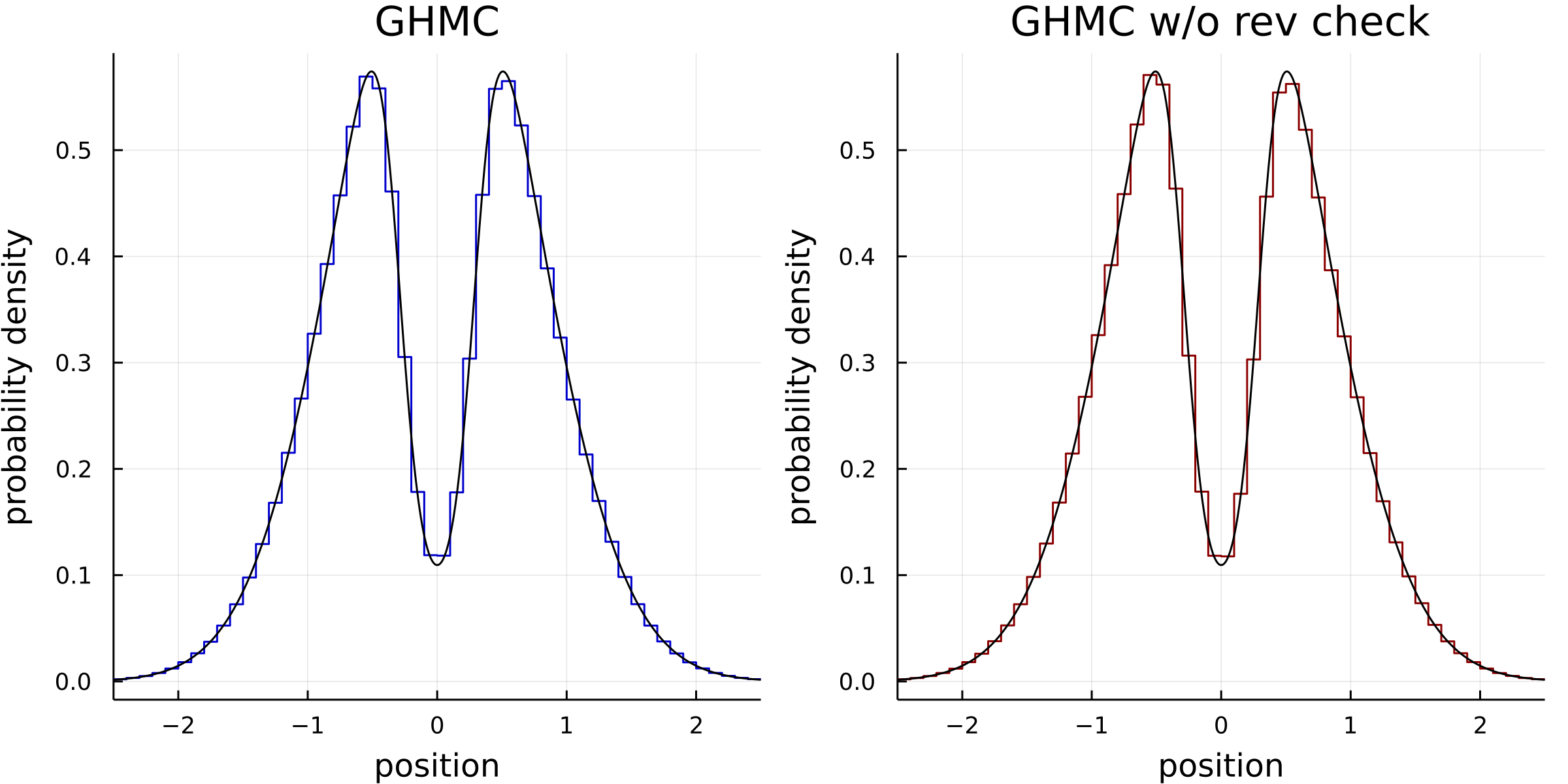}
  \caption{Sampling with a time step~$\Delta t=0.15$. The histogram on the left is obtained thanks to Algorithm~\ref{alg:ghmc_scheme_rev_check}, while the one on the right is obtained without checking~$S$-reversibility. There is no significant difference between the two choices.}
  \label{fig:hist_3}
\end{figure}

\begin{figure}
  \centering
  \includegraphics[scale=0.14]{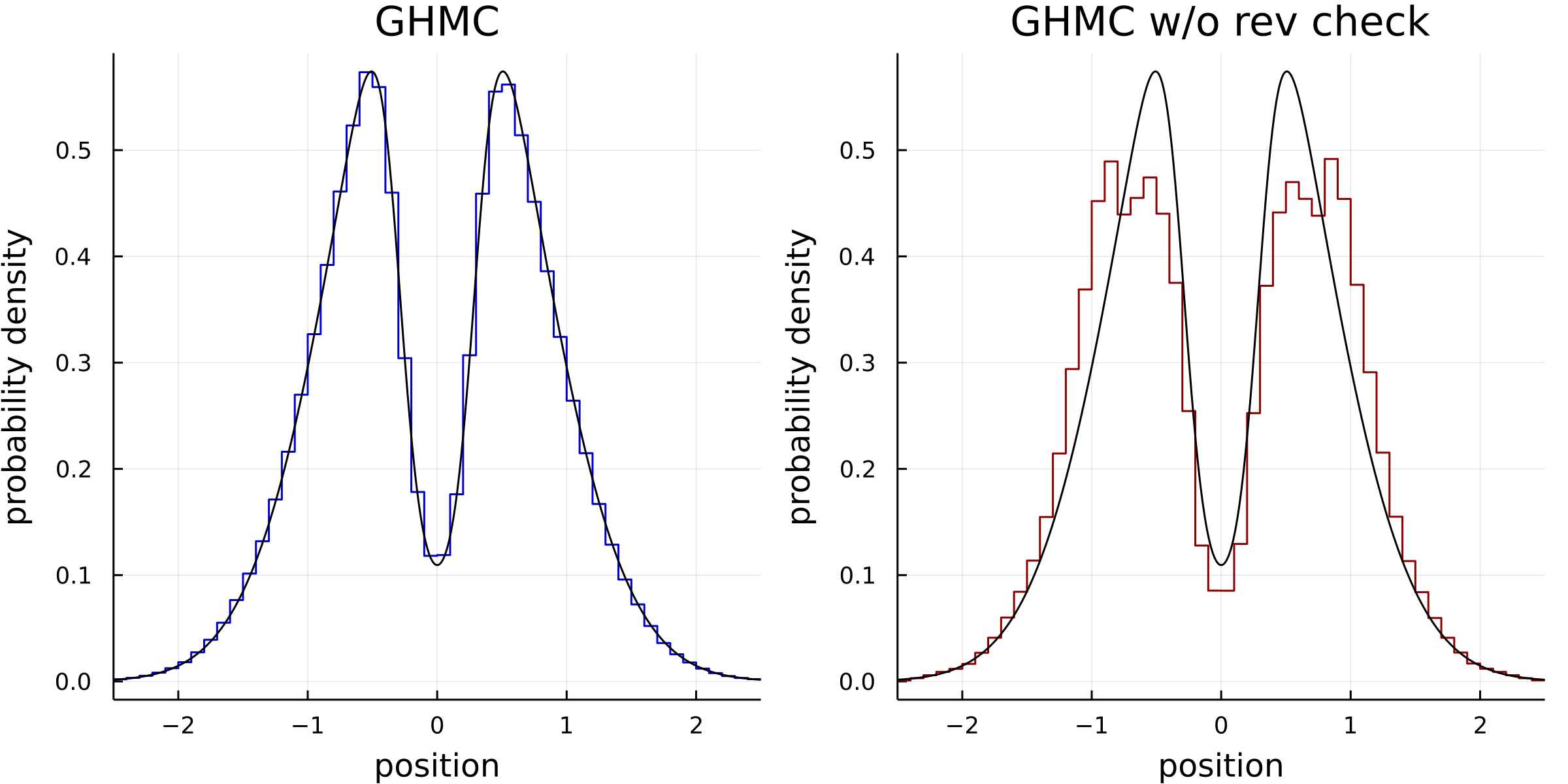}
  \caption{Same setting as Figure~\ref{fig:hist_3}, with a time step~$\Delta t=0.69$. Significant bias can be observed when~$S$-reversibility is not checked.}
  \label{fig:hist_5}
\end{figure}

\begin{figure}
  \centering
  \includegraphics[scale=0.14]{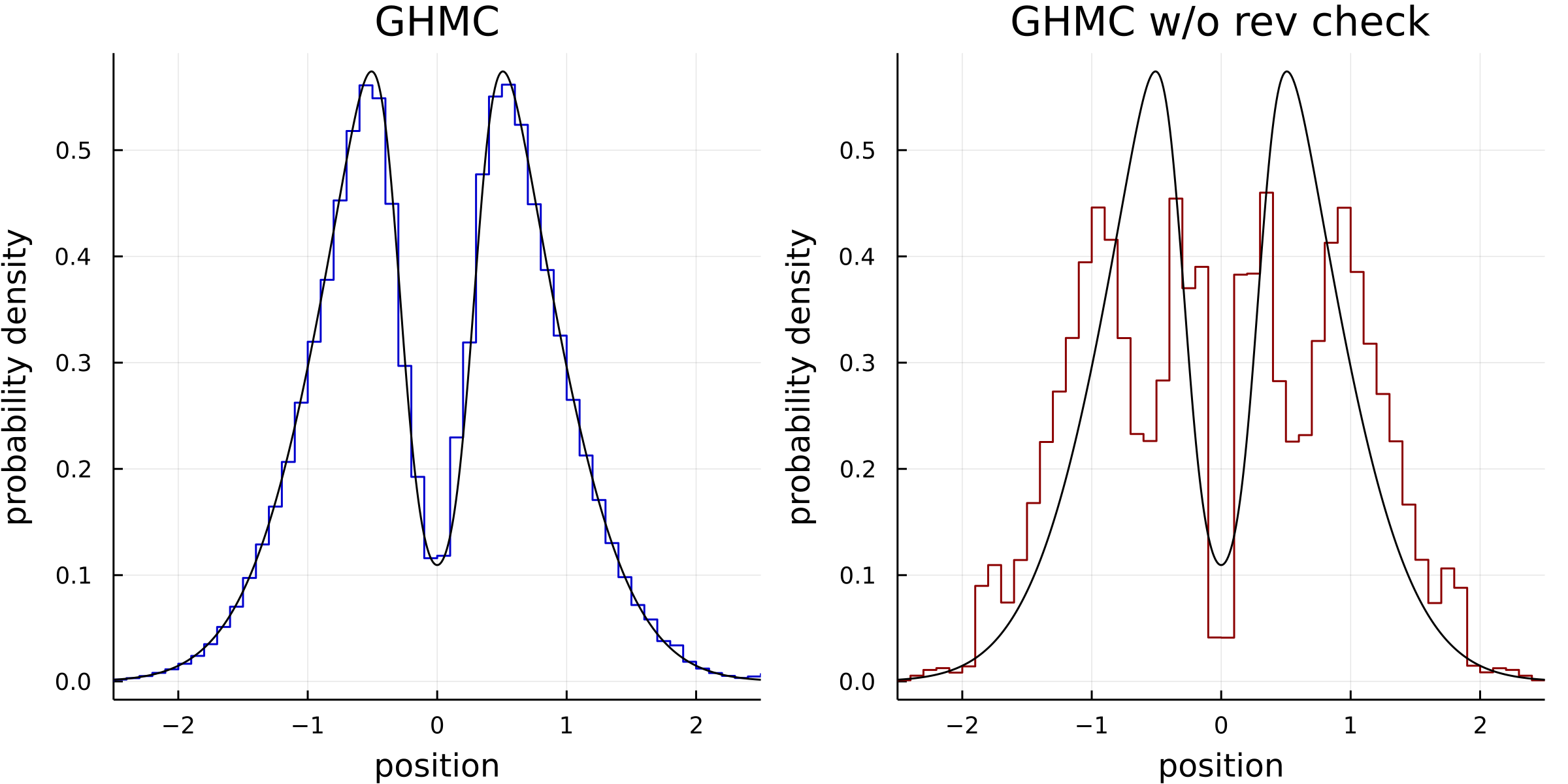}
  \caption{Same setting as Figure~\ref{fig:hist_3}, with a time step~$\Delta t=1.08$. Significant bias can be observed when~$S$-reversibility is not checked.}
  \label{fig:hist_7}
\end{figure}

\paragraph{Rejection probabilities and weak order consistency.} We present in Figure~\ref{fig:rejection_probabilities} the rejection probabilities as a function of the time step~$\Delta t$. We also give the values for the time steps associated with the histograms of Figures~\ref{fig:hist_3} to~\ref{fig:hist_7} in Table~\ref{tab:rejection_probabilities}. For small time steps, there is no rejection due to~$S$-reversibility checks, and the global rejection probability is only composed of the rejection due to the Metropolis--Hastings acceptance/rejection step. The global rejection probability therefore scales as~$\rmO(\Delta t^{3})$ for~$\Delta t$ small, and we observe numerically that the rejection probability due to the~$S$-reversibility is very small and that~$r_3\geqslant 3$ in any case. This is consistent with~\eqref{eq:rmhmc_rejection_probability}, and suggest that the HMC algorithm (Algorithm~\ref{alg:hmc_scheme_rev_check}) using the Hamiltonian function~\eqref{eq:H_RMHMC}, \emph{i.e.}~the RMHMC algorithm, yields a weakly consistent discretization of order~1 of the overdamped Langevin dynamics with a position-dependent diffusion coefficient~\eqref{eq:overdamped_langevin_diffusion}, following the result of Proposition~\ref{prop:one_step_hmc_weakly_consistent_overdamped_langevin}.

For large time steps, the rejection due to~$S$-reversibility checks become prominent. In particular, forward convergence and the~$S$-reversibility condition~\eqref{eq:s_rev_solver} are difficult to satisfy. It is interesting to note that, when there is convergence of the implicit forward problem, the implicit backward problem very often converges as well, but not necessarily to the solution compatible with~$S$-reversibility. This highlights the need to implement~$S$-reversibility checks to perform unbiased sampling.

\begin{table}
  \centering
  \begin{tabular}{c|ccccc}
    \toprule
    Time step & Forward & Backward & $S$-reversibility & Metropolis--Hastings & Global\\\midrule
    $0.15$ & $4.8\times10^{-1}$ & $5.1\times10^{-4}$ & $1.3\times10^{-3}$ & $2.6$ & $3.1$\\\midrule
    $0.69$ & $27$ & $0.5$ & $23.9$ & $13$ & $64$ \\\midrule
    $1.08$ & $34$ & $1.2$ & $44$ & $6.8$ & $86$\\
    \bottomrule
  \end{tabular}
  \caption{Decomposition in percentages of the rejection probabilities for various time steps~$\Delta t$. ``Forward'' refers to non-convergent iterations to solve the implicit forward problem; ``Backward'' refers to a convergence of the implicit forward problem but a non-convergent implicit backward problem; ``S-reversibility'' refers to convergent implicit forward and backward problems for which~\eqref{eq:s_rev_solver} is not satisfied; upon acceptance in the three previous steps, ``Metropolis--Hastings'' refers to a rejection in the acceptance/rejection step~\ref{step:ghmc_3}. Finally, ``Global'' is the global rejection probability, namely the sum of all the previous columns.}
  \label{tab:rejection_probabilities}
\end{table}

\begin{figure}
  \centering
  \includegraphics[scale=0.5]{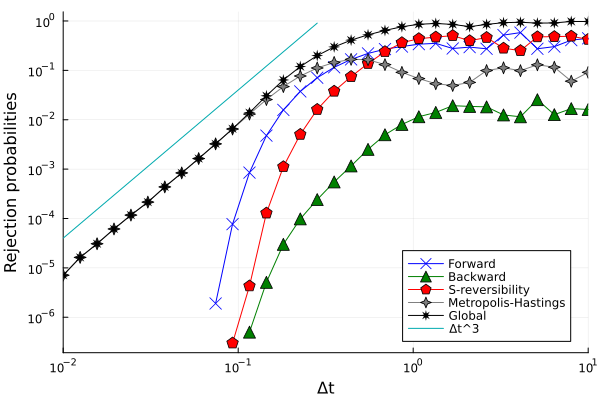}
  \caption{Rejection probabilities as a function of the time step~$\Delta t\in[10^{-2},10]$. We refer to the caption of Table~\ref{tab:rejection_probabilities} for the meanings of the various labels.}
  \label{fig:rejection_probabilities} 
\end{figure}
%---------------------------------------------------------
%---------------------------------------------------------
%---------------- TWO DIM --------------------------------
%---------------------------------------------------------
%---------------------------------------------------------
\subsection{Anisotropic two-dimensional potential}
\label{subsec:anisotropic_two_dimensional_potential}

\begin{figure}
  \centering
  \includegraphics[scale=0.65]{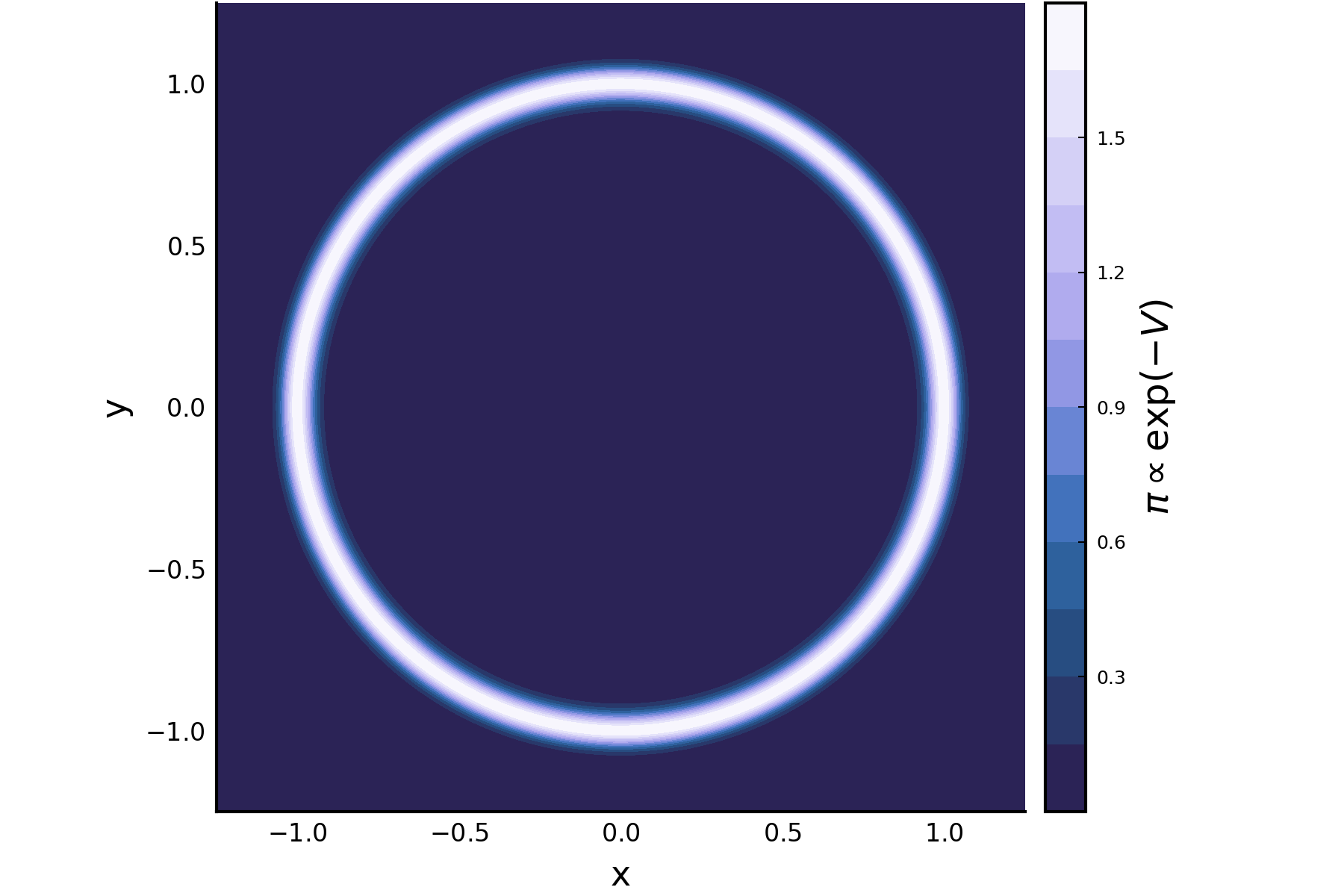}
  \caption{Invariant probability measure with density proportional to~$\rme^{-V}$ for the anisotropic two-dimensional potential~$V(x,y)=100(x^{2}+y^{2}-1)^{2}$.}
  \label{fig:circle_potential}
\end{figure}

We demonstrate on a representative example that using a non-constant diffusion coefficient helps to reach equilibrium faster. In particular, using the Hamiltonian function defined in~\eqref{eq:H_RMHMC}, we compare the use of a constant diffusion (which corresponds to the usual HMC scheme) to a diffusion coefficient which takes into account the anisotropy of the potential energy landscape (using the RMHMC algorithm). In the first case, GSV is an explicit scheme as the Hamiltonian function is separable, so that standard HMC results apply and unbiased numerical sampling is performed. In the second case, implicit schemes, hence~$S$-reversibility checks, have to be used in order to guarantee unbiasedness.

We consider a target probability measure~$\pi$ with density proportional to~$\rme^{-V(q)}$ with~$q=(x,y)\in\bbR^2$ and~$V(x,y)=100(x^{2}+y^{2}-1)^{2}$, whose support is essentially an annulus in~$\bbR^{2}$, see Figure~\ref{fig:circle_potential}. The width of the annulus is very small. This suggests that, when using an isotropic diffusion coefficient~$D(x,y)=\rmI_2$, the Metropolis--Hastings procedure will likely reject the proposal as it typically ends up in high energy regions. Instead, one can choose a diffusion coefficient which takes into account the geometry of the potential energy landscape so that the proposal is more likely to be close to the annulus. To that aim, we propose the following diffusion coefficient:
\begin{equation}
  \label{eq:rmhmc_anisotropic_diffusion_coefficient}
  D^{\aniso}(x,y)=\varepsilon\rmI_2+tt^{\sfT},
\end{equation}
where 
\begin{equation*}
  t=\frac{1}{\sqrt{x^{2}+y^{2}}}\begin{pmatrix}
    -y \\ x
  \end{pmatrix},
\end{equation*}
and~$\varepsilon>0$ is a small parameter ensuring that the diffusion coefficient is always positive definite. Looking at the overdamped Langevin dynamics~\eqref{eq:overdamped_langevin_diffusion}, the diffusion coefficient defined in~\eqref{eq:rmhmc_anisotropic_diffusion_coefficient} helps to project the noise onto the tangent plane to the circle of center 0 and radius 1. In the extended phase-space setting of (G)HMC, this is equivalent to drawing momenta which are essentially tangent to the circle.

We compare the speeds of convergence towards equilibrium of two Markov chains: one with the diffusion coefficient as defined in~\eqref{eq:rmhmc_anisotropic_diffusion_coefficient}, and one with the diffusion coefficient
\begin{equation}
  \label{eq:rmhmc_isotropic_diffusion_coefficient}
  D^{\iso}(q)=(1+\varepsilon)\rmI_2.  
\end{equation}
Both diffusion coefficients are normalized in order to have the same spectral radius~$1+\varepsilon$. To analyze the results, we consider the polar coordinates~$(x,y)=(r\cos\theta,r\sin\theta)$. As the potential exhibits a spherical symmetry, the random variable~$\theta$ should be uniformly distributed on~$[0,2\pi)$. In order to measure convergence, we thus compute the total variation distance between the empirical distribution on~$\theta$ and the uniform distribution on~$[0,2\pi)$ after a fixed number of iterations of the dynamics, and average over a given number of simulations, each time starting the dynamics at the point $q=(0,1)$. A small total variation indicates that the dynamics explores faster the annulus. In practice, the total variation is approximated by discretizing the interval~$[0,2\pi)$ into~$N_{\bins}$ bins of equal sizes, and comparing the probabilities of each bin. More precisely, for~$1\leqslant i\leqslant N_{\bins}$, we denote by~$B_i=[(i-1)/N_{\bins},i/N_{\bins})$ the~$i$-th bin. The total variation distance between two probability measures~$\mu,\nu$, namely~$\TV(\mu,\nu)=\frac{1}{2}\int\rmd\left\lvert\mu-\nu\right\rvert$, is numerically approximated as
\begin{equation*}
  \TV_{N_{\bins}}(\mu,\nu)=\frac{1}{2}\frac{2\pi}{N_{\bins}}\sum_{i=1}^{N_{\bins}}\left\lvert\widehat{\mu}_{i}-\widehat{\nu}_{i}\right\rvert,
\end{equation*}
where~$\widehat{\mu}_i,\widehat{\nu}_i\geqslant0$ are proportional to the fraction of points in the~$i$-th bin, and normalized as
\begin{equation*}
  \frac{2\pi}{N_{\bins}}\sum_{i=1}^{N_{\bins}}\widehat{\mu}_i=\frac{2\pi}{N_{\bins}}\sum_{i=1}^{N_{\bins}}\widehat{\nu}_i=1.
\end{equation*}
In particular, for the uniform law on~$[0,2\pi)$, it holds~$\widehat{\mu}_1=\dots=\widehat{\mu}_{N_\bins}=1/(2\pi)$. Note that the total variation distance is bounded by 1 with our convention.

We run the simulations for 32 values of~$\Delta t$, evenly spread log-wise on~$[10^{-3},2]$, the number of iterations being fixed to~$N_{\iter}=10^{6}$ for each value of~$\Delta t$. Each simulation is performed~$N_{\real}=10^{3}$ times. The approximation of the total variation distances relies on~$N_{\bins}=100$. The other parameters are the same as in Section~\ref{subsec:one_dimensional_confining_double_well_potential}. The results are presented in Figure~\ref{fig:circle_results}, where we plot the mean total variation distance for each time step. The standard deviation, quantified using the variance estimated over the independent~$N_{\real}$ realizations, is not shown in the plot as the error bars are smaller than the symbol sizes.

\begin{figure}
  \centering
  \includegraphics[scale=0.7]{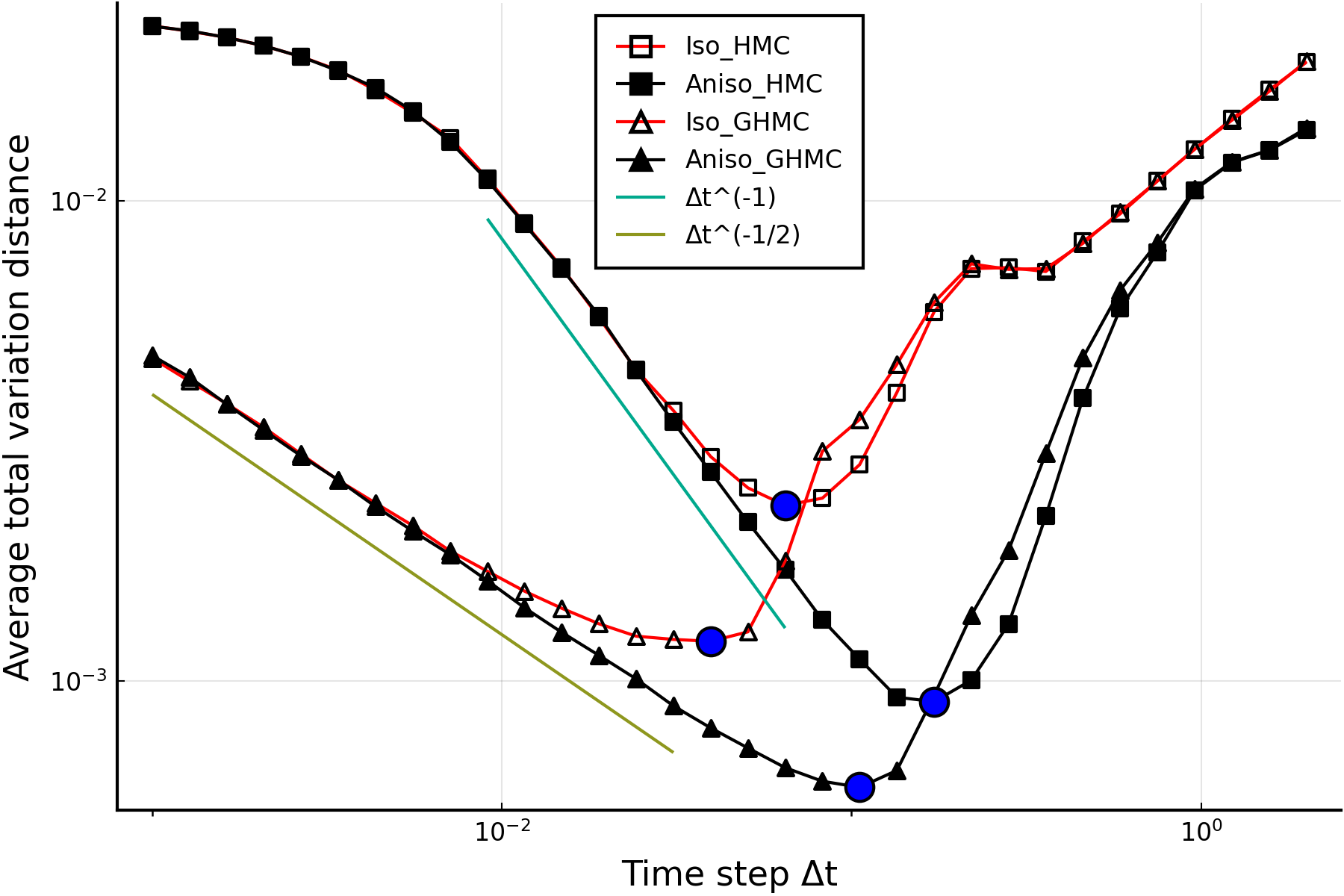}
  \caption{Comparison between the isotropic and anisotropic diffusion coefficients in log-log scale, using either Algorithms~\ref{alg:hmc_scheme_rev_check} (HMC) or~\ref{alg:ghmc_scheme_rev_check} (GHMC). ``Aniso'' corresponds to the diffusion coefficient defined by~\eqref{eq:rmhmc_anisotropic_diffusion_coefficient}, ``Iso'' to the one defined by~\eqref{eq:rmhmc_isotropic_diffusion_coefficient}. The red circles correspond to the minima of each curve, \emph{i.e.}~the point (optimal time step, minimum average TV distance). Minimum value/$\Delta t$ for Iso\_HMC:~$9.17\times 10^{-2}$,~$\Delta t = 6.46\times 10^{-2}$, Aniso\_HMC:~$3.58\times 10^{-2}$,~$\Delta t = 1.72\times 10^{-1}$, for Iso\_GHMC:~$4.78\times 10^{-2}$,~$\Delta t = 3.96\times 10^{-2}$, for Aniso\_GHMC:~$2.37\times10^{-2}$,~$\Delta t = 1.05\times 10^{-1}$.}
  \label{fig:circle_results}
\end{figure}

The results show that the optimal time step, \emph{i.e.}~the one which minimizes in average the total variation distance, yields a lower error of the empirical angle distribution with respect to the uniform distribution on~$[0,2\pi)$ when the anisotropic diffusion coefficient~\eqref{eq:rmhmc_anisotropic_diffusion_coefficient} is chosen. The error made is 2 times smaller in the HMC case and 2.5 times smaller in the GHMC case. Moreover, for a given algorithm, the optimal time step using the anisotropic diffusion coefficient is larger than the one with the isotropic diffusion coefficient. This is in agreement with the fact that larger moves are being accepted as the noise is projected back onto the circle.

Note also that the optimal minimum error corresponding to the GHMC method is lower than the one corresponding to HMC. Hence, the additional computation cost arising from integrating the Ornstein--Uhlenbeck process (namely steps~\ref{step:ghmc_1} and~\ref{step:ghmc_5}) is compensated by a faster exploration of the configuration space.

For small time steps, the total variation distance scales as~$\Delta t^{-1}$ for the HMC algorithm and as~$\Delta t^{-1/2}$ for the GHMC algorithm. This is consistent with the facts that:
\begin{enumerate}[label=(\roman*)]
  \item at the continuous level, the error is of order~$1/\sqrt{t_{\rmsim}}$ where~$t_{\rmsim}$ is the simulation time (by a standard central limit theorem for continuous time Markov processes);
  \item GHMC is formally a weakly consistent discretization of the Langevin dynamics~\eqref{eq:langevin_dynamics} (see Section~\ref{subsubsec:consistency_GHMC_Langevin_dynamics}), so that in this case~$t_{\rmsim}=N_{\iter}\Delta t$;
  \item HMC is a weakly consistent discretization of order 1 of the overdamped Langevin dynamics with a position-dependent diffusion coefficient~\eqref{eq:overdamped_langevin_diffusion} and effective time step~$h=\Delta t^{2}/2$ (see Proposition~\ref{prop:one_step_hmc_weakly_consistent_overdamped_langevin}), so that in this case~$t_{\rmsim}=N_{\iter}h$.
\end{enumerate}  
In this small time step regime, we also observe that there is no difference between using the isotropic or the anisotropic diffusion coefficient: the simulation time is not long enough to see the benefits from using an anisotropic diffusion coefficient. In fact, we see that the total variation is very close to 1 for the HMC algorithm, hence a saturation effect emerges: the average error simply converges to the total variation distance between the initial condition and the target probability measure. This saturation is less pronounced in GHMC thanks to the momentum variable, which allows for an initial ballistic exploration.

For large time steps, we observe that there is no difference between HMC and GHMC. Indeed, the proposal obtained by integrating (or attempting to integrate) the Hamiltonian dynamics is almost always rejected. Eventually, the algorithm only resamples the momenta, either by direct sampling in the HMC case (see step~\ref{step:hmc_1}) or by partial resampling in the GHMC case (see steps~\ref{step:ghmc_1}--\ref{step:ghmc_5}).

%---------------------------------------------------------
%---------------------------------------------------------
%---------------------------------------------------------
%---------------------------------------------------------
%---------------------------------------------------------
%---------------- PROOFS ---------------------------------
%---------------------------------------------------------
%---------------------------------------------------------
%---------------------------------------------------------
%---------------------------------------------------------
%---------------------------------------------------------
\section{Proof of various technical results}
\label{sec:proofs}
We provide in this section the proof of various technical results stated in Sections~\ref{subsec:regularity_numerical_solver_local_preservation_lebesgue_measure},~\ref{subsec:enforcing_S_reversibility},~\ref{subsec:RMHMC_overdamped_Langevin},~\ref{subsec:numerical_solvers_fixed_point} and~\ref{subsec:numerical_solver_newton_method_ift}.

%---------------------------------------------------------
%---------------------------------------------------------
%---------------- IMR/GSV SYMPLECTIC ---------------------
%---------------------------------------------------------
%---------------------------------------------------------
\subsection{Proof of Proposition~\ref{prop:numerical_flow_symplectic}}
\label{subsec:proof_symplectic}

\paragraph{The IMR numerical scheme.}
Using the definition of the IMR numerical scheme~\eqref{eq:Phi_IMR}, we readily compute
\begin{equation*}
  \forall (x,x_1)\in\calX^{2},\qquad 
  \nabla_{x}\Phi_{\Delta t}^{\IMR}(x,x_1)=\rmI_d-A(x,x_1),\qquad
  \nabla_{y}\Phi_{\Delta t}^{\IMR}(x,x_1)=\rmI_d+A(x,x_1),
\end{equation*}
where~$A(x,x_1)=\frac{\Delta t}{2}J\nabla^{2}H(x,x_1)$ with~$J$ the matrix defined in~\eqref{eq:hamiltonian_dynamics} and
\begin{equation*}
  \nabla^{2}H(x,x_1)=\begin{pmatrix}
    \nabla_q^{2}H\left(\frac{x+x_1}{2}\right) & \nabla_{p,q}^{2}H\left(\frac{x+x_1}{2}\right)\\[0.2cm]
    \nabla_{q,p}^{2}H\left(\frac{x+x_1}{2}\right) & \nabla_p^{2}H\left(\frac{x+x_1}{2}\right)
  \end{pmatrix}.
\end{equation*}
Note that~$\nabla^{2}H(x,x_1)$ is symmetric since~$H$ is~$\calC^{2}$. Let~$x\in\calA_{\Delta t}$. Since~$\nabla_{y}\Phi_{\Delta t}^{\IMR}\left(x,\varphi_{\Delta t}^{\IMR}(x)\right)$ is invertible, the matrix~$\rmI_d+A\left(x,\varphi_{\Delta t}^{\IMR}(x)\right)$ is invertible. In the remainder of the proof, we drop the arguments for readability: the argument of the matrix~$A$ is fixed to~$(x,\varphi_{\Delta t}^{\IMR}(x))$. Using the fact that~$J^{2}=-\rmI_d$,~$J^{-1}=J^{\sfT}=-J$,~$JA^{\sfT}=-AJ$ and~$A^{\sfT}J=-JA$, it holds~$J(\rmI_d-A)=(\rmI_d+A^{\sfT})J$. Thus, the matrix~$\rmI_d-A^{\sfT}$ is invertible and
\begin{equation*}
  (\rmI_d-A)^{-\sfT}J=J(\rmI_d+A)^{-1}.
\end{equation*}
By Lemma~\ref{lem:numerical_solver_C1}, it therefore holds
\begin{align*}
  \nabla\varphi_{\Delta}(x)^{\sfT}J\nabla\varphi_{\Delta t}(x)
  &=
  (\rmI_d+A)^{\sfT}(\rmI_d-A)^{-\sfT}J(\rmI_d-A)^{-1}(\rmI_d+A)\\
  &=
  (\rmI_d+A^{\sfT})J(\rmI_d+A)^{-1}(\rmI_d-A)^{-1}(\rmI_d+A)\\
  &=
  J(\rmI_d-A)(\rmI_d+A)^{-1}(\rmI_d-A)^{-1}(\rmI_d+A)=J,
\end{align*}
where we used the fact that the matrices~$(\rmI_d-A)^{-1}$ and~$(\rmI_d+A)^{-1}$ commute for the last equality. This shows that the numerical flow~$\varphi_{\Delta t}^{\IMR}=\chi_{\Delta t}^{\IMR}$ is symplectic on its domain.

\paragraph{The GSV numerical scheme.} Since the composition of symplectic methods is symplectic, and the GSV numerical scheme can be seen as the composition of the Euler~B and Euler~A symplectic methods, it should not be surprising that it is also symplectic. However, the proof still requires some care as we defined numerical flows on subsets of the whole configuration space and constructed the GSV numerical scheme by stacking the two Euler methods into a single vector (and not composing them). Therefore, one has to check that a numerical solver for the GSV numerical scheme allows us to define numerical flows for the Euler~B and Euler~A schemes, and that we still recover a symplectic numerical flow with our construction.

Using the definition of the GSV numerical scheme~\eqref{eq:Phi_GSV}, we readily compute (recall that differentiating with respect to~$y$ corresponds to computing the gradient with respect to the last~$k$ coordinates):
\begin{equation}
  \label{eq:GSV_nabla_y}
  \forall (x,x_1,x_2)\in\calX^{3},\qquad
  \left\lbrace
  \begin{aligned}
    &\nabla_{x}\Phi_{\Delta t}^{\GSV}(x,x_1,x_2)=\begin{pmatrix}
      \nabla_x\Phi_{\Delta t/2}^{\EulerB}(x,x_1)\\
      0_d
    \end{pmatrix}\in\bbR^{2d\times d},\\
    &\nabla_{y}\Phi_{\Delta t}^{\GSV}(x,x_1,x_2)=\begin{pmatrix}
      \nabla_{y}\Phi_{\Delta t/2}^{\EulerB}(x,x_1) & 0_{d}\\
      \nabla_{x}\Phi_{\Delta t/2}^{\EulerA}(x_1,x_2) & \nabla_{y}\Phi_{\Delta t/2}^{\EulerA}(x_1,x_2)
    \end{pmatrix}\in\bbR^{2d\times 2d}.
  \end{aligned}
  \right.
\end{equation}
Let~$\chi_{\Delta t}^{\GSV}\colon\calA_{\Delta t}^{\GSV}\to\calX^{2}$ be a numerical solver for the GSV numerical scheme. For~$x\in\calA_{\Delta t}^{\GSV}$, if~$\chi_{\Delta t}(x)=(x_1,x_2)$, we define~$\varphi_{\Delta t/2}^{\EulerB}(x)=x_1$ and~$\varphi_{\Delta t/2}^{\EulerA}(x_1)=x_2$. The map~$\varphi_{\Delta t/2}^{\EulerB}\colon\calA_{\Delta t}^{\GSV}\to\calX$ (respectively the map~$\varphi_{\Delta t/2}^{\EulerA}\colon\varphi_{\Delta t/2}^{\EulerB}\left(\calA_{\Delta t}^{\GSV}\right)\to\calX$) defines a numerical flow for the numerical scheme~$\Phi_{\Delta t/2}^{\EulerB}$ (respectively for the numerical scheme~$\Phi_{\Delta t/2}^{\EulerA}$). Indeed, if~$x\in\calA_{\Delta t}^{\GSV}$, then it holds
\begin{equation*}
  \Phi_{\Delta t}^{\GSV}\left(x,\varphi_{\Delta t/2}^{\EulerB}(x),\varphi_{\Delta t/2}^{\EulerA}\circ\varphi_{\Delta t/2}^{\EulerB}(x)\right)=0,  
\end{equation*}
which implies that~$\Phi_{\Delta t/2}^{\EulerB}\left(x,\varphi_{\Delta t/2}^{\EulerB}(x)\right)=0$ and~$\Phi_{\Delta t/2}^{\EulerA}\left(\varphi_{\Delta t/2}^{\EulerB}(x),\varphi_{\Delta t/2}^{\EulerA}\circ\varphi_{\Delta t/2}^{\EulerB}(x)\right)=0$. Moreover, the second equality in~\eqref{eq:GSV_nabla_y} shows that the matrix~$\nabla_{y}\Phi_{\Delta t}^{\GSV}\left(x,\chi_{\Delta t}^{\GSV}(x)\right)$ is invertible if and only if the matrices~$\nabla_{y}\Phi_{\Delta t/2}^{\EulerB}\left(x,\varphi_{\Delta t/2}^{\EulerB}(x)\right)$ and~$\nabla_{y}\Phi_{\Delta t/2}^{\EulerA}\left(\varphi_{\Delta t/2}^{\EulerB}(x),\varphi_{\Delta t/2}^{\EulerA}\circ\varphi_{\Delta t/2}^{\EulerB}(x)\right)$ are invertible. Finally, notice that~$\varphi_{\Delta t/2}^{\EulerB}\left(\calA_{\Delta t}^{\GSV}\right)$ is open and nonempty as a simple consequence of the implicit function theorem.

We now compute the gradient of~$\chi_{\Delta t}^{\GSV}$ using Lemma~\ref{lem:numerical_solver_C1}. We drop the use of arguments of functions for readability: we fix~$x\in\calA_{\Delta t}^{\GSV}$, so that~$\left(x,\varphi_{\Delta t/2}^{\EulerB}(x),\varphi_{\Delta t/2}^{\EulerA}(x)\right)$ is the argument of~$\Phi_{\Delta t}^{\GSV}$,~$\left(x,\varphi_{\Delta t/2}^{\EulerB}(x)\right)$ of~$\Phi_{\Delta t/2}^{\EulerB}$, $\left(\varphi_{\Delta t/2}^{\EulerB}(x),\varphi_{\Delta t/2}^{\EulerA}\circ\varphi_{\Delta t/2}^{\EulerB}(x)\right)$ of~$\Phi_{\Delta t/2}^{\EulerA}$, and~$x$ of~$\chi_{\Delta t}^{\GSV}$,~$\varphi_{\Delta t/2}^{\EulerB}$ and~$\varphi_{\Delta t/2}^{\EulerA}$.

First, notice that with straightforward computations, it holds
\begin{equation}
  \label{eq:GSV_nabla_y_inverse}
  \left(
    \nabla_{y}\Phi_{\Delta t}^{\GSV}
  \right)^{-1}=
  \begin{pmatrix}
    \left(
      \nabla_y\Phi_{\Delta t/2}^{\EulerB}
    \right)^{-1}&0_d\\
    0_d & \left(
      \nabla_{y}\Phi_{\Delta t/2}^{\EulerA}
    \right)^{-1}
  \end{pmatrix}
  \begin{pmatrix}
    \rmI_d & 0_d\\
    -\nabla_{x}\Phi_{\Delta t/2}^{\EulerA}\left(
      \nabla_{y}\Phi_{\Delta t/2}^{\EulerB}
    \right)^{-1} & \rmI_d
  \end{pmatrix}.
\end{equation}
Then multiplying~\eqref{eq:GSV_nabla_y_inverse} with the first equation in~\eqref{eq:GSV_nabla_y}, one obtains
\begin{align*}
  \nabla\chi_{\Delta t}^{\GSV}
  &=
  -
  \begin{pmatrix}
    \left(
      \nabla_y\Phi_{\Delta t/2}^{\EulerB}
    \right)^{-1}&0_d\\
    0_d & \left(
      \nabla_{y}\Phi_{\Delta t/2}^{\EulerA}
    \right)^{-1}
  \end{pmatrix}
  \begin{pmatrix}
    \nabla_{x}\Phi_{\Delta t/2}^{\EulerB}\\
    -\nabla_{x}\Phi_{\Delta t/2}^{\EulerA}\left(\nabla_{y}\Phi_{\Delta t/2}^{\EulerB}\right)^{-1}\nabla_x\Phi_{\Delta t/2}^{\EulerB}
  \end{pmatrix}\\
  &=
  -\begin{pmatrix}
    \left(\nabla_y\Phi_{\Delta t/2}^{\EulerB}\right)^{-1}\nabla_x\Phi_{\Delta t/2}^{\EulerB}\\
    -\left(\nabla_y\Phi_{\Delta t/2}^{\EulerA}\right)^{-1}\nabla_x\Phi_{\Delta t/2}^{\EulerA}\left(\nabla_y\Phi_{\Delta t/2}^{\EulerB}\right)^{-1}\nabla_x\Phi_{\Delta t/2}^{\EulerB}
  \end{pmatrix}\\
  &=
  \begin{pmatrix}
    \nabla\varphi_{\Delta t/2}^{\EulerB}\\
    \nabla\varphi_{\Delta t/2}^{\EulerA}\nabla\varphi_{\Delta t/2}^{\EulerB}
  \end{pmatrix},
\end{align*}
where we used Lemma~\ref{lem:numerical_solver_C1} for the last equality. In particular,
\begin{equation*}
  %\label{eq:GSV_composition_EulerB_EulerA}
  \nabla\varphi_{\Delta t}^{\GSV}=\nabla\varphi_{\Delta t/2}^{\EulerA}\nabla\varphi_{\Delta t/2}^{\EulerB}.
\end{equation*}
Since the two numerical flows~$\varphi_{\Delta t/2}^{\EulerB}$ and~$\varphi_{\Delta t/2}^{\EulerA}$ are symplectic on their domains (using similar arguments as above and~\cite[Theorem~VI.3.3]{hairer_2006}), it holds
\begin{align*}
  \left(\nabla\varphi_{\Delta t}^{\GSV}\right)^{\sfT}J\nabla\varphi_{\Delta t}
  &=
  \left(\nabla\varphi_{\Delta t/2}^{\EulerB}\right)^{\sfT}\left(\nabla\varphi_{\Delta t/2}^{\EulerA}\right)^{\sfT}J\nabla\varphi_{\Delta t/2}^{\EulerA}\nabla\varphi_{\Delta t/2}^{\EulerB}\\
  &=
  \left(\nabla\varphi_{\Delta t/2}^{\EulerB}\right)^{\sfT}J\nabla\varphi_{\Delta t/2}^{\EulerB}=J.
\end{align*}
This shows that the numerical flow~$\varphi_{\Delta t}^{\GSV}$ is symplectic on its domain.

%---------------------------------------------------------
%---------------------------------------------------------
%---------------- B OPEN SET -----------------------------
%---------------------------------------------------------
%---------------------------------------------------------
\subsection{Proof of Lemma~\ref{lem:B_open_set}}
\label{subsec:proof:B_open_set}
This proof is an adaptation of the proof of~\cite[Lemma~2.3]{lelievre_2019}. As both~$\varphi_{\Delta t}$ and~$S$ are continuous and~$\calA_{\Delta t}$ is open,~$\calA_{\Delta t}\bigcap\psi_{\Delta t}^{-1}(\calA_{\Delta t})$ is an open set of~$\calX$. Let us prove~\eqref{eq:reversibility_open_set} by contradiction. Consider a non-empty path connected component~$C$ of~$\calA_{\Delta t}\bigcap\psi_{\Delta t}^{-1}(\calA_{\Delta t})$, an element~$X_{0}\in C$ such that
\begin{equation*}
  \chi_{\Delta t}\circ\psi_{\Delta t}
  (X_{0})=
  \Bigl(
    S\circ\chi_{\Delta t,k-1}(X_{0}),\dots,S\circ\chi_{\Delta t,1}(X_{0}),S(X_{0})
  \Bigr). 
\end{equation*}
and assume that there exists~$X_{1}\in C$ such that
\begin{equation}
  \label{eq:contradiction_x_1_B_open_set}
  \chi_{\Delta t}\circ\psi_{\Delta t}
  (X_{1})\neq
  \Bigl(
    S\circ\chi_{\Delta t,k-1}(X_{1}),\dots,S\circ\chi_{\Delta t,1}(X_{1}),S(X_{1})
  \Bigr).
\end{equation}
Let us connect the elements~$X_{0}$ and~$X_{1}$ in~$C$ by a continuous path~$[0,1]\ni\theta\mapsto X_{\theta}$. Define 
\begin{equation*}
    \theta_{\star}=\sup\left\lbrace
        \theta\geqslant0\,\middle|\,
        \begin{aligned}
          &\forall s\in[0,\theta],\\
          &\chi_{\Delta t}\circ\psi_{\Delta t}(X_{s})=
          \Bigl(
            S\circ\chi_{\Delta t,k-1}(X_{s}),\dots,S\circ\chi_{\Delta t,1}(X_{s}),S(X_{s})
          \Bigr)
        \end{aligned}
    \right\rbrace.
\end{equation*}
By continuity of~$\chi_{\Delta t}$,~$S$ and~$\theta\mapsto X_{\theta}$, the value~$\theta_{\star}$ is actually a maximum, and it holds~$\theta_{\star}<1$ by~\eqref{eq:contradiction_x_1_B_open_set}. By definition of the supremum, there exists a sequence~$(\theta_{n})_{n\geqslant0}$ with values in~$(\theta_{\star},1)$ such that~$\theta_{n}\to\theta_{\star}$ as~$n\to+\infty$ and 
\begin{equation*}
    \forall n\geqslant0,\qquad \chi_{\Delta t}\circ\psi_{\Delta t}(X_{\theta_n})\neq
    \Bigl(
      S\circ\chi_{\Delta t,k-1}(X_{\theta_n}),\dots,S\circ\chi_{\Delta t,1}(X_{\theta_n}),S(X_{\theta_n})
    \Bigr).
\end{equation*}
We can therefore define the~$dk$-dimensional unit vectors 
\begin{equation}
  \label{eq:tau_n}
    \tau_{n}=\frac{
      \chi_{\Delta t}\circ\psi_{\Delta t}(X_{\theta_n})-
      \Bigl(
      S\circ\chi_{\Delta t,k-1}(X_{\theta_n}),\dots,S\circ\chi_{\Delta t,1}(X_{\theta_n}),S(X_{\theta_n})
      \Bigr)
      }
      {\left\lVert \chi_{\Delta t}\circ\psi_{\Delta t}(X_{\theta_n})-
      \Bigl(
      S\circ\chi_{\Delta t,k-1}(X_{\theta_n}),\dots,S\circ\chi_{\Delta t,1}(X_{\theta_n}),S(X_{\theta_n})
      \Bigr)\right\rVert
      },
\end{equation}
since the denominator is positive. By compactness of the~$dk$-dimensional unit ball and upon extraction (keeping the same notation for the extracted subsequence), there exists a subsequence~$(\theta_n)_{n\geqslant 0}$ and a unit vector~$\tau_{\star}$ such that~$\tau_{n}\to\tau_{\star}$ as~$n\to+\infty$. In particular, it holds~$\left\lVert\tau_{\star}\right\rVert=1$.
We can then rewrite~\eqref{eq:tau_n} as
\begin{equation*}
  \chi_{\Delta t}\circ\psi_{\Delta t}(X_{\theta_n})=
  \Bigl(
    S\circ\chi_{\Delta t,k-1}(X_{\theta_n}),\dots,S\circ\chi_{\Delta t,1}(X_{\theta_n}),S(X_{\theta_n})
  \Bigr)+\varepsilon_n\tau_n,  
\end{equation*}
with~$\left\lVert\tau_n\right\rVert=1$ and
\begin{equation*}
  \varepsilon_n=
  \left\lVert\chi_{\Delta t}\circ\psi_{\Delta t}(X_{\theta_n})-
  \Bigl(
  S\circ\chi_{\Delta t,k-1}(X_{\theta_n}),\dots,S\circ\chi_{\Delta t,1}(X_{\theta_n}),S(X_{\theta_n})
  \Bigr)\right\rVert\xrightarrow[n\to+\infty]{}0
\end{equation*}
by continuity of the mapping~$\theta\mapsto X_{\theta}$ and the fact that the latter quantity vanishes at~$\theta_{\star}$. Using the fact that~$\psi_{\Delta t}(X_{\theta_n})\in\calA_{\Delta t}$ and~\eqref{eq:phi_chi_relation}, and then performing a first order Taylor expansion on~$\Phi_{\Delta t}$,
\begin{equation}
  \label{eq:DL_PHI}
  \begin{aligned}
    0&=\Phi_{\Delta t}\left(\psi_{\Delta t}(X_{\theta_n}),\chi_{\Delta t}\circ\psi_{\Delta t}(X_{\theta_n})\right),\\
    &=
    \Phi_{\Delta t}\left(\psi_{\Delta t}(X_{\theta_n}),
    \Bigl(
      S\circ\chi_{\Delta t,k-1}(X_{\theta_n}),\dots,S\circ\chi_{\Delta t,1}(X_{\theta_n}),S(X_{\theta_n})
    \Bigr)\right)\\
    &\quad+\varepsilon_n
    \nabla_{y}\Phi_{\Delta t}
    \left(\psi_{\Delta t}(X_{\theta_n}),
    \Bigl(
      S\circ\chi_{\Delta t,k-1}(X_{\theta_n}),\dots,S\circ\chi_{\Delta t,1}(X_{\theta_n}),S(X_{\theta_n})
    \Bigr)
    \right)
    \tau_n\\
    &\quad+\rmO\left(\varepsilon_n\right).
  \end{aligned}
\end{equation}
The first term on the right-hand side of~\eqref{eq:DL_PHI} vanishes. Indeed, from the~$S$-reversibility of~$\Phi_{\Delta t}$ and~\eqref{eq:phi_chi_relation},
\begin{align*}
  \MoveEqLeft[5]
  \Phi_{\Delta t}\left(\psi_{\Delta t}(X_{\theta_n}),
  \Bigl(
    S\circ\chi_{\Delta t,k-1}(X_{\theta_n}),\dots,S\circ\chi_{\Delta t,1}(X_{\theta_n}),S(X_{\theta_n})
  \Bigr)\right)\\
  &=
  \Phi_{\Delta t}\Bigl(
    X_{\theta_n},
    S\circ S\circ\chi_{\Delta t,1}(X_{\theta_n}),\dots,S\circ S\circ\chi_{\Delta t,k-1}(X_{\theta_n}),S\circ S\circ\varphi_{\Delta t}(X_{\theta_n})
  \Bigr)\\
  &=
  \Phi_{\Delta t}\Bigl(X_{\theta_n},\chi_{\Delta t}(X_{\theta_n})\Bigr)
  =
  0.
\end{align*}
Dividing~\eqref{eq:DL_PHI} by~$\varepsilon_n\neq0$, and taking the limit as~$n\to+\infty$, we obtain
\begin{equation}
  \label{eq:contradiction}
  \nabla_{y}\Phi_{\Delta t}
  \Bigl(\psi_{\Delta t}(X_{\theta_{\star}}), S\circ\chi_{\Delta t,k-1}(X_{\theta_\star}),\dots,S\circ\chi_{\Delta t,1}(X_{\theta_\star}),S(X_{\theta_\star})
  \Bigr)
  \tau_{\star}=0.
\end{equation}
Now, the matrix~$\nabla_{y}\Phi_{\Delta t}(\psi_{\Delta t}(X_{\theta_{\star}}), S\circ\chi_{\Delta t,k-1}(X_{\theta_\star}),\dots,S\circ\chi_{\Delta t,1}(X_{\theta_\star}),S(X_{\theta_\star}))$ is invertible as, from the definition of~$\theta_{\star}$,
\begin{align*}
  &\nabla_{y}\Phi_{\Delta t}(\psi_{\Delta t}(X_{\theta_{\star}}),
    S\circ\chi_{\Delta t,k-1}(X_{\theta_\star}),\dots,S\circ\chi_{\Delta t,1}(X_{\theta_\star}),S(X_{\theta_\star}))\phantom{\tau_{\star}=0.}\\
  &\qquad
  =\left(
    \nabla_{y}\Phi_{\Delta t}
    \right)\left(\psi_{\Delta t}(X_{\theta_\star}),\chi_{\Delta t}\circ\psi_{\Delta t}(X_{\theta_\star})\right),
\end{align*}
the latter matrix being invertible by the fact that~$\psi_{\Delta t}(X_{\theta_{\star}})\in\calA_{\Delta t}$ and the definition of~$\calA_{\Delta t}$. As~$\tau_{\star}$ is a nonzero vector ($\left\lVert\tau_{\star}\right\rVert=1$), Equation~\eqref{eq:contradiction} leads to a contradiction. This shows that there is no element~$X_{1}\in C$ such that
\begin{equation*}
  \chi_{\Delta t}\circ\psi_{\Delta t}(X_{1})\neq
  \Bigl(
    S\circ\chi_{\Delta t,k-1}(X_{1}),\dots,S\circ\chi_{\Delta t,1}(X_{1}),S(X_{1})
  \Bigr),
\end{equation*}
which concludes the proof of~\eqref{eq:reversibility_open_set}.

%---------------------------------------------------------
%---------------------------------------------------------
%---------------- GSV/IMR WEAK CONSISTENCY ---------------
%---------------------------------------------------------
%---------------------------------------------------------
\subsection{Proof of Lemma~\ref{lem:one_step_hmc_weak_consistency_overdamped_langevin}}
\label{subsec:proof:HMC_weak_consistency}

We use the following notation: for~$q\in\calO$ and for any~$\xi_1,\xi_2\in\bbR^{m}$,
\begin{align*}
    &\nabla D(q)=\left(\frac{\partial D}{\partial q_i}(q)\right)_{1\leqslant i\leqslant m}\in\bbR^{m^{2}\times m},\\
    &\xi_{1}^{\sfT}\nabla D(q)\xi_2=\left(\xi_1^{\sfT}\frac{\partial D}{\partial q_i}(q)\xi_2\right)_{1\leqslant i\leqslant m}\in\bbR^{m},\\
    &\Tr\left(D(q)^{-1}\nabla D(q)\right)=\left(\Tr\left(D(q)^{-1}\frac{\partial D}{\partial q_i}(q)\right)\right)_{1\leqslant i\leqslant m}\in\bbR^{m},\\
    &F(q)=-\nabla V(q)+\frac{1}{2}\Tr\left(D(q)^{-1}\nabla D(q)\right)\in\bbR^{m}.
\end{align*}
When we write a remainder term as~$c(q,p)=\rmO(\Delta t^{\delta})$ (where~$c(q,p)$ is either a vector or a matrix), it means that there exist~$\Delta t_{\star}>0$ and~$M_{\alpha}\in\bbR_{+}$ (which depend on~$(q,p)$) such that
\begin{equation}
  \label{eq:remainder}
  \forall\Delta t\in(0,\Delta t_{\star}],\qquad
  \left\lVert c(q,p)\right\rVert\leqslant M_{\alpha}\Delta t^{\alpha}.
\end{equation}
In all this section, we fix~$(q^{n},p^{n})\in\calA_{\Delta t}$.

\subsubsection{GSV numerical scheme}
%\label{subsubsec:GSV}

As~$\nabla\left(\det\,D\right)=\left(\det\,D\right)\Tr\left(D^{-1}\nabla D\right)$, the GSV numerical scheme for the Hamiltonian function~\eqref{eq:H_RMHMC} writes, since~$(q^{n},p^{n})\in\calA_{\Delta t}$,
\begin{equation*}
    \left\lbrace
        \begin{aligned}
            \widetilde{p}^{n+1/2} &= p^{n}-\dfrac{\Delta t}{2}\left(-F(q^{n})+\dfrac{1}{2}\left(\widetilde{p}^{n+1/2}\right)^{\sfT}\nabla D(q^{n})\widetilde{p}^{n+1/2}\right),\\
            \widetilde{q}^{n+1} &= q^{n}+\dfrac{\Delta t}{2}\left(D\left(\widetilde{q}^{n+1}\right)+D(q^{n})\right)\widetilde{p}^{n+1/2},\\
            \widetilde{p}^{n+1} &= \widetilde{p}^{n+1/2}-\dfrac{\Delta t}{2}\left(-F\left(\widetilde{q}^{n+1}\right)+\dfrac{1}{2}\left(\widetilde{p}^{n+1/2}\right)^{\sfT}\nabla D\left(\widetilde{q}^{n+1}\right)\widetilde{p}^{n+1/2}\right).
        \end{aligned}
    \right.
\end{equation*}
Therefore,
\begin{align}
    \widetilde{q}^{n+1}
    &=
    q^{n}+\dfrac{\Delta t}{2}\left(
        D(q^{n})+D\left(\widetilde{q}^{n+1}\right)
    \right)\widetilde{p}^{n+1/2}\nonumber\\
    &=
    q^{n}+\dfrac{\Delta t}{2}\left(
        D(q^{n})+D\left(\widetilde{q}^{n+1}\right)
    \right)\left[
        p^{n}-\dfrac{\Delta t}{2}\left(
            -F(q^{n})
            +\dfrac{\left(\widetilde{p}^{n+1/2}\right)^{\sfT}\nabla D(q^{n})\widetilde{p}^{n+1/2}}{2}
        \right)
    \right].\label{eq:53}
\end{align}
We expand the latter equality in powers of~$\Delta t$ up to order 4 by consecutively looking at the terms of the right-hand side of~\eqref{eq:53}. We write
\begin{equation}
  \label{eq:expansion_D}
  D\left(\widetilde{q}^{n+1}\right)= D(q^{n})+\nabla D(q^{n})\odot \left(\widetilde{q}^{n+1}-q^{n}\right)+\dfrac{1}{2}\nabla^{2}D(q^{n})\odot\left(\widetilde{q}^{n+1}-q^{n}\right)^{\otimes 2}+\rmO(\Delta t^{3}),
\end{equation}
where we use the following notation: for any~$u\in\bbR^{m}$,
\begin{align*}
    &\nabla D(q)\odot u=\dps{\sum_{i=1}^{m}\dfrac{\partial D}{\partial q_i}(q)u_{i}}\in\bbR^{m\times m},\\
    &\nabla^{2}D(q)\odot u^{\otimes2}=\dps{\sum_{i,j=1}^{m}}\dfrac{\partial^{2}D}{\partial q_{i}\partial q_{j}}(q)u_{i}u_j\in\bbR^{m\times m}.
\end{align*}
Note that the remainder in~\eqref{eq:expansion_D} satisfies~\eqref{eq:remainder}: this follows from Taylor's formula with integral remainder together with the implicit function theorem to obtain estimates on~$\left\lVert\widetilde{q}^{n+1}-q^{n}\right\rVert$ as done in~\cite{hairer_2006}. The latter estimate relies on the fact that~$(\widetilde{q}^{n+1},\widetilde{p}^{n+1})$ is close to~$(q^{n},p^{n})$ (see~\eqref{eq:solution_close} and Assumption~\ref{ass:RMHMC} since~$D$ and its derivatives grow at most polynomially. This reasoning is used everytime we write~$\rmO(\Delta t^{\alpha})$ below.

One thus obtains
\begin{align}
    \dfrac{\Delta t}{2}\left(
        D(q^{n})+D\left(\widetilde{q}^{n+1}\right)
    \right)p^{n}
    &=
    \Delta t~D(q^{n})p^{n}\nonumber\\
    &\quad\label{eq:56}
    +\dfrac{\Delta t^{2}}{4}\left[
        \nabla D(q^{n})\odot\left(\left(
        D(q^{n})+D(\widetilde{q}^{n+1})
    \right)\widetilde{p}^{n+1/2}\right)\right]
    p^{n}\\
    &\quad
    +\dfrac{\Delta t^{3}}{4}\left[\nabla^{2}D(q^{n})\odot\left(D(q^{n})p^{n}\right)^{\otimes2}\right]p^{n}\nonumber\\
    &\quad +\rmO(\Delta t^{4}).\nonumber
\end{align}
Further expanding the second term on the right-hand side of~\eqref{eq:56} we obtain
\begin{align*}
    &\dfrac{\Delta t^{2}}{4}
    \left\lbrace\nabla D(q^{n})\odot\left[
        \left(
            D(q^{n})+D\left(\widetilde{q}^{n+1}\right)
        \right)\widetilde{p}^{n+1/2}
    \right]
    \right\rbrace p^{n}\\
    &=
    \dfrac{\Delta t^{2}}{4}
    \Bigg\lbrace\nabla D(q^{n})\odot
    \left[
        2D(q^{n})+\nabla D(q^{n})\odot\left(\widetilde{q}^{n+1}-q^{n}\right)+\rmO(\Delta t^{2})
    \right]\\
    &\quad\times\left[
        p^{n}-\dfrac{\Delta t}{2}\left(
            -F(q^{n})+\dfrac{\left(\widetilde{p}^{n+1/2}\right)^{\sfT}\nabla D(q^{n})\widetilde{p}^{n+1/2}}{2}
        \right)
    \right]\Bigg\rbrace p^{n},\\
    &
        \begin{aligned}
            =~&\dfrac{\Delta t^{2}}{2}
            \left(
                \nabla D(q^{n})\odot\left[
                    D(q^{n})p^{n}
                \right]
            \right)p^{n}\\
            &
            -\dfrac{\Delta t^{3}}{4}
            \left\lbrace
                \nabla D(q^{n})\odot\left[
                    D(q^{n})\left(
                        -F(q^{n})+\dfrac{\left(p^{n}\right)^{\sfT}\nabla D(q^{n})p^{n}}{2}
                    \right)
                \right]
            \right\rbrace p^{n}\\
            &
            +\dfrac{\Delta t^{3}}{4}
            \Bigg\lbrace
                \nabla D(q^{n})\odot\Big[\big(
                    \nabla D(q^{n})\odot\left[
                        D(q^{n})p^{n}
                    \right]
                \big)p^{n}\Big]
            \Bigg\rbrace p^{n}\\
            &+\rmO(\Delta t^{4}).
        \end{aligned}
\end{align*}
Note that~$\nabla D=\nabla D^{\sfT}$ (since~$D(q)$ is a symmetric matrix) so that
\begin{align*}
    \left(\widetilde{p}^{n+1/2}\right)^{\sfT}\nabla D(q^{n})\widetilde{p}^{n+1/2}
    &=\left(p^{n}\right)^{\sfT}\nabla D(q^{n})p^{n}\\
    &\quad-\Delta t\left(p^{n}\right)^{\sfT}\nabla D(q^{n})\left[
        -F(q^{n})+\dfrac{\left(p^{n}\right)^{\sfT}\nabla D(q^{n})p^{n}}{2}
    \right]+\rmO(\Delta t^{2}).
\end{align*}
The last term in~\eqref{eq:53} can therefore be expanded as
\begin{align}
    \dfrac{\Delta t^{2}}{4}&\left(
        D(q^{n})+D\left(\widetilde{q}^{n+1}\right)
    \right)\left(-F(q^{n})+\dfrac{\left(\widetilde{p}^{n+1/2}\right)^{\sfT}\nabla D(q^{n})\widetilde{p}^{n+1/2}}{2}\right)\nonumber\\
    &=
    \dfrac{\Delta t^{2}}{2}D(q^{n})\left(
        -F(q^{n})+\dfrac{\left(p^{n}\right)^{\sfT}\nabla D(q^{n})p^{n}}{2}
    \right)\nonumber\\
    &\quad
    -\dfrac{\Delta t^{3}}{4}D(q^{n})\left[
        \left(p^{n}\right)^{\sfT}\nabla D(q^{n})\left(
            -F(q^{n})+\dfrac{\left(p^{n}\right)^{\sfT}\nabla D(q^{n})p^{n}}{2}
        \right)
    \right]\nonumber\\
    &\quad
    +\dfrac{\Delta t^{3}}{4}\left[
        \nabla D(q^{n})\odot\left(
            D(q^{n})p^{n}
        \right)
    \right]\left[
        -F(q^{n})+\dfrac{\left(p^{n}\right)^{\sfT}\nabla D(q^{n})p^{n}}{2}
    \right]\label{eq:39}\\
    &\quad+\rmO(\Delta t^{4}).\nonumber
\end{align}
By gathering the previous equalities~\eqref{eq:56} and~\eqref{eq:39} in~\eqref{eq:53},
\begin{align}
    \widetilde{q}^{n+1}
    &=q^{n}+\Delta t\,D(q^{n})p^{n}\nonumber\\
    &\quad
    +\dfrac{\Delta t^{2}}{2}\left\lbrace
        \left[
            \nabla D(q^{n})\odot\left[
                D(q^{n})p^{n}
            \right]
        \right]p^{n}
        -
        D(q^{n})\left[
            -F(q^{n})+\dfrac{\left(p^{n}\right)^{\sfT}\nabla D(q^{n})p^{n}}{2}
        \right]
    \right\rbrace\nonumber\\
    &\quad
    +\Delta t^{3}
    \left[
        \begin{array}[]{l}
            -\left\lbrace\dfrac{\nabla D(q^{n})}{4}\odot\left[
                D(q^{n})\left(
                    -F(q^{n})+\dfrac{\left(p^{n}\right)^{\sfT}\nabla D(q^{n})p^{n}}{2}
                \right)
            \right]\right\rbrace p^{n}\\[0.3cm]
            +\left\lbrace\dfrac{\nabla D(q^{n})}{4}\odot\left[
                \left(\nabla D(q^{n})\odot\left[
                    D(q^{n})p^{n}
                \right]\right)p^{n}
            \right]\right\rbrace p^{n}\\[0.3cm]
            +\dfrac{D(q^{n})}{4}\left\lbrace
                \left(p^{n}\right)^{\sfT}\nabla D(q^{n})\left[
                    -F(q^{n})+\dfrac{\left(p^{n}\right)^{\sfT}\nabla D(q^{n})p^{n}}{2}
                \right]
            \right\rbrace\\[0.3cm]
            +\dfrac{1}{4}\left\lbrace
                \nabla^{2}D(q^{n})\odot\left(D(q^{n})p^{n}\right)^{\otimes2}
            \right\rbrace p^{n}\\[0.3cm]
            -\dfrac{1}{4}\left[
                \nabla D(q^{n})\odot\left[
                    D(q^{n})p^{n}
                \right]
            \right]\left[
                -F(q^{n})+\dfrac{\left(p^{n}\right)^{\sfT}\nabla D(q^{n})p^{n}}{2}
            \right]
        \end{array}
    \right]\nonumber\\
    &\quad+\rmO(\Delta t^{4}),\nonumber\\
    &=q^{n} + \Delta t\,\calF_{1}(q^{n}, p^{n}) + \dfrac{\Delta t^{2}}{2}\calF_{2}(q^{n}, p^{n})+\Delta t^{3}\calF_{3}(q^{n},p^{n})+\rmO(\Delta t^{4}),\label{eq:gsv_expansion_functions}
\end{align}
where the functions~$\calF_1,\calF_2,\calF_3$ are defined by comparing the power expansion in~$\Delta t$ between the first and second equality of~\eqref{eq:gsv_expansion_functions}. The remainder term in~\eqref{eq:gsv_expansion_functions} satisfies~\eqref{eq:expansion_position_rmhmc_bounds}, which shows that~\eqref{eq:expansion_position_rmhmc} holds for the GSV numerical scheme.

Since~$p^{n}\sim\calN(0,D(q^{n})^{-1})$ in step~\ref{step:hmc_1} of Algorithm~\ref{alg:hmc_scheme_rev_check}, we replace the momenta by~$p^{n}=D(q^{n})^{-1/2}\rmG^{n}$ where~$\rmG^{n}\sim\calN(0,\rmI_m)$. Then,
\begin{equation*}
    \left\lbrace
        \begin{aligned}
            \calF_1(q^{n}, D(q^{n})^{-1/2}\rmG^{n})&= D(q^{n})^{1/2}\rmG^{n},\\
            \calF_2(q^{n}, D(q^{n})^{-1/2}\rmG^{n})&= -D(q^{n})\nabla V(q^{n})+\dfrac{D(q^{n})}{2}\Tr\left(
                D(q^{n})^{-1}\nabla D(q^{n})
            \right)\\
            &\phantom{=~}
            -\dfrac{D(q^{n})}{2}\left[
                \left[D(q^{n})^{-1/2}\rmG^{n}\right]^{\sfT}\nabla D(q^{n})\left[D(q^{n})^{-1/2}\rmG^{n}\right]
            \right]\\
            &\phantom{=~}
            +\left(\nabla D(q^{n})\odot\left[D(q^{n})^{1/2}\rmG^{n}\right]\right)D(q^{n})^{-1/2}\rmG^{n}.
        \end{aligned}
    \right.
\end{equation*}
The functions~$\calF_{1}$ and~$\calF_{3}$ being odd in~$p^{n}$ hence in~$\rmG^{n}$, and the function~$\calF_2$ being even in~$p^{n}$ hence in~$\rmG^{n}$, it holds
\begin{equation}
    \label{eq:expectation_odd_f}
    \left\lbrace
    \begin{aligned}
        \bbE_{\rmG^{n}}\left[\calF_{1}(q^{n},D(q^{n})^{-1/2}\rmG^{n})\right]&=0,\\
        \bbE_{\rmG^{n}}\left[\calF_{3}(q^{n},D(q^{n})^{-1/2}\rmG^{n})\right]&=0,\\
        \bbE_{\rmG^{n}}\left[\calF_{1}(q^{n},D(q^{n})^{-1/2}\rmG^{n})\otimes\calF_{2}(q^{n},D(q^{n})^{-1/2}\rmG^{n})\right]&=0,
    \end{aligned}
    \right.
\end{equation}
where here~$\otimes$ simply denotes the outer product.
Finally, to show that~\eqref{eq:computation_f2} holds, let us compute the various terms in~$\bbE_{\rmG^{n}}\left[\calF_2(q^{n},D(q^{n})^{-1/2}\rmG^{n}))\right]$. For~$1\leqslant i\leqslant m$, it holds
\begin{align*}
    &\bbE_{\rmG^{n}}
    \left(
        \left[
            \left\lbrace
                \nabla D(q^{n})\odot
                \left(
                    D(q^{n})^{1/2}\rmG^{n}
                \right)
            \right\rbrace
            D(q^{n})^{-1/2}\rmG^{n}
        \right]_{i}
    \right)\\
    &\quad=
    \bbE_{\rmG^{n}}
    \left(
        \sum_{k=1}^{m}
        \left[
            \nabla D(q^{n})\odot
            \left(
                D(q^{n})^{1/2}\rmG^{n}
            \right)
        \right]_{i,k}
        \left[
            D(q^{n})^{-1/2}\rmG^{n}
        \right]_{k}
    \right)\\
    &\quad=
    \bbE_{\rmG^{n}}
    \left(
        \sum_{k=1}^{m}\sum_{j=1}^{m}\partial_{q_j}
        D_{i,k}(q^{n})
        \left[
            D(q^{n})^{1/2}\rmG^{n}
        \right]_{j}
        \left[
            D(q^{n})^{-1/2}\rmG^{n}
        \right]_{k}
    \right)\\
    &\quad=
    \sum_{k=1}^{m}\sum_{j=1}^{m}\partial_{q_j}D_{i,k}(q^{n})
    \bbE_{\rmG^{n}}
    \left(
        \left[
            D(q^{n})^{1/2}\rmG^{n}
        \right]_{j}
        \left[
            D(q^{n})^{-1/2}\rmG^{n}
        \right]_{k}
    \right)\\
    &\quad=
    \sum_{k=1}^{m}\sum_{j=1}^{m}\partial_{q_j}D_{i,k}(q^{n})
    \underbrace{\sum_{\ell=1}^{m}\left[D(q^{n})^{1/2}\right]_{j,\ell}\left[D(q^{n})^{-1/2}\right]_{k,\ell}}_{=~\left[D(q^{n})^{1/2}D(q^{n})^{-\sfT/2}\right]_{j,k}=~\delta_{j,k}}\\
    &\quad=
    \sum_{k=1}^{m}\partial_{q_{k}}D_{i,k}(q^{n})=\left[\div D(q^{n})\right]_{i}.
\end{align*}
Additionally,
\begin{align*}
    &\bbE_{\rmG^{n}}
    \left(
        \left[
            D(q^{n})
            \left\lbrace
                \left(D(q^{n})^{-1/2}\rmG^{n}\right)^{\sfT}\nabla D(q^{n})\left(D(q^{n})^{-1/2}\rmG^{n}\right)
            \right\rbrace
        \right]_{i}
    \right)\\
    &\quad=
    \bbE_{\rmG^{n}}
    \left(
        \sum_{k=1}^{m}D_{i,k}(q^{n})
        \left[
            \left(D(q^{n})^{-1/2}\rmG^{n}\right)^{\sfT}\nabla D(q^{n})\left(
                D(q^{n})^{-1/2}\rmG^{n}
            \right)
        \right]_{k}
    \right)\\
    &\quad=
    \bbE_{\rmG^{n}}
    \left(
        \sum_{k=1}^{m}D_{i,k}(q^{n})
        \sum_{j,\ell=1}^{m}\partial_{q_k}
        D_{j,\ell}(q^{n})
        \left[
            D(q^{n})^{-1/2}\rmG^{n}
        \right]_{j}
        \left[
            D(q^{n})^{-1/2}\rmG^{n}
        \right]_{\ell}
    \right)\\
    &\quad=
    \sum_{k=1}^{m}D_{i,k}(q^{n})\sum_{j,\ell=1}^{m}\partial_{q_k}
    D_{j,\ell}(q^{n})
    \underbrace{\bbE_{\rmG^{n}}\left[
        \left[
            D(q^{n})^{-1/2}\rmG^{n}
        \right]_{j}\left[
            D(q^{n})^{-\sfT/2}\rmG^{n}
        \right]_{\ell}
    \right]}_{=~\sum_{\alpha=1}^{m}\left[D(q^{n})^{-1/2}\right]_{j,\alpha}\left[D(q^{n})^{-\sfT/2}\right]_{\alpha,\ell}=~\left[D(q^{n})^{-1}\right]_{j,\ell}}\\
    &\quad=
    \sum_{k=1}^{m}D_{i,k}(q^{n})\sum_{j,\ell=1}^{m}\partial_{q_k}\left[D(q^{n})\right]_{j,\ell}\left[D(q^{n})^{-1}\right]_{j,\ell}.
\end{align*}
By noticing that
\begin{align*}
    \left[D(q^{n})\Tr\left(D(q^{n})^{-1}\nabla D(q^{n})\right)\right]_{i}
    &=
    \sum_{k=1}^{m}D_{i,k}(q^{n})\left[\Tr\left(
        D(q^{n})^{-1}\nabla D(q^{n})
    \right)\right]_{k}\\
    &=\sum_{k=1}^{m}D_{i,k}(q^{n})
        \Tr\left(
            D(q^{n})^{-1}\partial_{q_{k}}D(q^{n})
        \right)\\
    &=\sum_{k=1}^{m}D_{i,k}(q^{n})\sum_{j,\ell=1}^{m}\partial_{q_k}\left[D(q^{n})\right]_{j,\ell}\left[D(q^{n})^{-1}\right]_{j,\ell},
\end{align*}
we thus conclude that
\begin{equation*}
    \bbE_{\rmG^{n}}\left[\calF_{2}(q^{n},D(q^{n})^{-1/2}\rmG^{n})\right]=-D(q^{n})\nabla V(q^{n})+\div D(q^{n}),
\end{equation*}
which is exactly~\eqref{eq:computation_f2}. This concludes the proof of Lemma~\ref{lem:one_step_hmc_weak_consistency_overdamped_langevin} for the GSV numerical scheme.

\subsubsection{IMR numerical scheme}
The IMR numerical scheme for the Hamiltonian function~\eqref{eq:H_RMHMC} writes
\begin{equation*}
    \left\lbrace
    \begin{aligned}
        \widetilde{q}^{n+1}&=q^{n}+\Delta t D\left(\frac{q^{n}+\widetilde{q}^{n+1}}{2}\right)\left(\frac{p^{n}+\widetilde{p}^{n+1}}{2}\right),\\
        \widetilde{p}^{n+1}&=p^{n}-\Delta t
            \left(-F\left(\frac{q^{n}+\widetilde{q}^{n+1}}{2}\right)
            +\frac{1}{2}\left(\frac{p^{n}+\widetilde{p}^{n+1}}{2}\right)^{\sfT}\nabla D\left(\frac{q^{n}+\widetilde{q}^{n+1}}{2}\right)\left(\frac{p^{n}+\widetilde{p}^{n+1}}{2}\right)
            \right)
        .
    \end{aligned}
    \right.
\end{equation*}
Introduce~$y^{n}=(q^{n},p^{n})$. We start by rewriting the update of the position as
\begin{align*}
    \widetilde{q}^{n+1}
    &=q^{n}+\Delta t~D\left(\dfrac{q^{n}+\widetilde{q}^{n+1}}{2}\right)\left(\dfrac{p^{n}+\widetilde{p}^{n+1}}{2}\right),\\
    &=q^{n}+A(\Delta t, y^{n}, \widetilde{y}^{n+1})-B(\Delta t, y^{n}, \widetilde{y}^{n+1}),
\end{align*}
where the term~$A(\Delta t,y^{n}, \widetilde{y}^{n+1})$ is
\begin{align*}
    A(\Delta t,y^{n}, \widetilde{y}^{n+1})
    &=\Delta t~D\left(\dfrac{q^{n}+\widetilde{q}^{n+1}}{2}\right)p^{n}\\
    &=
    \Delta t~D(q^{n})p^{n}\\
    &\quad
    +\dfrac{\Delta t^{2}}{2}\left\lbrace
        \nabla D(q^{n})\odot\left[
            D\left(\dfrac{q^{n}+\widetilde{q}^{n+1}}{2}\right)\left(\dfrac{p^{n}+\widetilde{p}^{n+1}}{2}\right)
        \right]
    \right\rbrace p^{n}\\
    &\quad
    +\dfrac{\Delta t^{3}}{8}\left[
        \nabla^2 D(q^{n})\odot\left(D(q^{n})p^{n}\right)^{\otimes2}
    \right]p^{n}\\
    &\quad+\rmO(\Delta t^{4}),
\end{align*}
and the term~$B(\Delta t,y^{n},\widetilde{y}^{n+1})$ is
\begin{equation*}
    \dfrac{\Delta t^{2}}{2}D\left(\frac{q^{n}+\widetilde{q}^{n+1}}{2}\right)\left[
        -F\left(\dfrac{q^{n}+\widetilde{q}^{n+1}}{2}\right)+\dfrac{1}{2}\left(\dfrac{p^{n}+\widetilde{p}^{n+1}}{2}\right)^{\sfT}\nabla D\left(\dfrac{q^{n}+\widetilde{q}^{n+1}}{2}\right)\left(\dfrac{p^{n}+\widetilde{p}^{n+1}}{2}\right)
    \right].
\end{equation*}
An expansion in powers of~$\Delta t$ of the second term in~$A$ gives
\begin{align*}
    &\dfrac{\Delta t^{2}}{2}\left\lbrace
        \nabla D(q^{n})\odot\left[
            D\left(\dfrac{q^{n}+\widetilde{q}^{n+1}}{2}\right)\left(\dfrac{p^{n}+\widetilde{p}^{n+1}}{2}\right)
        \right]
    \right\rbrace p^{n}\\
    &=\dfrac{\Delta t^{2}}{2}\left\lbrace
        \nabla D(q^{n})\odot\left[
            D(q^{n})p^{n}
        \right]
    \right\rbrace p^{n}\\
    &\quad+\dfrac{\Delta t^{3}}{4}\left\lbrace
        \nabla D(q^{n})\odot\left[
            \left(
                \nabla D(q^{n})\odot\left[
                    D(q^{n})p^{n}
                \right]
            \right)p^{n}
        \right]
    \right\rbrace p^{n}\\
    &\quad-
    \dfrac{\Delta t^{3}}{4}\left\lbrace
        \nabla D(q^{n})\odot\left[
            D(q^{n})\left(
                -F(q^{n})+\left(p^{n}\right)^{\sfT}\dfrac{\nabla D(q^{n})}{2}p^{n}
            \right)
        \right]
    \right\rbrace p^{n}\\
    &\quad+\rmO(\Delta t^{4}).
\end{align*}
To expand the terms present in~$B(\Delta t,y^{n},\widetilde{y}^{n+1})$, we need to expand~$F\left(\dfrac{q^{n}+\widetilde{q}^{n+1}}{2}\right)$ in powers of~$\Delta t$. As~$F(q)=-\nabla V(q)+\Tr(D(q)^{-1}\nabla D(q))/2$, we first write
\begin{equation*}
    \nabla V\left(
        \frac{q^{n}+\widetilde{q}^{n+1}}{2}
    \right)=\nabla V(q^{n})+\frac{\Delta t}{2}\nabla^{2}V(q^{n})D(q^{n})p^{n}+\rmO(\Delta t^{2}).
\end{equation*}
As for the second term~$\Tr\left[D\left(\dfrac{q^{n}+\widetilde{q}^{n+1}}{2}\right)^{-1}\nabla D\left(\dfrac{q^{n}+\widetilde{q}^{n+1}}{2}\right)\right]$, it holds
\begin{align*}
    D\left(\dfrac{q^{n}+\widetilde{q}^{n+1}}{2}\right)^{-1}
    &=D(q^{n})^{-1}-D(q^{n})^{-1}\left[\nabla D(q^{n})\odot\left[\dfrac{\widetilde{q}^{n+1}-q^{n}}{2}\right]\right]D(q^{n})^{-1}+\rmO(\Delta t^{2})\\
    &=D(q^{n})^{-1}-\dfrac{\Delta t}{2}D(q^{n})^{-1}\left[\nabla D(q^{n})\odot\left( D(q^{n}) p^{n}\right)\right]D(q^{n})^{-1}+\rmO(\Delta t^{2}),
\end{align*}
and
\begin{equation*}
    \nabla D\left(\dfrac{q^{n}+\widetilde{q}^{n+1}}{2}\right)=\nabla D(q^{n})+\dfrac{\Delta t}{2}\nabla^2 D(q^{n})\odot\left[D(q^{n})p^{n}\right]+\rmO(\Delta t^{2}).
\end{equation*}
Therefore,
\begin{align*}
    &
    \dfrac{\Delta t^{2}}{2}D\left(\dfrac{q^{n}+\widetilde{q}^{n+1}}{2}\right)\left[
        -F\left(\dfrac{q^{n}+\widetilde{q}^{n+1}}{2}\right)+\dfrac{1}{2}\left(\dfrac{p^{n}+\widetilde{p}^{n+1}}{2}\right)^{\sfT}\nabla D\left(\dfrac{q^{n}+\widetilde{q}^{n+1}}{2}\right)\left(\dfrac{p^{n}+\widetilde{p}^{n+1}}{2}\right)
    \right]\\
    &=
    \dfrac{\Delta t^{2}}{2}D(q^{n})\left(
        -F(q^{n})+\dfrac{1}{2}\left(p^{n}\right)^{\sfT}\nabla D(q^{n})p^{n}
    \right)\\
    &\quad+
    \dfrac{\Delta t^{3}}{4}\left[
        \Bigl[
            \nabla D(q^{n})\odot\left[
                D(q^{n})p^{n}
            \right]
        \Bigr]
        \left(
            -F(q^{n})+\dfrac{1}{2}\left(p^{n}\right)^{\sfT}\nabla D(q^{n})p^{n}
        \right)
    \right]\\
    &\quad+
    \dfrac{\Delta t^{3}}{4}D(q^{n})\nabla^2 V(q^{n})D(q^{n})p^{n}\\
    &\quad-\dfrac{\Delta t^{3}}{8}D(q^{n})
    \Tr\left[
        D(q^{n})^{-1}
        \Bigl(
            \nabla^2 D(q^{n})\odot\left[D(q^{n})p^{n}\right]
        \Bigr)
    \right]\\
    &\quad+\dfrac{\Delta t^{3}}{8}D(q^{n})
    \Tr\left[
        \Bigl(
            D(q^{n})^{-1}
            \left[
                \nabla D(q^{n})\odot 
                \left(
                    D(q^{n})p^{n}
                \right)
            \right]
            D(q^{n})^{-1}
        \Bigr)
        \nabla D(q^{n})
    \right]\\
    &\quad-\dfrac{\Delta t^{3}}{4}D(q^{n})
    \left[
        \left(p^{n}\right)^{\sfT}\nabla D(q^{n})\left(
            -F(q^{n})+\dfrac{1}{2}\left(p^{n}\right)^{\sfT}\nabla D(q^{n})p^{n}
        \right)
    \right]\\
    &\quad+\dfrac{\Delta t^{3}}{8}D(q^{n})\left[
        \left(p^{n}\right)^{\sfT}\Bigl\lbrace
            \nabla^2 D(q^{n})\odot\left[D(q^{n})p^{n}\right]
        \Bigr\rbrace p^{n}
    \right]\\
    &\quad+\rmO(\Delta t^{4}).
\end{align*}
By gathering the previous equalities, one finally obtains
\begin{align*}
    \widetilde{q}^{n+1}
    &=q^{n}+\Delta t\,D(q^{n})p^{n}\\
    &\quad
    +\dfrac{\Delta t^{2}}{2}\left\lbrace
        \left[
            \nabla D(q^{n})\odot\left[
                D(q^{n})p^{n}
            \right]
        \right]p^{n}
        -
        D(q^{n})\left[
            -F(q^{n})+\dfrac{\left(p^{n}\right)^{\sfT}\nabla D(q^{n})p^{n}}{2}
        \right]
    \right\rbrace\\
    &\quad
    +\Delta t^{3}
    \left[
        \begin{array}[]{l}
            +\dfrac{1}{4}\left\lbrace\nabla D(q^{n})\odot\left[\left(\nabla D(q^{n})\odot\left[D(q^{n})p^{n}\right]\right)p^{n}\right]\right\rbrace p^{n}\\[0.2cm]
            -\dfrac{1}{4}\left\lbrace\nabla D(q^{n})\odot\left[D(q^{n})\left(-F(q^{n})+\left(p^{n}\right)^{\sfT}\dfrac{\nabla D(q^{n})}{2}p^{n}\right)\right]\right\rbrace p^{n}\\[0.2cm]
            +\dfrac{1}{4}\left[\nabla D(q^{n})\odot\left[D(q^{n})p^{n}\right]\right]\left(-F(q^{n})+\left(p^{n}\right)^{\sfT}\dfrac{\nabla D(q^{n})}{2}p^{n}\right)\\[0.2cm]
            +\dfrac{1}{4}D(q^{n})\nabla^{2}V(q^{n})D(q^{n})p^{n}\\[0.2cm]
            -\dfrac{1}{8}D(q^{n})\Tr\left[D(q^{n})^{-1}\left(\nabla^{2}D(q^{n})\odot\left[D(q^{n})p^{n}\right]\right)\right]\\[0.2cm]
            +\dfrac{1}{8}D(q^{n})\Tr\left[\left(D(q^{n})^{-1}\left[\nabla D(q^{n})\odot\left(D(q^{n})p^{n}\right)\right]D(q^{n})^{-1}\right)\nabla D(q^{n})\right]\\[0.2cm]
            -\dfrac{1}{4}D(q^{n})\left[\left(p^{n}\right)^{\sfT}\nabla D(q^{n})\left(-F(q^{n})+\left(p^{n}\right)^{\sfT}\dfrac{\nabla D(q^{n})}{2}p^{n}\right)\right]\\[0.2cm]
            +\dfrac{1}{8}D(q^{n})\left[\left(p^{n}\right)^{\sfT}\left\lbrace\nabla^{2}D(q^{n})\odot\left[D(q^{n})p^{n}\right]\right\rbrace p^{n}\right]\\[0.2cm]
            +\dfrac{1}{8}\left[\nabla^{2}D(q^{n})\odot\left(D(q^{n})p^{n}\right)^{\otimes2}\right]p^{n}
        \end{array}
    \right]\\
    &\quad+\rmO(\Delta t^{4})\\
    &=q^{n} + \Delta t\,\calF_{1}(q^{n}, p^{n}) + \dfrac{\Delta t^{2}}{2}\calF_{2}(q^{n}, p^{n})+\Delta t^{3}\calF_{3}(q^{n},p^{n})+\rmO(\Delta t^{4}),
\end{align*}
where~$\calF_1$ and~$\calF_2$ are the same functions as in~\eqref{eq:gsv_expansion_functions}, whereas~$\calF_3$ is different but still odd in~$\rmG^{n}$. In particular,~$\bbE_{\rmG^{n}}\left[\calF_3(q^{n},D(q^{n})^{-1/2}\rmG^{n})\right]=0$ and~$\bbE_{\rmG^{n}}\left[\calF_{2}(q^{n},D(q^{n})^{-1/2}\rmG^{n})\right]=-D(q^{n})\nabla V(q^{n})+\div D(q^{n})$ following the same computations as in the previous section. This concludes the proof for the IMR numerical scheme.

\subsection{Proof of Proposition~\ref{prop:one_step_hmc_weakly_consistent_overdamped_langevin}}
\label{subsec:proof:HMC_weakly_consistent}

First note that, if~$f:\calX\to\calX$ grows at most polynomially, with derivatives growing at most polynomially, then~\eqref{eq:ass_D_bound} implies that there exist~$\alpha \in\bbN$ and~$K\in\bbR_{+}$ such that
\begin{equation*}
  \forall q\in\calO,\qquad
  \left\lvert \Delta t^{\delta} f\left(q,D(q)^{-1/2}\rmG\right)\right\rvert\leqslant K\Delta t^{\delta}\left(1+\left\lVert \rmG\right\rVert^{\alpha }+\left\lVert q\right\rVert^{\alpha }\right).
\end{equation*}
We write such remainder terms as~$\Delta t^{\delta}f\left(q,D(q)^{-1/2}\rmG\right)=\rmO(\Delta t^{\delta})$. In particular,
\begin{equation*}
  \bbE_{\rmG}\left[\left\lvert \Delta t^{\delta}f\left(q,D(q)^{-1/2}\rmG\right)\right\rvert^{r}\right]\leqslant K_r\Delta t^{\delta r}\left(1+\left\lVert q\right\rVert^{\alpha r}\right),  
\end{equation*}
with~$K_r\in\bbR_{+}$ for any~$r\geqslant 0$, which we may also write~$\bbE_{\rmG}\left[\left\lvert \Delta t^{\delta}f\left(q,D(q)^{-1/2}\rmG\right)\right\rvert^{r}\right]=\rmO\left(\Delta t^{\delta r}\right)$. 
We now prove that, for any~$\Delta t\in(0,\Delta t_{\star}]$, it holds
\begin{equation}
  \label{eq:add_Gaussians_not_in_B}
  \bbE_{\rmG}\left[
    f\left(q,D(q)^{-1/2}\rmG\right)
  \right]=\bbE_{\rmG}\left[
    f\left(q,D(q)^{-1/2}\rmG\right)\mathbf{1}_{\left(q,D(q)^{-1/2}\rmG\right)\in\calB_{\Delta t}}
  \right]+\rmO\left(\Delta t^{2(r_3+1)}\right).
\end{equation}
Indeed, consider~$K\in\bbR_{+}$ and~$\alpha\in\bbN$ such that, for all~$q\in\calO$, it holds~$\left\lvert f\left(q,D(q)^{-1/2}\rmG\right)\right\rvert \leqslant K(1+\left\lVert q\right\rVert^{\alpha}+\left\lVert\rmG\right\rVert^{\alpha})$. Then, using~\eqref{eq:rmhmc_rejection_probability} with~$a=0$ and~$a=\alpha$,
\begin{align*}
  \MoveEqLeft[6]
  \left\lvert\bbE_{\rmG}\left[
    f\left(q,D(q)^{-1/2}\rmG\right)\mathbf{1}_{\left(q,D(q)^{-1/2}\rmG\right)\in\calX\setminus\calB_{\Delta t}}
  \right]\right\rvert\\
  &\quad\leqslant
  K(1+\left\lVert q\right\rVert^{\alpha})\bbE_{\rmG}\left[\mathbf{1}_{\left(q,D(q)^{-1/2}\rmG\right)\in\calX\setminus\calB_{\Delta t}}\right]+K\bbE_{\rmG}\left[\left\lVert\rmG\right\rVert^{\alpha}\mathbf{1}_{\left(q,D(q)^{-1/2}\rmG\right)\in\calX\setminus\calB_{\Delta t}}\right]\\
  &\quad\leqslant
  K(1+\left\lVert q\right\rVert^{\alpha})C_0\Delta t^{2(r_3+1)}+KC_{\alpha}\left(1+\left\lVert q\right\rVert^{b_{\alpha}}\right)\Delta t^{2(r_3+1)}
  =\rmO\left(\Delta t^{2(r_3+1)}\right).
\end{align*}

We now fix~$f\in\scrF_{\pol}$ and~$q\in\calO$. The discrete position evolution operator~$P_{\Delta t}$ using one step of Algorithm~\ref{alg:hmc_scheme_rev_check} with either the IMR or the GSV numerical scheme acts on~$f$ as
\begin{equation*}
  \left(
    P_{\Delta t}f
  \right)(q)=
  \bbE^{q}\left[
    f(q^{1})
  \right].
\end{equation*}
Denote by~$A_{\Delta t}(q,\rmG)$ the acceptance probability in the Metropolis--Hastings step~\ref{step:hmc_3}, which is defined for~$\left(q,D(q)^{-1/2}\rmG\right)\in\calA_{\Delta t}$ by
\begin{equation*}
  A_{\Delta t}(q,\rmG)=\min\left(1,\rme^{-\Delta H\left(q,\rmG\right)}\right),
\end{equation*}
where
\begin{equation*}
  \Delta H\left(q,\rmG\right)=H\left(\varphi_{\Delta t}\left(q,D(q)^{-1/2}\rmG\right)\right)-H\left(q,D(q)^{-1/2}\rmG\right).
\end{equation*}
If~$\rmG$ is such that~$\left(q,D(q)^{-1/2}\rmG\right)\in\calX\setminus\calB_{\Delta t}$ or if the proposal is rejected in the Metropolis--Hastings procedure, then~$q^{1}=q$. Therefore,
\begin{align}
  (P_{\Delta t}f-f)(q)\nonumber
  &=\bbE^{q}\left[
    \left(
      f(q^{1})-f(q)
    \right)\mathbb{1}_{\left(q,D(q)^{-1/2}\rmG\right)\in\calB_{\Delta t}}A_{\Delta t}(q,\rmG)
  \right]\nonumber\\
  &=\label{eq:weak1}
  \begin{aligned}[t]
    &\bbE^{q}\left[
      \left(f(q^{1})-f(q)\right)\mathbb{1}_{\left(q,D(q)^{-1/2}\rmG\right)\in\calB_{\Delta t}}
    \right]\\
    &+
    \bbE^{q}\left[
      \left(A_{\Delta t}\left(q,\rmG\right)-1\right)\left(f(q^{1})-f(q)\right)\mathbb{1}_{\left(q,D(q)^{-1/2}\rmG\right)\in\calB_{\Delta t}}
    \right].
  \end{aligned}
\end{align}
We first focus on the first term of~\eqref{eq:weak1}. We expand
\begin{align*}
    f(q^{1})-f(q)
    &=\nabla f(q)\cdot\left(q^{1}-q\right)+\frac{1}{2}\nabla^{2}f(q):\left(q^{1}-q\right)^{\otimes2}\\
    &\quad+\frac{1}{6}\nabla^{3}f(q):\left(q^{1}-q\right)^{\otimes3}+\frac{1}{6}\int_{0}^{1}(1-\theta)^{3}\nabla^{4}f((1-\theta)q^{1}+\theta q):(q^{1}-q)^{\otimes4}\rmd\theta,
\end{align*}
where we denoted, for any positions~$(q,\tilde{q})$ in~$\calO$,
\begin{equation*}
    \nabla^{3}f(q):\tilde{q}^{\otimes 3}=\sum_{\alpha,\beta,\gamma = 1}^{m}\frac{\partial^{3}f}{\partial_{q_\alpha}\partial_{q_\beta}\partial_{q_\gamma}}(q)\,\tilde{q}_\alpha\tilde{q}_\beta\tilde{q}_\gamma,
\end{equation*}
and similarly for~$\nabla^{4}f$.

Since~$\calB_{\Delta t}\subset\calA_{\Delta t}$, we may plug the expansion~\eqref{eq:expansion_position_rmhmc} into the first term of~\eqref{eq:weak1}. In order to write the remainder~$\Delta t^{4}\calR_{\Delta t}(q,p)$ of the expansion~\eqref{eq:expansion_position_rmhmc} as~$\rmO(\Delta t^{4})$ when integrating with respect to~$\rmG$, we need to show that~$\calR_{\Delta t}(q,p)$ is bounded by a polynomial function of~$\left\lVert q\right\rVert$ and~$\left\lVert p\right\rVert$. We show more generally that~$\varphi_{\Delta t}(q,p)$ (which is the proposal of~$(q,p)$ and not simply the proposal of the position) is bounded by a polynomial function of~$\left\lVert q\right\rVert$ and~$\left\lVert p\right\rVert$. Indeed, let~$(q,p)\in\calX$ and~$g\colon \bbR_{+}\to\calX$ defined by~$g(t)=\phi_{t}(q,p)$ where~$\phi_t$ is the exact flow of~\eqref{eq:hamiltonian_dynamics} with initial condition~$(q_{0},p_{0})=(q,p)$. Note first that~$g(t)$ is well defined for all times~$t\in\bbR_{+}$. Indeed, the solution to~\eqref{eq:hamiltonian_dynamics} is well defined for small times by the Cauchy--Lipschitz theorem, and globally defined for all times using the Hamiltonian function as a Lyapunov function. To make the latter point precise, we show in fact the stronger result that the norm of~$g(s)=(q_s,p_s)$ is controlled by a polynomial function of~$\left\lVert q\right\rVert$ and~$\left\lVert p\right\rVert$. 
Since the Hamiltonian is exactly preserved by~$\phi_s$, is holds~$H(q_s,p_s)=H(q,p)$, which can be rewritten as
\begin{equation*}
  V(q_{s})-\frac{1}{2}\ln\det D(q_{s})+\frac{1}{2}p_{s}^{\sfT}D(q_{s})p_{s}=V(q)-\frac{1}{2}\ln\det D(q)+\frac{1}{2}p^{\sfT}D(q)p.
\end{equation*}
By Assumption~\ref{ass:RMHMC}, the right-hand side is bounded by a polynomial function of~$\left\lVert q\right\rVert$ and~$\left\lVert p\right\rVert$. Using that~$V$ is bounded from below and that~$D$ is uniformly bounded (from above and from below), it follows that~$\left\lVert p_s\right\rVert$ is bounded by a polynomial function of~$\left\lVert q\right\rVert$ and~$\left\lVert p\right\rVert$. The first equation of~\eqref{eq:hamiltonian_dynamics}, namely
\begin{equation*}
  %\label{eq:ham_position_polynomial_bound}
  \frac{\rmd}{\rmd t}q_t=D(q_t)p_t,
\end{equation*}
then implies that~$\left\lVert q_s\right\Vert$ is bounded by a polynomial function of~$\left\lVert q\right\rVert$ and~$\left\lVert p\right\rVert$ (uniformly on bounded intervals). Combined with~\eqref{eq:remainder_term_uniform_bound}, this shows that~$\varphi_{\Delta t}(q,p)$ is also bounded by a polynomial function of~$\left\lVert q\right\rVert$ and~$\left\lVert p\right\rVert$. Similar reasonings underly remainder terms appearing below.

Using~\eqref{eq:add_Gaussians_not_in_B}, and in view of~\eqref{eq:expectation_odd_f} and~\eqref{eq:computation_f2}, 
\begin{align*}
  \MoveEqLeft[2]
  \bbE_{\rmG}\left[
      \left(f(q^{1})-f(q)\right)\mathbb{1}_{\left(q,D(q)^{-1/2}\rmG\right)\in\calB_{\Delta t}}
  \right]\\
  &\quad=
  \frac{\Delta t^{2}}{2}\nabla f(q)\cdot
      \left(
          -D(q)\nabla V(q)+\div D(q)
      \right)\\
  &\qquad+\frac{\Delta t^{2}}{2}\bbE_{\rmG}\left[
      \nabla^{2} f(q):\left(\calF_1\left(q,D(q)^{-1/2}\rmG\right)\right)^{\otimes2}
  \right]+\rmO\left(\Delta t^{2(\min(1,r_3)+1)}\right).
\end{align*}
The expectation in the latter term can be computed as
\begin{align*}
    \MoveEqLeft[4]
    \bbE_{\rmG}\left(
        \nabla^{2}f(q):\left(\calF_1\left(q,D(q)^{-1/2}\rmG\right)\right)^{\otimes2}
    \right)\\
    &=
    \bbE_{\rmG}\left(
        \sum_{i,j=1}^{m}\frac{\partial^{2}f}{\partial q_i\partial q_j}(q)\left[D(q)^{1/2}\rmG\right]_{i}\left[D(q)^{1/2}\rmG\right]_{j}
    \right)\\
    &=
    \bbE_{\rmG}\left(
        \sum_{i,j=1}^{m}\frac{\partial^{2}f}{\partial q_i\partial q_j}(q)\sum_{k,\ell=1}^{m}\left[D(q)^{1/2}\right]_{i,k}\rmG_{k}\left[D(q)^{1/2}\right]_{j,\ell}\rmG_\ell
    \right)\\
    &=
    \sum_{i,j=1}^{m}\frac{\partial^{2}f}{\partial q_i\partial q_j}(q)\sum_{k=1}^{m}\left[D(q)^{1/2}\right]_{i,k}\left[D(q)^{1/2}\right]_{j,k}\\
    &=
    \sum_{i,j=1}^{m}\frac{\partial^{2}f}{\partial q_i\partial q_j}(q)\left[D(q)\right]_{i,j}=D(q):\nabla^{2}f(q).
\end{align*}
This shows that
\begin{equation*}
  \bbE_{\rmG}\left[
    \left(f(q^{1})-f(q)\right)\mathbb{1}_{\left(q,D(q)^{-1/2}\rmG\right)\in\calB_{\Delta t}}
\right]=\frac{\Delta t^{2}}{2}\calL f(q)+\rmO\left(\Delta t^{2(\min(1,r_3)+1)}\right),
\end{equation*}
where~$\calL$ is the generator of the dynamics~\eqref{eq:overdamped_langevin_diffusion} defined by~\eqref{eq:generator_overdamped_langevin}.

We now focus on the second term of~\eqref{eq:weak1}. In view of~\eqref{eq:remainder_term_uniform_bound}, it holds~$\varphi_{\Delta t}(q,p)=\phi_{\Delta t}(q,p)+\Delta t^{3}\widetilde{R}_{\Delta t}(q,p)$, where~$\widetilde{R}_{\Delta t}$ is a function growing at most polynomially. Since the Hamiltonian is exactly preserved by~$\phi_t$, there exists, for any~$\Delta t\in(0,\Delta t_{\star}]$ and~$(q,p)\in\calA_{\Delta t}$, a real number~$\theta_{\Delta t}(q,p)\in[0,1]$ such that
\begin{align*}
  H(\varphi_{\Delta t}(q,p))
  =H(q,p)
  +\Delta t^{3}\nabla H\left(\phi_{\Delta t}(q,p)+\Delta t^{3}\theta_{\Delta t}(q,p)\widetilde{R}_{\Delta t}(q,p)\right)\cdot\widetilde{R}_{\Delta t}(q,p).
\end{align*}
Using~$0\leqslant 1-\min(1,\rme^{-x})\leqslant\max(0,x)\leqslant\left\lvert x\right\rvert$, it holds for~$\left(q,p\right)\in\calA_{\Delta t}$ with~$p=D(q)^{-1/2}\rmG$,
\begin{align}
  0\leqslant 1-A_{\Delta t}(q,\rmG)
  &\leqslant \Delta t^{3}\left\lvert\nabla H\left(\phi_{\Delta t}\left(q,p\right)+\Delta t^{3}\theta_{\Delta t}(q,p)\widetilde{R}_{\Delta t}\left(q,p\right)\right)\cdot\widetilde{R}_{\Delta t}\left(q,p\right)\right\rvert,\label{eq:weak2}
\end{align}
the right-hand side being bounded by a polynomial function of~$\left\lVert q\right\rVert$ and~$\left\lVert p\right\rVert$. Similarly, for~$\left(q,D(q)^{-1/2}\rmG\right)\in\calA_{\Delta t}$,
\begin{equation*}
  f(q^{1})-f(q)=\Delta t\nabla f(q)\cdot D(q)^{1/2}\rmG+\rmO(\Delta t^{2}).
\end{equation*}
Plugging first the previous estimate and~\eqref{eq:weak2} into the second term of~\eqref{eq:weak1}, and then using~\eqref{eq:add_Gaussians_not_in_B} to remove the indicator function, one obtains
\begin{align*}
  &\left\lvert\bbE_{\rmG}\left[
    \left(A_{\Delta t}\left(q,\rmG\right)-1\right)\left(f(q^{1})-f(q)\right)\mathbb{1}_{\left(q,D(q)^{-1/2}\rmG\right)\in\calB_{\Delta t}}
  \right]\right\rvert\\
  &\qquad\qquad\leqslant\bbE_{\rmG}\left[
    \left\lvert A_{\Delta t}\left(q,\rmG\right)-1\right\rvert\left\lvert f(q^{1})-f(q)\right\rvert\mathbb{1}_{\left(q,D(q)^{-1/2}\rmG\right)\in\calB_{\Delta t}}
  \right]=\rmO(\Delta t^{4}).
\end{align*}
We therefore conclude that
\begin{equation*}
  (P_{\Delta t}f-f)(q)=\frac{\Delta t^{2}}{2}\calL f(q)+\rmO\left(\Delta t^{2(\min(1,r_3)+1)}\right),
\end{equation*}
so that~$P_{\Delta t}f-\rme^{\Delta t^{2}\calL/2} f=\rmO\left(\Delta t^{2(\min(1,r_3)+1)}\right)$ in view of~\eqref{eq:stability_continous_process}, which leads to~\eqref{eq:rmhmc_weak_estimate}.

%---------------------------------------------------------
%---------------------------------------------------------
%---------------- GSV FP ---------------------------------
%---------------------------------------------------------
%---------------------------------------------------------
\subsection{Proof of Proposition~\ref{prop:GSV_FP}}
\label{subsec:proof_GSV_FP}

Let~$x=(q,p)\in\calX$. The configurations~$x_1=(q_1,p_1)\in\calX$ and~$x_2=(q_2,p_2)\in\calX$ such that~$\Phi_{\Delta t}^{\GSV}(x,x_1,x_2)=0$ satisfy
\begin{equation*}
  \left\lbrace
    \begin{aligned}
      x_1&=x+\frac{\Delta t}{2}J\nabla H(q,p_1),\\
      x_2&=x_1+\frac{\Delta t}{2}J\nabla H(q_2,p_1),
    \end{aligned}
  \right.
\end{equation*}
where~$J$ is defined in~\eqref{eq:hamiltonian_dynamics}. Following the same reasoning and using the same notation as in the proof of Proposition~\ref{prop:IMR_FP}, fix~$R>0$. For~$\Delta t\left\lVert\nabla H\right\rVert_{\calC^{0}\left(B_x^{2R}\right)}\leqslant 2R$, define
\begin{equation*}
  \calF_{\Delta t}^{x}\colon\left\lbrace
  \begin{array}[]{ccl}
    B_x^R &\to&B_x^R \\
    x_1&\mapsto &x+\dfrac{\Delta t}{2} J\nabla H\left(q,p_1\right),
  \end{array}
  \right.
\end{equation*}
and for~$x_1\in B_x^R$,
\begin{equation*}
  \scrF_{\Delta t}^{x_1}\colon\left\lbrace
  \begin{array}[]{ccl}
    B_{x_1}^R &\to &B_{x_1}^R \\
    x_2&\mapsto&x_1+\dfrac{\Delta t}{2} J\nabla H\left(q_2,p_1\right).
  \end{array}
  \right.
\end{equation*}
The map~$\calF_{\Delta t}^{x}$ is well-defined since~$\frac{\Delta t}{2}\left\lVert\nabla H\right\rVert_{\calC^{0}(B_x^R)}\leqslant\frac{\Delta t}{2}\left\lVert\nabla H\right\rVert_{\calC^{0}(B_x^{2R})}\leqslant R$. The map~$\scrF_{\Delta t}^{x_1}$ is also well-defined since, if~$x_1\in B_x^R$ and~$x_2\in B_{x_1}^{R}$, then~$x_2\in B_{x}^{2R}$ so that~$\frac{\Delta t}{2}\left\lVert\nabla H\right\rVert_{\calC^{0}(B_{x_1}^R)}\leqslant\frac{\Delta t}{2}\left\lVert\nabla H\right\rVert_{\calC^{0}(B_x^{2R})}\leqslant R$. It is then easy to show that, under the condition
\begin{equation*}
  \Delta t\frac{\left\lVert\nabla^{2}H\right\rVert_{\calC^{0}(B_x^{2R})}}{2}<1,  
\end{equation*}
the two maps~$\calF_{\Delta t}^{x}$ and~$\scrF_{\Delta t}^{x_1}$ are contractions. Define~$F_1\colon\calX\to(0,+\infty]$ as
\begin{equation*}
  F_1(x)=
  \left\lbrace
  \begin{array}[]{ll}
    \min\left(\dfrac{2R}{\left\lVert\nabla H\right\rVert_{\calC^{0}\left(B_x^{2R}\right)}},\dfrac{2}{\left\lVert\nabla^{2}H\right\rVert_{\calC^{0}\left(B_x^{2R}\right)}}\right)&\text{if }\left\lVert\nabla H\right\rVert_{\calC^{0}\left(B_x^{2R}\right)}\left\lVert\nabla^{2}H\right\rVert_{\calC^{0}\left(B_x^{2R}\right)}\neq0,\\
    +\infty&\text{otherwise},
  \end{array}
  \right.
\end{equation*}
and~$F_2\colon\calX\to(0,+\infty]$ as
\begin{equation*}
  F_2(x)=\min\left(F_1(x),\min\limits_{x_1\in B_x^{R}}\left\lbrace F_1(x_1)\right\rbrace\right).
\end{equation*}
Note that~$F_1$ (respectively~$F_2$) is continuous on~$\dom F_1=\left\lbrace x\in\calX,F_1(x)<+\infty\right\rbrace$ (respectively on~$\dom F_2=\left\lbrace x\in\calX,f_2(x)<+\infty\right\rbrace$). Note also that~$F_2(x)$ is positive (and can be infinite). If~$\Delta t<F_2(x)$, there exists by the Banach fixed point theorem a unique~$x_1\in B_x^{R}$ such that~$x_1-\calF_{\Delta t}^{x}(x_1)=\Phi_{\Delta t/2}^{\EulerB}(x,x_1)=0$ and~$x_2\in B_{x_1}^{R}\subset B_{x}^{2R}$ such that~$x_2-\scrF_{\Delta t}^{x_1}(x_2)=\Phi_{\Delta t/2}^{\EulerA}(x_1,x_2)=0$, so that~$\Phi_{\Delta t}^{\GSV}(x,x_1,x_2)=0$. We then denote by~$\chi_{\Delta t}^{\GSV,\FP}(x):=(x_1,x_2)$. This naturally defines the numerical flow~$\chi_{\Delta t}^{\GSV,\FP}$ on the open set
\begin{equation*}
  \calA_{\Delta t}^{\GSV,\FP}:=F_{2}^{-1}(\Delta t,+\infty)\cup\Int F_2^{-1}\left\lbrace+\infty\right\rbrace.
\end{equation*}
For~$\Delta t>0$ small enough, this set is nonempty. Indeed, if for all~$x\in\calX$,~$F_2(x)=+\infty$, then~$\calA_{\Delta t}^{\GSV,\FP}=\calX$. Otherwise, there exists~$x_0\in\calX$ such that~$F(x_0)$ is finite. In this case, we choose~$\Delta t_{\star}=F(x_0)/2>0$ (which will later be reduced but will still remain positive), and~$\calA_{\Delta t_{\star}}^{\GSV,\FP}$ is nonempty. Moreover, for any~$\Delta t\in(0,\Delta t_{\star}]$, it holds~$\calA_{\Delta t}^{\GSV,\FP}\subset\calA_{\Delta t_{\star}}^{\GSV,\FP}$.
We next define the map~$G\colon\calX\to(0,+\infty]$ by
\begin{equation*}
  G(x)=\min\left(F_2(x), \min\limits_{y\in B_x^{2R}}\left\lbrace F_2(S(y))\right\rbrace\right).
\end{equation*}
Note that~$G(x)$ is positive (and can be infinite). Assume that~$\Delta t<G(x)$ for some~$x\in\calX$. Then~$x_2:=\varphi_{\Delta t}^{\GSV,\FP}(x)$ and~$(x_1,x_2):=\chi_{\Delta t}^{\GSV,\FP}(x)$ are well-defined. Moreover, since in this case~$\Delta t<F_1(S(x_2))$, there exists a unique~$x'_1\in B_{S(x_2)}^R$ such that~$\Phi_{\Delta t/2}^{\EulerB}(S(x_2),x'_1)=0$. Since~$\Phi_{\Delta t}^{\GSV}$ is~$S$-reversible (see Proposition~\ref{prop:IMR_GSV_S_reversibility}), it holds~$\Phi_{\Delta t}^{\EulerB}(S(x_2),S(x_1))=0$. Using the fact that~$S$ is an isometry,~$\left\lVert S(x_2)-S(x_1)\right\rVert=\left\lVert x_2-x_1\right\rVert\leqslant R$. Thus, by the uniqueness of the zero of~$y\mapsto\Phi_{\Delta t}^{\EulerB}(S(x_2),y)$ in~$B_{S(x_2)}^{R}$, one can conclude that~$x'_1=S(x_1)$. The same reasoning shows that~$x'_2=S(x)$. Therefore, it holds 
\begin{equation*}
  \chi_{\Delta t}^{\GSV,\FP}\circ S\circ \varphi_{\Delta t}^{\GSV,\FP}(x)=\left(
    S\circ\chi_{\Delta t,1}^{\GSV,\FP}(x),S(x)
  \right).
\end{equation*}
Following the same reasoning as for the construction of~$\calA_{\Delta t}^{\GSV,\FP}$, there exists~$\Delta t_{\star}>0$ such that for any~$\Delta t\in(0,\Delta t_{\star}]$, the map~$\chi_{\Delta t}^{\GSV,\FP}$ is~$S$-reversible. Then the set~$\calB_{\Delta t}^{\GSV,\FP}$ defined by~\eqref{eq:B_GSV_FP} is open by Lemma~\ref{lem:B_open_set} and nonempty as it contains~$G^{-1}(\Delta t,+\infty]$.

To conclude, it remains to prove that~$\chi_{\Delta t}^{\GSV,\FP}\colon\calA_{\Delta t}^{\GSV,\FP}\to\calX^{2}$ is~$\calC^{1}$ for~$\Delta t\in(0,\Delta t_{\star}]$ and that the second statement of~\eqref{eq:gsv_fp} holds. Note that, for~$x\in\calA_{\Delta t}^{\GSV,\FP}$, one has~$\Phi_{\Delta t}^{\GSV}\left(x,\chi_{\Delta t}^{\GSV,\FP}(x)\right)=0$, and the matrices
\begin{equation*}
  \left\lbrace
    \begin{aligned}
      \nabla_{y}\Phi_{\Delta t/2}^{\EulerB}&=\rmI_d-\frac{\Delta t}{2}\begin{pmatrix}
        0_m & \nabla_p^{2}H\\
        0_m & -\nabla_{p,q}^{2}H
      \end{pmatrix},\\[0.2cm]
      \nabla_{y}\Phi_{\Delta t/2}^{\EulerA}&=\rmI_d-\frac{\Delta t}{2}\begin{pmatrix}
        \nabla_{q,p}^{2}H & 0_m\\
        -\nabla_{p}^{2}H & 0_m
      \end{pmatrix},
    \end{aligned}
  \right.
\end{equation*}
evaluated respectively at~$\left(x,\chi_{\Delta t,1}^{\GSV,\FP}(x)\right)$ and at~$\left(\chi_{\Delta t,1}^{\GSV,\FP}(x),\chi_{\Delta t,2}^{\GSV,\FP}(x)\right)$, are invertible since~$\Delta t<2/\left\lVert\nabla^{2}H\right\rVert_{\calC^{0}(B_x^{2R})}$ so that the matrix
\begin{equation*}
  \nabla_{y}\Phi_{\Delta t}^{\GSV}=
  \begin{pmatrix}
    \nabla_{y}\Phi_{\Delta t/2}^{\EulerB} & 0_d\\[0.2cm]
    \star & \nabla_{y}\Phi_{\Delta t/2}^{\EulerA}
  \end{pmatrix}
\end{equation*}
evaluated at~$\left(x,\chi_{\Delta t}^{\GSV,\FP}(x)\right)$ is invertible. Therefore, the map~$\chi_{\Delta t}^{\GSV,\FP}$ is~$\calC^{1}$ by the implicit function theorem, and~\eqref{eq:gsv_fp} holds. This concludes the proof.

%---------------------------------------------------------
%---------------------------------------------------------
%---------------- LAMBDA IMPLICIT ------------------------
%---------------------------------------------------------
%---------------------------------------------------------
\subsection{Proof of Proposition~\ref{prop:Dimp}}
\label{subsec:proof:prop_D_imp}
We follow the strategy of the proof of~\cite[Proposition~2.5]{lelievre_2019}.
We start by stating and proving two useful lemmas, and introduce to this end some notation. Since~$\calX$ is an open subset of~$\bbR^d$, there exists a sequence~$(K_n)_{n\geqslant 1}$ of compact subsets of~$\bbR^{d}$ such that~$\calX=\bigcup_{n\geqslant 1}K_n$. For instance, when~$\calX=\calO\times\bbR^{m}$ with~$\calO\subsetneq\bbR^{m}$, one can choose
\begin{equation*}
    K_n=\left\lbrace x\in\calX,\, \left\lVert x\right\rVert\leqslant n\text{ and }\dist(x,\bbR^d\setminus\calX)\geqslant\frac{1}{n}\right\rbrace.
\end{equation*}

\begin{lemma}
    For any~$1\leqslant n,m< +\infty$, the set
    \begin{equation*}
        \calG_{\Delta t}^{n,m}
        =
        \left\lbrace
            (x,y)\in K_n\times K_n^{k},\, \Phi_{\Delta t}(x,y)=0,\,\nabla_{y}\Phi_{\Delta t}(x,y)\text{ is invertible and }\left\lVert\nabla_{y}\Phi_{\Delta t}(x,y)^{-1}\right\rVert\leqslant m
        \right\rbrace.
    \end{equation*}
    is compact.
\end{lemma}
\begin{proof}
    Since~$\calG_{\Delta t}^{n,m}\subset K_n\times K_n^k$ and~$K_n$ is compact, we only have to show that the set is closed. Let~$((x_j,y_j))_{j\geqslant1}$ be a sequence of elements of~$\calG_{\Delta t}^{n,m}$ converging to~$(x,y)\in\calX\times\calX^k$. Since~$K_n$ is closed,~$(x,y)\in K_n\times K_n^k$. Moreover,~$\Phi_{\Delta t}(x,y)=0$ follows from the continuity of~$\Phi_{\Delta t}$. Now, upon extraction, there exists a matrix~$B\in\bbR^{dk}\times\bbR^{dk}$ such that
    \begin{equation*}
        \left(\nabla_{y}\Phi_{\Delta t}(x_j,y_j)\right)^{-1}\xrightarrow[j\to+\infty]{}B,
    \end{equation*}
    with~$\left\lVert B\right\rVert\leqslant m$. For all~$j\geqslant 1$, one has
    \begin{equation*}
        \nabla_{y}\Phi_{\Delta t}(x_j,y_j)\left(\nabla_{y}\Phi_{\Delta t}(x_j,y_j)\right)^{-1}=\rmI_{dk},
    \end{equation*}
    so that from the~$\calC^1$ regularity of~$\Phi_{\Delta t}$, taking the limit~$j\to+\infty$ implies that
    \begin{equation*}
        \nabla_{y}\Phi_{\Delta t}(x,y)B=\rmI_{dk}.
    \end{equation*}
    Thus, the matrix~$\nabla_{y}\Phi_{\Delta t}(x,y)$ is invertible, its inverse being~$B$, and it holds~$\left\lVert\left(\nabla_{y}\Phi_{\Delta t}(x,y)\right)^{-1}\right\rVert\leqslant m$. It follows that~$(x,y)\in\calG_{\Delta t}^{n,m}$ and~$\calG_{\Delta t}^{n,m}$ is a compact set.
\end{proof}

Note that from Assumption~\ref{ass:existence_solution_implicit_problem}, the set~$\calG_{\Delta t}^{n,m}$ is nonempty for~$n,m\geqslant1$ sufficiently large.

\begin{lemma}
    There exists an open subset~$\calD_{\Delta t}^{n,m}$ such that~$\calG_{\Delta t}^{n,m}\subset\calD_{\Delta t}^{n,m}$, and a~$\calC^{1}$ function~$\Lambda_{\Delta t}^{n,m}\colon\calD_{\Delta t}^{n,m}\to\calX^{k}$ such that 
    \begin{equation*}
        \forall (x,y)\in\calD_{\Delta t}^{n,m},\qquad \Phi_{\Delta t}(x,y+\Lambda_{\Delta t}^{n,m}(x,y))=0
    \end{equation*}
    and
    \begin{equation*}
        \forall (x,y)\in\calD_{\Delta t}^{n,m},\qquad\nabla_{y}\Phi_{\Delta t}(x,y+\Lambda_{\Delta t}^{n,m}(x,y))\text{ is invertible}.
    \end{equation*}
    Moreover, the set~$\calD_{\Delta t}^{n,m}$ can be chosen such that there exists~$\widetilde{\alpha}(n,m)>0$ for which,
    \begin{align}
        \forall(x,y)\in\calD_{\Delta t}^{n,m},&\,\left\lVert\Lambda_{\Delta t}^{n,m}(x,y)\right\rVert<\widetilde{\alpha}(n,m)\text{ and}\nonumber\\
        &\forall\lambda\in\calX^{k}\setminus\left\lbrace\Lambda_{\Delta t}^{n,m}(x,y)\right\rbrace,\,\Phi_{\Delta t}(x,y+\lambda)=0\Longrightarrow\left\lVert\lambda\right\rVert\geqslant\widetilde{\alpha}(n,m).\label{eq:smallest_projection_zeta}
    \end{align}
\end{lemma}

\begin{proof}
    Fix~$n,m\geqslant1$ and a couple~$(x_0,y_0)\in\calG_{\Delta t}^{n,m}$. Let~$\calU_0\subset\calX\times\calX^k\times\calX^k$ be an open neighborhood of~$(x_0,y_0,0)$ such that
    \begin{equation*}
        \forall (x,y,\lambda)\in\calU_0,\qquad y+\lambda\in\calX^k.
    \end{equation*}
    This neighborhood exists since~$\calX$ is an open set. Let us consider the function
    \begin{equation*}
        F\colon
        \left\lbrace
            \begin{aligned}
                \calU_0&\to\calX^{k}\\
                (x,y,\lambda)&\mapsto\Phi_{\Delta t}(x,y+\lambda).
            \end{aligned}
        \right.
    \end{equation*}
    This is a~$\calC^{1}$ function such that~$F(x_0,y_0,0)=0$ and the matrix
    \begin{equation*}
        \nabla_{\lambda}F(x_0,y_0,0)=\nabla_{y}\Phi_{\Delta t}(x_0,y_0)
    \end{equation*}
    is invertible. Therefore, by the implicit function theorem, there exist~$\eta(x_0,y_0)>0$,~$\alpha(x_0,y_0)>0$ and a~$\calC^{1}$ function~$\Lambda_{\Delta t}^{n,m,0}\colon\calB(x_0,\eta(x_0,y_0))\times\calB(y_0,\eta(x_0,y_0))\to\calB(0,\alpha(x_0,y_0))$ such that~$\calB(x_0,\eta(x_0,y_0))\times\calB(y_0,\eta(x_0,y_0))\times\calB(0,\alpha(x_0,y_0))\subset\calU_0$ and
    \begin{multline}
        \label{eq:ift_uniqueness}
        \forall (x,y,\lambda)\in\calB(x_0,\eta(x_0,y_0))\times\calB(y_0,\eta(x_0,y_0))\times\calB(0,\alpha(x_0,y_0)),\\
        \Phi_{\Delta t}(x,y+\lambda)=0\iff\lambda=\Lambda_{\Delta t}^{n,m,0}(x,y).
    \end{multline}
    By a continuity argument, upon reducing~$\eta(x_0,y_0)$, one can assume that the matrix~$\nabla_{y}\Phi_{\Delta t}(x,y+\Lambda_{\Delta t}^{n,m,0}(x,y))$ is invertible for all~$(x,y)\in\calB(x_0,\eta(x_0,y_0))\times\calB(y_0,\eta(x_0,y_0))$, since this is true for~$(x,y)=(x_0,y_0)$. Note that, by the definition of~$F$, for any~$(y,\lambda)\in\calB(y_0,\eta(x_0,y_0))\times\calB(0,\alpha(x_0,y_0))$, it holds~$y+\lambda\in\calX^k$ so that~$\Phi_{\Delta t}(x,y+\Lambda_{\Delta t}^{n,m,0}(x,y))$ is well-defined whenever~$(x,y)\in\calB(x_0,\eta(x_0,y_0))\times\calB(y_0,\eta(x_0,y_0))$.

    Remark now that~$\bigcup_{(x_0,y_0)\in\calG_{\Delta t}^{n,m}}\calB(x_0,\eta(x_0,y_0))\times\calB(y_0,\eta(x_0,y_0))$ is an open over of the compact set~$\calG_{\Delta t}^{n,m}$, from which it is possible to extract a finite cover:
    \begin{equation*}
        \calG_{\Delta t}^{n,m}\subset\bigcup_{i=1}^{\ell^{n,m}}\calB(x_i,\eta_0(x_i,y_i))\times\calB(y_i,\eta_0(x_i,y_i))=:\calD_{\Delta t}^{n,m},
    \end{equation*}
    where~$(x_i,y_i)_{1\leqslant i\leqslant \ell^{n,m}}$ are elements of~$\calG_{\Delta t}^{n,m}$. Note that the set~$\calD_{\Delta t}^{n,m}$ is an open set contained in~$\calX\times\calX^{k}$. Let us define the function~$\Lambda_{\Delta t}^{n,m}\colon\calD_{\Delta t}^{n,m}\to\calX^{k}$ as follows: for~$(x,y)\in\calD_{\Delta t}^{n,m}$ and~$1\leqslant i\leqslant \ell^{n,m}$ for which~$(x,y)\in\calB(x_i,\eta(x_i,y_i))\times\calB(y_i,\eta(x_i,y_i))$,
    \begin{equation*}
        \Lambda_{\Delta t}^{n,m}(x,y)=\Lambda_{\Delta t}^{n,m,i}(x,y).
    \end{equation*}
    Let us show that this function is well-defined. If~$(x,y)\in\calB(x_i,\eta(x_i,y_i))\times\calB(y_i,\eta(x_i,y_i))\cap\calB(x_j,\eta(x_j,y_j))\times\calB(y_j,\eta(x_j,y_j))$ with~$i\neq j$, we can assume without loss of generality that~$\alpha(x_i,y_i)\leqslant\alpha(x_j,y_j)$. Then 
    \begin{equation*}
        \Lambda_{\Delta t}^{n,m,i}(x,y)\in\calB(0,\alpha(x_i,y_i))\subset\calB(0,\alpha(x_j,y_j))
    \end{equation*}
    with~$\Phi_{\Delta t}\left(x,y+\Lambda_{\Delta t}^{n,m,i}(x,y)\right)=0$. By~\eqref{eq:ift_uniqueness}, this shows that~$\Lambda_{\Delta t}^{n,m,i}(x,y)=\Lambda_{\Delta t}^{n,m,j}(x,y)$. 

    Let us now define 
    \begin{equation}
        \label{eq:tilde_alpha}
        \widetilde{\alpha}(n,m)=\min\limits_{1\leqslant i\leqslant \ell^{n,m}}\alpha(x_i,y_i)>0.
    \end{equation}
    Upon reducing~$\calD_{\Delta t}^{n,m}$ to a smaller open set still containing~$\calG_{\Delta t}^{n,m}$, one can assume that for all~$(x,y)\in\calG_{\Delta t}^{n,m}$, it holds~$\Lambda_{\Delta t}^{n,m}(x,y)\in\calB(0,\widetilde{\alpha}(n,m))$ since~$\Lambda_{\Delta t}^{n,m}$ is~$\calC^{1}$ and~$\Lambda_{\Delta t}^{n,m}=0$ on~$\calG_{\Delta t}^{n,m}$. The property~\eqref{eq:smallest_projection_zeta} is then a consequence of~\eqref{eq:ift_uniqueness}.
\end{proof}

We are now in position to prove Proposition~\ref{prop:Dimp}.
\begin{proof}[Proof of Proposition~\ref{prop:Dimp}]
    Let us define the set
    \begin{equation*}
        \calD_{\Delta t}^{\imp}=\bigcup_{n,m\geqslant1}\calD_{\Delta t}^{n,m}.
    \end{equation*}
    The set~$\calD_{\Delta t}^{\imp}$ is open and contained in~$\calX\times\calX^k$ as it is the union of open sets contained in~$\calX\times\calX^k$. By construction, the set~$\calD_{\Delta t}^{\imp}$ contains~$\calG_{\Delta t}$ since~$\calG_{\Delta t}=\bigcup_{n,m\geqslant1}\calG_{\Delta t}^{n,m}$.

    Let us define the function~$\Lambda_{\Delta t}^{\imp}\colon\calD_{\Delta t}^{\imp}\to\calX^{k}$ as follows: for~$(x_0,y_0)\in\calD_{\Delta t}^{\imp}$, there exist~$n_0,m_0\geqslant1$ such that~$(x_0,y_0)\in\calD_{\Delta t}^{n_0,m_0}$; then,
    \begin{equation*}
        \Lambda_{\Delta t}^{\imp}(x,y)=\Lambda_{\Delta t}^{n_0,m_0}(x,y).
    \end{equation*}
    The function~$\Lambda_{\Delta t}^{\imp}$ is well-defined. Indeed, if~$n,m,n',m'\geqslant1$ are such that~$(x,y)\in\calD_{\Delta t}^{n,m}\cap\calD_{\Delta t}^{n',m'}$, we can assume without loss of generality that~$\widetilde{\alpha}(n,m)\leqslant\widetilde{\alpha}(n',m')$ where~$\widetilde{\alpha}(n,m)$ is defined by~\eqref{eq:tilde_alpha}. Then,
    \begin{equation*}
        \Lambda_{\Delta t}^{n,m}(x,y)\in\calB(0,\widetilde{\alpha}(n,m))\subset\calB(0,\widetilde{\alpha}(n',m')),
    \end{equation*}
    with~$\Phi_{\Delta t}(x,y+\Lambda_{\Delta t}^{n,m}(x,y))=0$, so that~\eqref{eq:ift_uniqueness} implies that~$\Lambda_{\Delta t}^{n,m}(x,y)=\Lambda_{\Delta t}^{n',m'}(x,y)$. The last property of Proposition~\ref{prop:Dimp} follows from~\eqref{eq:smallest_projection_zeta} by considering~$\calV_0=\calD_{\Delta t}^{n_0,m_0}$ and~$\alpha_0=\widetilde{\alpha}(n_0,m_0)$ when~$(x,y)\in\calD_{\Delta t}^{n_0,m_0}$. This concludes the proof.
\end{proof}

%---------------------------------------------------------
%---------------------------------------------------------
%---------------------------------------------------------
%---------------------------------------------------------
%---------------------------------------------------------
%---------------- ACKNOWLEDGEMENTS -----------------------
%---------------------------------------------------------
%---------------------------------------------------------
%---------------------------------------------------------
%---------------------------------------------------------
%---------------------------------------------------------
\paragraph{Acknowledgements.}
The works of T.L., R.S. and G.S. benefit from fundings from the European Research Council (ERC) under the European Union's Horizon 2020 research and innovation program (project EMC2, grant agreement No 810367), and from the Agence Nationale de la Recherche through the grants ANR-19-CE40-0010-01 (QuAMProcs) and ANR-21-CE40-0006 (SINEQ). This project was initiated as T.L. was a visiting professor at Imperial College of London (ICL), with a visiting professorship grant from the Leverhulme Trust. The Department of Mathematics at ICL and the Leverhulme Trust are warmly thanked for their support. 

%---------------------------------------------------------
%---------------------------------------------------------
%---------------------------------------------------------
%---------------------------------------------------------
%---------------------------------------------------------
%---------------- BIBLIOGRAPHY ---------------------------
%---------------------------------------------------------
%---------------------------------------------------------
%---------------------------------------------------------
%---------------------------------------------------------
%---------------------------------------------------------
\bibliographystyle{plain}
\bibliography{biblio.bib}

\end{document}